\pdfoutput=1
\documentclass[12pt]{amsart}
\usepackage{lmodern}
\usepackage{ifthen}
\usepackage{amsfonts}
\usepackage{amsxtra}
\usepackage{amssymb}
\usepackage{mathdots}
\usepackage{array}
\usepackage[margin=0.8in]{geometry}
\usepackage{xcolor}
\definecolor{cite}{rgb}{0.30,0.60,1.00}
\definecolor{url}{rgb}{0.00,0.00,0.80}
\definecolor{link}{rgb}{0.40,0.10,0.20}
\usepackage[colorlinks,linkcolor=link,urlcolor=url,citecolor=cite,pagebackref,breaklinks]{hyperref}
\usepackage{bbm}
\usepackage{mathtools}
\usepackage{mathrsfs}
\usepackage{appendix}
\usepackage[all]{xy}
\usepackage[lite,abbrev,msc-links,alphabetic]{amsrefs}

\usepackage{graphicx}
\usepackage{multirow}
\usepackage{pstricks}
\usepackage{pst-pdf}
\usepackage{enumitem}
\usepackage{pifont}
\usepackage{marvosym}
\usepackage{txfonts}

\usepackage[OT2,T1]{fontenc}
\DeclareSymbolFont{cyrletters}{OT2}{wncyr}{m}{n}
\DeclareMathSymbol{\Sha}{\mathalpha}{cyrletters}{"58}

\usepackage{graphicx}

\makeatletter
\providecommand*{\Dashv}{%
  \mathrel{%
    \mathpalette\@Dashv\vDash
  }%
}
\newcommand*{\@Dashv}[2]{%
  \reflectbox{$\m@th#1#2$}%
}
\makeatother

\setitemize[0]{leftmargin=20pt}
\setenumerate[0]{leftmargin=30pt}

\numberwithin{equation}{section}

\theoremstyle{plain}
\newtheorem{proposition}{Proposition}[section]

\newtheorem{corollary}[proposition]{Corollary}
\newtheorem{lem}[proposition]{Lemma}
\newtheorem{theorem}[proposition]{Theorem}

\theoremstyle{definition}
\newtheorem{definition}[proposition]{Definition}

\newtheorem{notation}[proposition]{Notation}

\theoremstyle{remark}
\newtheorem{remark}[proposition]{Remark}
\newtheorem{example}[proposition]{Example}

\renewcommand{\b}[1]{\mathbf{#1}}
\renewcommand{\c}[1]{\mathcal{#1}}
\renewcommand{\d}[1]{\mathbb{#1}}
\newcommand{\f}[1]{\mathfrak{#1}}
\renewcommand{\r}[1]{\mathrm{#1}}
\newcommand{\s}[1]{\mathscr{#1}}
\renewcommand{\sf}[1]{\mathsf{#1}}
\renewcommand{\(}{\left(}
\renewcommand{\)}{\right)}
\newcommand{\res}{\mathbin{|}}

\newcommand{\ol}[1]{\overline{#1}{}}

\newcommand{\ul}{\underline}
\renewcommand{\leq}{\leqslant}
\renewcommand{\geq}{\geqslant}

\newcommand{\bA}{\b A}

\newcommand{\cV}{\c V}
\newcommand{\cW}{\c W}

\newcommand{\dC}{\d C}

\newcommand{\dH}{\d H}

\newcommand{\dL}{\d L}

\newcommand{\dN}{\d N}

\newcommand{\dQ}{\d Q}
\newcommand{\dR}{\d R}

\newcommand{\dZ}{\d Z}
\newcommand{\fA}{\f A}

\newcommand{\fD}{\f D}

\newcommand{\fF}{\f F}

\newcommand{\fS}{\f S}
\newcommand{\fT}{\f T}
\newcommand{\fU}{\f U}
\newcommand{\fV}{\f V}
\newcommand{\fW}{\f W}

\newcommand{\fg}{\f g}
\newcommand{\fh}{\f h}

\newcommand{\fl}{\f l}

\newcommand{\fp}{\f p}

\newcommand{\rC}{\r C}
\newcommand{\rD}{\r D}
\newcommand{\rE}{\r E}
\newcommand{\rF}{\r F}
\newcommand{\rG}{\r G}
\newcommand{\rH}{\r H}
\newcommand{\rI}{\r I}

\newcommand{\rL}{\r L}
\newcommand{\rM}{\r M}

\newcommand{\rO}{\r O}

\newcommand{\rR}{\r R}

\newcommand{\rU}{\r U}

\newcommand{\rZ}{\r Z}

\newcommand{\rd}{\,\r d}
\newcommand{\re}{\r e}

\newcommand{\rh}{\r h}

\newcommand{\sA}{\s A}

\newcommand{\sD}{\s D}

\newcommand{\sF}{\s F}
\newcommand{\sG}{\s G}

\newcommand{\sO}{\s O}

\newcommand{\sS}{\s S}

\newcommand{\sV}{\s V}
\newcommand{\sW}{\s W}

\newcommand{\sfM}{\sf M}

\newcommand{\sft}{\sf t}

\newcommand{\sfv}{\sf v}

\newcommand{\sfy}{\sf y}
\newcommand{\sfz}{\sf z}

\newcommand{\tB}{\mathtt{B}}

\newcommand{\tF}{\mathtt{F}}

\newcommand{\tS}{\mathtt{S}}
\newcommand{\tT}{\mathtt{T}}

\newcommand{\tc}{\mathtt{c}}

\newcommand{\te}{\mathtt{e}}

\newcommand{\ti}{\mathtt{i}}

\newcommand{\tq}{\mathtt{q}}

\newcommand{\tw}{\mathtt{w}}

\newcommand{\balpha}{\boldsymbol{\alpha}}
\newcommand{\bbeta}{\boldsymbol{\beta}}

\newcommand{\pres}[2]{\prescript{#1}{}{#2}}

\newcommand{\hs}{\r{hs}}
\newcommand{\Herm}{\r{Herm}}

\newcommand{\RE}{\r{Re}\,}

\DeclareMathOperator{\Aut}{Aut}
\DeclareMathOperator{\BC}{BC}
\DeclareMathOperator{\CH}{CH}

\DeclareMathOperator{\Div}{Div}

\DeclareMathOperator{\End}{End}

\DeclareMathOperator{\GL}{GL}

\DeclareMathOperator{\Hom}{Hom}

\DeclareMathOperator{\Ker}{ker}

\DeclareMathOperator{\Mp}{Mp}
\DeclareMathOperator{\Mat}{Mat}

\DeclareMathOperator{\rank}{rk}
\DeclareMathOperator{\Res}{Res}

\DeclareMathOperator{\SO}{SO}
\DeclareMathOperator{\Sp}{Sp}

\DeclareMathOperator{\Sym}{Sym}

\DeclareMathOperator{\tr}{tr}
\DeclareMathOperator{\Tr}{Tr}

\DeclareMathOperator{\vol}{vol}

\begin{document}

\title[Archimedean arithmetic Siegel--Weil formula]
{Archimedean arithmetic Siegel--Weil formula for general weights over Shimura curves}

\author{Yifeng Liu}
\address{Institute for Advanced Study in Mathematics, Zhejiang University, Hangzhou 310058, China}
\email{liuyf0719@zju.edu.cn}

\date{\today}

\begin{abstract}
  We prove an averaging formula for the canonical archimedean height pairing of special divisors with weights over orthogonal and unitary Shimura curves in terms of derivatives of Whittaker functions.
\end{abstract}

\maketitle

\tableofcontents

\section{Introduction}

Height is one of the most important concepts in number theory, with earliest prototype dating back to the method of \emph{infinite descent} by Fermat. A century ago, Mordell systemically used the notion of height to prove the famous Mordell theorem on the finite generation of rational points on elliptic curves. The next highlight moment of height was Faltings' proof of the Mordell conjecture and the Gross--Zagier formula toward the BSD conjecture, both in the 1980s. The current article is a small step toward a direction which we would describe as \emph{height with coefficients}.

Recall that the Gross--Zagier formula \cite{GZ86} computes the central derivative of the Rankin--Selberg $L$-function of a cusp newform of weight $2$ and trivial Nebentypus and a ring class group character of an imaginary quadratic field, jointly under the so-called Heegner condition, through the height of the so-called Heegner point. Under the philosophy of this formula, one expects an arithmetic interpretation of the central derivative of $L(s,f,\chi)$, where $f$ is a Hilbert cusp newform over a totally real field $F$ and $\chi$ is a Hecke character of a CM extension $E/F$, jointly satisfying
\begin{itemize}
  \item \emph{self-duality condition}: the Nebentypus of $f$ is inverse to the restriction of $\chi$ to $F$ (which would imply that $L(s,f,\chi)$ is self-dual);

  \item \emph{odd-parity condition}: the root number of $L(s,f,\chi)$ is $-1$ (so that the central vanishing order is odd);

  \item \emph{weight-balancing condition}: the weight of $\chi$ at every complex embedding of $E$ is at most the weight of $f$ at the underlying real place of $F$ minus $2$.
\end{itemize}

When $f$ has parallel weight $2$ (so that $\chi$ must be a Dirichlet character), such expectation has been achieved by Yuan--Zhang--Zhang \cite{YZZ13} unconditionally. In higher weight case, S.~Zhang \cite{Zha97} obtained a formula, still under the Heegner condition (so that $F=\dQ$), for $f$ a cusp newform of arbitrary even weight and trivial Nebentypus and $\chi$ a ring class group character, through the height of the so-called Heegner cycle. However, unlike the weight $2$ case, Zhang's formula crucially replies on the Heegner condition as the construction involves the Kuga--Sato variety, which does not exist for other quaternionic Shimura curves.\footnote{For quaternionic Shimura curves over $\dQ$, one does have an analogue of the Kuga--Sato variety; but it remains the challenge to construct the Heegner cycle.} For this reason, there is almost no further progress in the case of general weights since Zhang's work, except for a recent preprint \cite{LS} where the authors allow the character to have nonzero weights (but still under the Heegner condition).

In the ongoing work \cite{Liu26}, we will provide a solution to this problem. We propose a (conjectural) formula computing the central $L$-derivative of an arbitrary cohomological isobaric automorphic representation $\Pi$ of $\GL_{2r,E}$ that is conjugate self-dual with respect to a CM extension $E/F$, in terms of arithmetic theta lifting, as long as the weight of $\Pi$ satisfies a \emph{balancing condition}.\footnote{Informally speaking, this means that $0$ can be inserted in the middle place of the weight with respect to any complex embedding of $E$.} Note that when $\Pi$ is of minimal weight (that is, the zero weight), this has already been carried out in \cites{Liu11,Liu12} for $r=1$ unconditionally, and in \cites{LL21,LL22} for general $r$ under certain conditions. For a general weight $\lambda$, we perform the following acts in \cite{Liu26}:
\begin{itemize}
  \item construct a relative motive $\rM_\lambda$ over a certain unitary Shimura variety $X$ over $E$ whose complex uniformization is an open ball of dimension $2r-1$;\footnote{Currently, we have not found the parallel construction for orthogonal Shimura varieties.}

  \item when $\lambda$ is balanced, define generating series with values in the motivic cohomology $\rH^{2r}(X,\rM_\lambda(r))$ (which we call \emph{motivic generating series}), generalizing the famous work of Kudla--Millson \cite{KM90} where $\rM_\lambda$ is the unit motive so that $\rH^{2r}(X,\rM_\lambda(r))=\CH^r(X)$;

  \item define the \emph{motivic arithmetic theta lifting} of $\Pi$ as above assuming the modularity of the motivic generating series;

  \item define a \emph{global height pairing} for $\rH^{2r}(X,\rM_\lambda(r))$ and propose a conjectural \emph{motivic arithmetic inner product formula} relating $L'(1/2,\Pi)$ and the height of the motivic arithmetic theta lifting of $\Pi$;

  \item prove the motivic arithmetic inner product formula when $r=1$.
\end{itemize}
Now going back to the $L$-function $L(s,f,\chi)$ we have discussed (with $f,\chi$ satisfying the three conditions), the last act indeed provides a formula for its central derivative by taking $\Pi$ to be $\BC(\pi_f)\otimes\chi$, where $\pi_f$ is the automorphic representation of $\GL_{2,F}$ generated by $f$ and $\BC$ denotes the automorphic quadratic base change from $F$ to $E$.\footnote{Note that $\Pi$ is not necessarily cuspidal; but it is cohomological, self-dual, and of balanced weight.} It is worth noting that we do not even require the weights of $f$ to have the same parity at different real places -- the self-duality condition only implies that at a given archimedean place, the weight of $f$ and that of $\chi$ have the same parity.

\begin{remark}
When the weights of $f$ do not have the same parity at different real places, it is necessary to pass to the unitary group since $\pi_f$ alone is \emph{not} cohomological as an automorphic representation of $\GL_{2,F}$; but $\Pi$ is still a cohomological automorphic representation of $\GL_{2,E}$.
\end{remark}

The global height on $\rH^{2r}(X,\rM_\lambda(r))$ can also be decomposed into a sum of local heights over all places of $E$ for certain representatives of classes, similar to the case of constant coefficients. Gladly, when $r=1$ (so $X$ is a curve), such local heights coincide with the one defined and studied by Brylinski in \cite{Bry89}. In some sense, our work \cite{Liu26} fulfills an expectation of Brylinski,\footnote{See the last sentence of \cite{Bry89}*{\S0}.} at least for unitary Shimura curves, that there should be a global theory overseeing these local heights.

The primary purpose of this article is to compute the archimedean local height of the motivic generating series when $r=1$, that is, over unitary Shimura curves. In fact, we also perform the parallel computation over orthogonal Shimura curves.\footnote{Although we do not know yet how to define the motivic generating series in the orthogonal case, its ``archimedean realization'' can be defined.} More precisely, here is what we will do in this article:
\begin{itemize}
  \item In \S\ref{ss:2}, we review the Funke--Millson theory \cite{FM06} -- the generalization of the Kudla--Millson theory \cites{KM86,KM90} to general weights -- for a quadratic space of signature $(m-2,2),(m,0),(m,0),\dots$, including Funke--Millson forms over the corresponding orthogonal Shimura variety, their generating series, together with main properties.

  \item In \S\ref{ss:3}, we develop the parallel theory for a hermitian space of signature $(p,q),(m,0),(m,0),\dots$ ($p+q=m$), following the same line in \cite{FM06}.

  \item In \S\ref{ss:4}, we review the notion of relative Hodge structure (a.k.a. variation of Hodge structures) over complex manifolds, especially the one with complex coefficients, which is not so well recorded in literature.

  \item In \S\ref{ss:5}, we review the notion of height pairing of divisors with coefficients in a relative Hodge structure, both real and complex, following the work \cite{Bry89}.

  \item In \S\ref{ss:6} and \S\ref{ss:7}, we define generating series of special divisors with coefficients in a relative Hodge structure associated with a balanced weight in the orthogonal case with $m=3$; we construct the associated Green functions and compute their height pairing; finally, we deduce an averaging formula for the \emph{canonical} archimedean height pairing for the special divisors.

  \item In \S\ref{ss:8} and \S\ref{ss:9}, we define generating series of special divisors with coefficients in a relative Hodge structure associated with a balanced weight in the unitary case with $m=2$ and $p=q=1$; we construct the associated Green functions and compute their height pairing; finally, we deduce an averaging formula for the \emph{canonical} archimedean height pairing for the special divisors.
\end{itemize}

To end the introduction, we would like to present a theorem that will actually be used in \cite{Liu26}. Denote by $G_1$ the isometry group of the quasi-split skew-hermitian form of rank $2$ over $E/F$ and write $\dH_1$ for its hermitian symmetric domain (of complex dimension $[F:\dQ]$).

Fix a complex embedding $u_0\colon E\to\dC$ with the underlying real embedding $v_0\colon F\to\dR$. Let $V$ be a hermitian space over $E$ of signature $(1,1)$ at $v_0$ and $(2,0)$ at others. Put $H\coloneqq\rU(V)$ and take a balanced weight $\lambda$ for $H$, which, in this case, is just a tuple $b=(b_u)_u\in\dN^{\Hom(E,\dC)}$. Choose a sufficiently small open compact subgroup $L$ of $H(\bA_F^\infty)$ ($\bA$ for ad\`{e}les) and denote by $X_L$ the corresponding unitary Shimura curve over $\dC$ (via the embedding $u_0$). From $\lambda$, we have a complex relative Hodge structure $\cV_\lambda$ over $X_L$ (Definition \ref{de:hodge_unitary}) whose underlying (complex) local system we denote by $V_\lambda$. For every Schwartz function $\phi^\infty\in\sS(V\otimes_F\bA_F^\infty)^L$ and every $t\in F^\times$, we define in \S\ref{ss:8} a weighted special divisor with coefficients in $\cV_\lambda$:
\[
\rD^\lambda_t(\phi^\infty)(\tau)\in\Div(X_L,\cV_\lambda)\coloneqq\bigoplus_{x\in X_L} (V_\lambda)_x\cap(\cV_\lambda^{(0,0)})_x,
\]
which is a function on $\tau\in\dH_1$. Now assume that $V$ is anisotropic so that $X_L$ is proper, and that $\lambda$ is nontrivial. In this case, $\rD^\lambda_t(\phi^\infty)(\tau)$ is automatically cohomologically trivial.

\begin{theorem}[Theorem \ref{th:canonical_unitary}]
Let $K$ be an open compact subgroup of $G_1(\bA_F^\infty)$ and choose a fundamental domain $\dH_1^K$ of $\dH_1$ for the congruent subgroup $G_1(F)\cap K$. Consider an element $t_2\in F$ and a pair $(\phi^\infty_1,\phi^\infty_2)\in(\sS(V\otimes_F\bA_F^\infty)^L)^2$ such that $\phi^\infty_1\otimes\ol{\phi^\infty_2}$ is regularly supported (Definition \ref{de:regular_unitary}). For every (holomorphic) hermitian Hilbert cusp form $f$ on $\dH_1$ of (classical) weights $(2b_u+1,-2b_{\ol{u}}-1)_u$ and level $G_1(F)\cap K$,
\begin{align*}
&\vol(L,\rd^\natural h)\int_{\dH_1^K}\ol{f(\tau_1)}\sum_{t_1\in F}
\left\langle\rD^\lambda_{t_1}(\phi^\infty_1)(\tau_1),\rD^\lambda_{t_2}(\phi^\infty_2)(\tau_2)\right\rangle_{X_L}\rd\tau_1 \\
&\qquad\quad=\frac{1}{2}\int_{\dH_1^K}\ol{f(\tau_1)}\sum_{t_1\in F}\sum_{T\in\fT(t_1,t_2)_{v_0}}
W'_T(0,g_{\tau_1,\tau_2},\Phi_\lambda\otimes(\phi^\infty_1\otimes\ol{\phi^\infty_2}))
\rd\tau_1
\end{align*}
holds in which both sides are absolutely convergent, where
\begin{itemize}
  \item $\vol(L,\rd^\natural h)$ is the volume of $L$ under the Siegel--Weil measure (Definition \ref{de:measure_unitary});

  \item $\langle\;,\;\rangle_{X_L}$ denotes the canonical (hermitian) archimedean height pairing between cohomologically trivial divisors with coefficients (Definition \ref{de:canonical_complex});

  \item $\rd\tau_1$ denotes the measure on $\dH_1$ that induces the Petersson inner product;

  \item $\fT(t_1,t_2)_{v_0}$ is a certain set of moment matrices (Notation \ref{no:moment_unitary}); and

  \item $W_T(s,g,\Phi)$ is the $T$-th Whittaker function, whose precise definition we leave to \S\ref{ss:8}.
\end{itemize}
\end{theorem}

The proof of the above theorem replies on a non-harmonic Green function $\rG^\lambda_t(\phi^\infty)$ for $\rD^\lambda_t(\phi^\infty)$, which is the analogue of the Kudla--Millson Green function in the case of trivial weights. To compute the height pairing between $\rG^\lambda_{t_1}(\phi^\infty_1)(\tau_1)$ and $\rG^\lambda_{t_2}(\phi^\infty_2)(\tau_2)$, we follow the strategy in \cite{GS19}. There is a key difference between the case of general weights and the case of trivial weights, namely, the $\rU(2)$-invariance property \cite{Liu11}*{Proposition~4.10} does not hold in general. As a result, the approach via direct computation in \cite{Liu11} does not work anymore; however, the conceptual approach of \cite{GS19} still works, although we do not find an interpretation of $\rG^\lambda_t(\phi^\infty)$ via superconnections.

\subsubsection*{Notation}

\begin{itemize}
  \item The imaginary unit is denoted by $\ti$; and the natural constant is denoted by $\te$.

  \item For every $t\in\dR_{>0}$ and $s\in\dC$, $t^s$ means $\te^{(\ln t)\cdot s}$. For $z\in\dC^\times$, $|z|\coloneqq(z\ol{z})^{1/2}$.

  \item For $b\in\dN=\{0,1,2,\dots\}$, denote by $\fS_b$ the group of permutations on $\{1,\dots,b\}$. For a finite tuple $b=(b_v)_v$ of elements in $\dN$, put $\fS_b\coloneqq\prod_{v}\fS_{b_v}$.

  \item All quadratic spaces (over a ring $F$) or hermitian spaces (over an \'{e}tale quadratic extension $E/F$ of rings) are assumed to be of finite rank and nondegenerate. For such a space $V$ with the form $(\;,\;)_V$ when $2\in F^\times$, we have the moment matrix
      \[
      T(x)=\frac{1}{2}\((x_i,x_j)_V\)_{1\leq i,j\leq r}
      \]
      for a tuple $x=(x_1,\dots,x_r)\in V^r$; and for an $r$-by-$r$ matrix $T$ with entries in $F$ or $E$, put
      \[
      V^r_T\coloneqq\left\{x\in V^r\res T(x)=T\right\}.
      \]

  \item Denote by $\psi\colon\dQ\backslash\bA_\dQ\to\dC^\times$ the standard additive automorphic character, that is, the unique character such that $\psi_{\dQ,\infty}(x)=\re^{2\pi ix}$. For every number field $F$, put $\psi_F\coloneqq\psi\circ\Tr_{F/\dQ}\colon F\backslash\bA_F\to\dC^\times$.
\end{itemize}

Let $m\geq 0$ be an integer.
\begin{itemize}
  \item Denote by $\fW_m$ the set of tuple $\lambda=(b_1,b_2,\dots,b_m)$ of decreasing integers. For $\lambda\in\fW_m$, put $b(\lambda)\coloneqq\sum_{i=1}^m b_i$. We denote the zero tuple by $\b{0}$ (when the length $m$ is clear).

  \item We have an involution map on $\fW_m$ sending $\lambda$ to $\ol\lambda\coloneqq(-b_m,\dots,-b_2,-b_1)$. Denote by $\fW'_m$ the subset of $\fW_m$ consisting of $\lambda$ satisfying $\ol\lambda=\lambda$.

  \item For $\lambda\in\fW_m$, denote by $0\leq s=s(\lambda)\leq m$ the least integer such that $b_s>0$ ($s=0$ if $b_1\leq 0$).

  \item Denote by $\fW_m^+$ the subset of $\fW_m$ consisting of $\lambda$ such that $b_m\geq 0$. We will adopt the convention that an element $\lambda\in\fW_m^+$ is naturally an element in $\fW_r^+$ for every integer $r\geq s(\lambda)$ by adding or deleting zeros.

  \item For $\lambda\in\fW_m$, put $\lambda^+\coloneqq(b_1,\dots,b_s,0,\dots,0)\in\fW_m^+$.

  \item For $\lambda\in\fW_m^+$, denote by $\fF(\lambda,m)$ the set of semistandard fillings of the Young diagram of $\lambda$ by $\{1,\dots,m\}$ \cite{FM06}*{Definition~3.6}.
\end{itemize}

\subsubsection*{Acknowledgements}

The author is partially supported by the National Key R\&D Program of China No.~2022YFA1005300.

\section{Funke--Millson theory for orthogonal groups}
\label{ss:2}

In this section, we review the work of Funke--Millson \cite{FM06}. Let $F$ be a ring, $r$ and $m$ positive integers.

\begin{notation}\label{no:lambda_orthogonal}
Consider a free $F$-module $U$ of rank $r$ and an element $\lambda\in\fW_r^+$ such that $b!$ is invertible in $F$, where $b\coloneqq b(\lambda)$. We have an idempotent $\pi_\lambda\in\dZ[1/b!][\fS_b]$ (with respect to the standard filling) hence a direct summand
\[
\tS_\lambda(U)\coloneqq\pi_\lambda\(\tT^b(U)\)
\]
of $\tT^b(U)$ of $F[\GL_F(U)]$-modules. See \cite{FM06}*{\S3.1.1} for more details.

In general, for a finite set $\fV$ and a tuple $\lambda=(\lambda_v)_v\in(\fW_r^+)^\fV$ such that $b_v!$ ($b_v\coloneqq b(\lambda_v)$) is invertible in $F$ for every $v\in\fV$, we have
\[
\tT^b(U)\coloneqq\bigotimes_{v\in\fV}\tT^{b_v}(U)
\]
(tensor over $F$) and similarly $\tS_\lambda(U)$, together with the projector
\[
\pi_\lambda\coloneqq\bigotimes_{v\in\fV}\pi_{\lambda_v}.
\]
\end{notation}

\begin{definition}\label{de:schur_orthogonal}
Consider a (nondegenerate) quadratic space $V$ over $F$ of rank $m$ and an element $\lambda\in\fW'_m$ such that $m\cdot b!$ is invertible in $F$, where $b\coloneqq b(\lambda^+)$. For every pair $I=(i,j)$ with $1\leq i<j\leq b$, we have the contraction operator
\[
C_I\colon\tT^b(V)\to\tT^{b-2}(V)
\]
by ``pairing'' the $i$-th and $j$-th factors of $\tT^b(V)$, and the expansion operator
\[
A_I\colon\tT^{b-2}(V)\to\tT^b(V)
\]
by inserting the map $F\to V\otimes_FV^\vee\to V\otimes_F V$ to the $(i,j)$-th spot of $\tT^b(V)$, where the first arrow is the coevaluation map and the second one is the isomorphism induced from the nondegenerate pairing. If we put
\[
\tT_{[b]}(V)\coloneqq\bigcap_{I}\Ker A_I,
\]
then
\begin{align}\label{eq:harmonic_decomposition_orthogonal}
\tT^b(V)=\tT_{[b]}(V)\oplus\(\bigoplus_{r=1}^{\lfloor b/2\rfloor}\tT_{[b]}(V)_{b-2r}\),
\end{align}
where
\[
\tT_{[b]}(V)_{b-2r}\coloneqq\sum A_{I_1}\circ\cdots\circ A_{I_r}\(\tT^{[b-2r]}(V)\).
\]
It is clear that the decomposition \eqref{eq:harmonic_decomposition_orthogonal} is stable under the action of both $\rO(V)$ and $\fS_b$. In particular, we may define
\[
\tS_{[\lambda]}(V)\coloneqq\pi_{\lambda^+}\(\tT_{[b]}(V)\),
\]
which is a direct summand of the $F[\rO(V)]$-module $\tT_{[b]}(V)$ hence $\tT^b(V)$.
\end{definition}

\begin{example}\label{ex:ratio_orthogonal}
Suppose that $m=3$ and $\lambda=(b,0,-b)$ for some $b\in\dN$. There exists a unique homogenous polynomial $Q_\lambda\in\dZ[(3\cdot b!)^{-1}][T_{11},T_{12},T_{21},T_{22}]$ of degree $b$, symmetric in $T_{12}$ and $T_{21}$, such that for every three-dimensional quadratic space $V$ over a $\dZ[(3\cdot b!)^{-1}]$-ring $F$, one has
\[
([x_1^{\otimes b}],[x_2^{\otimes b}])_{\tS_{[\lambda]}(V)}=Q_\lambda((x_1,x_1)_V,(x_1,x_2)_V,(x_2,x_1)_V,(x_2,x_2)_V),\quad \forall x_1,x_2\in V
\]
(note that $[x_1^{\otimes b}]$ and $[x_2^{\otimes b}]$ are symmetric, hence already belong to the direct summand $\tS_{[\lambda]}(V)$ of $\tT_{[b]}(V)$ as $m=3$). Put $q_\lambda\coloneqq Q_\lambda(1,1,1,1)$, which is a nonzero rational number.
\end{example}

Now we recall the construction of generating series with weights and the corresponding Funke--Millson forms in the orthogonal case. Suppose that $F$ is a totally real field. Let $W_r$ be the standard symplectic space over $F$ of dimension $2r$, whose symplectic form is given by the matrix
\[
\tw_r\coloneqq\begin{pmatrix}  & 1_r \\ -1_r &  \end{pmatrix}.
\]

\begin{notation}
Denote by $\Sym_r$ the scheme of $r$-by-$r$ symmetric matrices. Put $\Sym^\circ_r\coloneqq\Sym_r\cap\GL_r$. Denote by $\Sym_r(F)^+\subseteq\Sym_r(F)$ the subset of totally semi-positive definite elements.
\end{notation}

Put $G_r\coloneqq\Sp(W_r)$ as a reductive group over $F$, equipped with homomorphisms
\begin{align*}
m&\colon\GL_{r,F}\to G_r,\qquad  a\mapsto
\begin{pmatrix} a  &  \\  &\pres{t}a^{-1}   \end{pmatrix}, \\
n&\colon\Sym_{r,F}\to G_r,\qquad  b\mapsto
\begin{pmatrix} 1_r  & b \\  & 1_r   \end{pmatrix}
\end{align*}
of algebraic groups over $F$.

For every $v\in\Hom(F,\dR)$, denote by $\fg_{r,v}$ the complexified Lie algebra of $G_{r,v}$. Denote by
\[
\dH_r\coloneqq\left\{\tau=x_{\tau}+\ti y_{\tau}\res x_{\tau}\in\Sym_r(F_\infty),y_{\tau}\in\Sym_r^\circ(F_\infty)^+\right\}
\]
the corresponding Siegel upper half space of genus $r$. For every $T\in\Sym_r(F)$, we introduce the holomorphic function
\[
\tq^T(\tau)\coloneqq\prod_{v\in\Hom(F,\dR)}\te^{2\pi\ti\tr(T\tau)_v}
\]
on $\dH_r$.

For every $v\in\Hom(F,\dR)$, the map
\[
A+\ti B\mapsto
\begin{pmatrix} A & B \\ -B & A \end{pmatrix}
\]
identifies $\rU(r)$ with a maximal compact subgroup $K_{r,v}$ of $G_r(F_v)$.\footnote{Here, $\rU(r)\subseteq\GL_r(\dC)$ is the standard unitary group defined by the identity matrix.} In particular, we have the character $\delta_v\colon K_{r,v}\to\dC^\times$ that is the determinant on $\rU(r)$. Put
\[
K_{r,\infty}\coloneqq\prod_{v\in\Hom(F,\dR)}K_{r,v},\quad
\delta\coloneqq\prod_{v\in\Hom(F,\dR)}\delta_v\colon K_{r,\infty}\to\dC^\times.
\]
Define
\[
\delta^{1/2}\colon\widetilde{K_{r,\infty}}\to\dC^\times
\]
as the pullback of $\det$ along the map $\dC^\times\to\dC^\times$ sending $z$ to $z^2$. For every tuple $b=(b_v)_v\in\dN^{\Hom(F,\dR)}$, we have the representation $\tT^b(\dC^r)$ (Notation \ref{no:lambda_orthogonal}) of $\fS_b\times K_{r,\infty}$, where $\dC^r$ is regarded as the standard representation of $K_{r,v}$ for every $v\in\Hom(F,\dR)$.

Fix a distinguished real place $v_0\in\Hom(F,\dR)$. Let $V$ be a (nondegenerate) quadratic space over $F$ of signature $(m-2,2)$ at $v_0$ and $(m,0)$ at other real places, for some integer $m\geq 3$. Put $H\coloneqq\SO(V)$ as a reductive group over $F$. Write $V_\infty$ for $V\otimes_\dQ\dR$. Take a tuple $\lambda=(\lambda_v)_v\in(\fW'_m)^{\Hom(F,\dR)}$. Put
\[
\tS_{[\lambda]}(V_\infty)\coloneqq\bigotimes_{v\in\Hom(F,\dR)}\tS_{[\lambda_v]}(V_v)
\]
(tensor over $\dR$) (Definition \ref{de:schur_orthogonal}) as a representation of $H(F_\infty)$ with real coefficients.

The corresponding symmetric domain $\fD$ of $\Res_{F/\dQ}H$ is the set of \emph{oriented} negative definite 2-planes in $V\otimes_{F,v_0}\dR$, which is naturally a hermitian symmetric domain of (complex) dimension $m-2$. Let $L\subseteq H(\bA_F^\infty)$ be a neat open compact subgroup. We have the complex Shimura variety
\[
X_L\coloneqq H(F)\backslash\(\fD\times H(\bA_F^\infty)/L\),
\]
together with the real local system
\[
H(F)\backslash\(\fD\times \tS_{[\lambda]}(V_\infty)\times H(\bA_F^\infty)/L\)
\]
on $X_L$, which we still denote as $\tS_{[\lambda]}(V_\infty)$ by abuse of notation.

For every tuple $x=(x_1,\dots,x_r)\in V^r$ satisfying $T(x)\in\Sym_r(F)^+$ and every element $h\in H(\bA_F^\infty)/L$, we have a special morphism
\[
\zeta_{x,h}\colon Y_{(hLh^{-1})^x}\xrightarrow{\cdot h} X_L
\]
give by the right translation by $h$, where $Y_{L_x}$ is the complex Shimura variety for the point-wise stabilizer $H^x$ of $x$ in $H$ and the open compact subgroup $(hLh^{-1})^x\coloneqq hLh^{-1}\cap H^x(\bA_F^\infty)$. The morphism $\zeta_{x,h}$ is finite and unramified of codimension $\rank T(x)$.

Now we suppose that $r\geq s(\lambda_v)$ for every $v\in\Hom(F,\dR)$, so that $\lambda^+\coloneqq(\lambda_v^+)_v$ is naturally an element of $(\fW_r^+)^{\Hom(F,\dR)}$. Put $b=(b_v\coloneqq b(\lambda_v^+))_v$. We have the real representation
\[
\tT_b(V_\infty)\coloneqq\bigotimes_{v\in\Hom(F,\dR)}\tT^{b_v}(V_v)
\]
(tensor over $\dR$) of $\fS_b\times H(F_\infty)$, together with canonical projections
\begin{align}\label{eq:harmonic_orthogonal}
[\phantom{a}]\colon\tT_b(V_\infty)\to
\tT_{[b]}(V_\infty)\coloneqq\bigotimes_{v\in\Hom(F,\dR)}\tT_{[b_v]}(V_v)
\end{align}
(Definition \ref{de:schur_orthogonal}) and
\begin{align}\label{eq:schur_orthogonal}
\pi_{\lambda^+}\coloneqq\otimes_v\pi_{\lambda_v^+}\colon\tT_{[b]}(V_\infty)\to\tS_{[\lambda]}(V_\infty)
\end{align}
(Notation \ref{no:lambda_orthogonal}). We adopt the convention that $\tT_b(V_\infty)=\{0\}$ for $b\in\dZ^{\Hom(F,\dR)}\setminus\dN^{\Hom(F,\dR)}$.

\begin{notation}
Put $\fF(\lambda^+,r)\coloneqq\prod_{v\in\Hom(F,\dR)}\fF(\lambda^+_v,r)$.
\begin{enumerate}
  \item Denote by $\{\epsilon_1,\dots,\epsilon_r\}$ the standard basis of $\dQ^r$. For every element $f=(f_v)_v\in\fF(\lambda^+,r)$, we have the vector
      \[
      \epsilon_f\coloneqq \otimes_{v\in\Hom(F,\dR)}\epsilon_{f_v}\in\tT^b(\dQ^r),
      \]
      so that $\{\pi_{\lambda^+}(\epsilon_f)\}$ is a basis of $\tS_{\lambda^+}(\dQ^r)$ \cite{FM06}*{Theorem~3.7}.

  \item For every tuple $x=(x_v)_v\in\prod_{v\in\Hom(F,\dR)}V_v^r=V_\infty^r$ with $x_v=(x_{v,1},\dots,x_{v,r})\in V_v^r$ and every element $f=(f_v)_v\in\fF(\lambda^+,r)$, we have the vector
      \[
      x_f\coloneqq\otimes_{v\in\Hom(F,\dR)}(x_v)_{f_v}\in\tT_b(V_\infty),
      \]
      hence the element $\pi_{\lambda^+}([x_f])\in\tS_{[\lambda]}(V_\infty)$.
\end{enumerate}
\end{notation}

For every $f\in\fF(\lambda^+,r)$, the element $\pi_{\lambda^+}([x_f])$ is invariant under $H^x(F_\infty)$. Thus we obtain a (unique) map
\begin{align}\label{eq:special_cycle_orthogonal}
\tS_{\lambda^+}(\dR^r)\to\zeta_{x,h}^*\tS_{[\lambda]}(V_\infty)
\end{align}
of real local systems on $X_L$ sending $\pi_{\lambda^+}(\epsilon_f)$ to the global section $\zeta_{x,h}^*\pi_{\lambda^+}([x_f])$ for every $f\in\fF(\lambda^+,r)$, where $\tS_{\lambda^+}(\dR^r)$ is regarded as the constant local system. By purity, \eqref{eq:special_cycle_orthogonal} induces a map
\[
\rH^0(Y_{L^x},\tS_{\lambda^+}(\dR^r))=\rH^0(Y_{L^x},\dR)\otimes_\dR\tS_{\lambda^+}(\dR^r)
\to\rH^{\rank T(x),\rank T(x)}(X_L,\tS_{[\lambda]}(V_\infty)_\dC)
\]
or
\[
\rH^0(Y_{L^x},\dR)\to\tS_{\lambda^+}(\dR^r)^\vee\otimes_\dR\rH^{\rank T(x),\rank T(x)}(X_L,\tS_{[\lambda]}(V_\infty)_\dC).
\]
Denote by
\[
\rZ^\lambda(x,h)\in\tS_{\lambda^+}(\dR^r)^\vee\otimes_\dR\rH^{\rank T(x),\rank T(x)}(X_L,\tS_{[\lambda]}(V_\infty)_\dC)
\]
the image of the fundamental class of $Y_{L^x}$ under the above map.\footnote{The element $\rZ^\lambda(x,h)$ vanishes unless $\rank T(x)\geq s(\lambda_v)$ for every $v\in\Hom(F,\dR)$.} Finally, put
\[
\rC^\lambda(x,h)\coloneqq \rZ^\lambda(x,h)\wedge \Omega_L^{r-\rank T(x)}\in
\tS_{\lambda^+}(\dR^r)^\vee\otimes_\dR\rH^{r,r}(X_L,\tS_{[\lambda]}(V_\infty)_\dC),
\]
where $\Omega_L\in\rH^{1,1}(X_L,\dC)$ denotes the Chern class of the dual tautological (line) bundle on $X_L$. It is easy to see that $\rC^\lambda(x,h)$ depends only on $h^{-1}x\in L\backslash(V^\infty)^r$. Thus, it makes sense to define, for every element $x\in L\backslash(V^\infty)^r$ satisfying $T(x)\in\Sym_r(F)^+$, an element
\[
\rC^\lambda(x)\in
\tS_{\lambda^+}(\dR^r)^\vee\otimes_\dR\rH^{r,r}(X_L,\tS_{[\lambda]}(V_\infty)_\dC),
\]
which is compatible under the pullback along the map $X_{L'}\to X_L$ for $L'\subseteq L$.

\begin{definition}\label{de:cycle_orthogonal}
Take a Schwartz function $\phi^\infty\in\sS((V^\infty)^r)^L$. For every $T\in\Sym_r(F)^+$, define
\[
\rC^\lambda_T(\phi^\infty)\coloneqq\sum_{x\in L\backslash(V^\infty)^r_T}
\phi^\infty(x)\cdot \rC^\lambda(x),
\]
which is indeed a finite sum. The \emph{generating series of special cycles in cohomology} with weight $\lambda$ is defined as the formal $q$-expansion (of genus $r$):
\[
\rC^\lambda(\phi^\infty)\coloneqq\sum_{T\in\Sym_r(F)^+}\rC^\lambda_T(\phi^\infty)\cdot \tq^T
\]
with coefficients in $\tS_{\lambda^+}(\dR^r)^\vee\otimes_\dR\rH^{r,r}(X_L,\tS_{[\lambda]}(V_\infty)_\dC)$.
\end{definition}

We recall the Funke--Millson form associated with the special cycles. In \cite{FM06}*{\S5}, the authors defined an element (with $q=2$ which we omit from the subscript, so that $\varphi_{r,b}$ really means $\varphi_{r\cdot 2,b}$ in \cite{FM06})
\[
\varphi_{r,b}\in\Hom_\dC\(\tT^b(\dC^r)\otimes\delta^{m/2},\sS(V_\infty^r)\otimes_\dC\sA^{r,r}(\fD)\otimes_\dR\tT_b(V_\infty)\)
^{\widetilde{K_{r,\infty}}\times H(F_\infty)\times\fS_b},
\]
where $\widetilde{K_{r,\infty}}\times H(F_\infty)$ acts on $\sS(V_\infty^r)$ via the Weil representation $\omega_{\psi_{F,\infty}}$ (with respect to $\psi_{F,\infty}$), and the other actions are obvious ones. Applying \eqref{eq:harmonic_orthogonal}, we obtain an element
\[
\varphi_{r,[b]}\in\Hom_\dC\(\tT^b(\dC^r)\otimes\delta^{m/2},\sS(V_\infty^r)\otimes_\dC\sA^{r,r}(\fD)
\otimes_\dR\tT_{[b]}(V_\infty)\)^{\widetilde{K_{r,\infty}}\times H(F_\infty)\times\fS_b}.
\]
Then applying the projector $\pi_{\lambda^+}$ to $\varphi_{r,[b]}$, we obtain the element
\[
\varphi_{r,\lambda}\in\Hom_\dC\(\tS_{\lambda^+}(\dC^r)\otimes\delta^{m/2},\sS(V_\infty^r)
\otimes_\dC\sA^{r,r}(\fD)\otimes_\dR\tS_{[\lambda]}(V_\infty)\)
^{\widetilde{K_{r,\infty}}\times H(F_\infty)}.
\]
For every $x\in V_\infty^r$ and $h\in H(\bA_F^\infty)$, we can evaluate $\varphi_{r,\lambda}$ at $x$ and then take its pushforward along the composite map
\begin{align}\label{eq:translate_orthogonal}
\fD\to X_{hLh^{-1}}\xrightarrow{\cdot h} X_L
\end{align}
to obtain an element
\[
\varphi_\lambda(x,h)_L\in\tS_{\lambda^+}(\dC^r)^\vee\otimes_\dC\sA^{r,r}(X_L,\tS_{[\lambda]}(V_\infty)),
\]
where $\sA^{\bullet,\bullet}(X_L,\tS_{[\lambda]}(V_\infty))\coloneqq\sA^{\bullet,\bullet}(X_L)\otimes_\dR\tS_{[\lambda]}(V_\infty)$. It is clear that $\varphi_\lambda(x,h)$ is invariant under the diagonal action of $H(F)$ on $V_\infty^r\times H(\bA_F^\infty)$.

Take a Schwartz function $\phi^\infty\in\sS((V^\infty)^r)^L$.

\begin{definition}[Generating function of Funke--Millson forms]\label{de:form_orthogonal}
Define a function $\rF^\lambda(\phi^\infty)$ on $\dH_r$ valued in $\tS_{\lambda^+}(\dC^r)^\vee\otimes_\dC\sA^{r,r}(X_L,\tS_{[\lambda]}(V_\infty)_\dC)$ by the formula
\begin{align*}
\rF^\lambda(\phi^\infty)(\tau)\coloneqq\sum_{(x,h)\in H(F)\backslash V^r\times H(\bA_F^\infty)/L}\phi^\infty(h^{-1}x)\cdot \(a.\varphi_\lambda(xa,h)_L\)\cdot\prod_{v\in\Hom(F,\dR)}\te^{2\pi\ti \tr(T(x)x_{\tau})_v},
\end{align*}
where $a\in\GL_r(F_\infty)$ is an arbitrary element satisfying $y_{\tau}=a\cdot\pres{t}a$ and $\det a\in(F_\infty)_{>0}$, which acts on $\phi_\lambda(xa)$ via the factor $\tS_{\lambda^+}(\dC^r)^\vee$. It is clear that $\rF^\lambda(\phi^\infty)(\tau)$ is absolutely convergent and well-defined (that is, independent of the choice of $a$).
\end{definition}

\begin{proposition}\label{pr:form_orthogonal}
Take a Schwartz function $\phi^\infty\in\sS((V^\infty)^r)^L$.
\begin{enumerate}
  \item The function $\rF^\lambda(\phi^\infty)$ is a (not necessarily holomorphic) Siegel modular form (over $F$ of genus $r$) for the $\widetilde{K_{r,\infty}}$-type $\tS_{\lambda^+}(\dC^r)^\vee\otimes\delta^{-m/2}$ with coefficients in closed forms in $\sA^{r,r}(X_L,\tS_{[\lambda]}(V_\infty))$.

  \item In view of (1), the associated cohomology class of $\rF^\lambda(\phi^\infty)$ coincides with $\rC^\lambda(\phi^\infty)$.

  \item The $q$-expansion $\rC^\lambda(\phi^\infty)$ is a (holomorphic) Siegel modular form for the $\widetilde{K_{r,\infty}}$-type $\tS_{\lambda^+}(\dC^r)^\vee\otimes\delta^{-m/2}$ with coefficients in $\rH^{r,r}(X_L,\tS_{[\lambda]}(V_\infty)_\dC)$. It is a cusp form if $r=\max_v\{s(\lambda_v)\}$.
\end{enumerate}
\end{proposition}

\begin{proof}
This follows from \cite{FM06}*{Theorem~7.1\&7.6} (together with Remark \ref{re:form_orthogonal} below).
\end{proof}

\begin{remark}\label{re:form_orthogonal}
Careful readers may notice that \cite{FM06} only addressed the case where weights away from $v_0$ are zero. In fact, the general case can be proved in the same way using the following modification.

Take an element $v\in\Hom(F,\dR)$. For $(j,k)\in\{1,\dots,b_v\}\times\{1,\dots,b_v-1\}$, we have an operator
\[
A_{jk}^v\colon\tT^{b_v-2}(V_v)\to\tT^{b_v}(V_v)
\]
given by the insertion of the map $\dR\to V_v\otimes_\dR V_v$ that is the dual of the quadratic form to the $(j,k+1)$-th (resp.\ $(k,j)$-th) spot when $j\leq k$ (resp.\ $k<j$); we regard it as a map
\[
A_{jk}^v\colon\tT_{b-2_v}(V_\infty)\to\tT_b(V_\infty),
\]
where $n_v=(0,\dots,n,\dots,0)$ denotes the element in $\dN^{\Hom(F,\dR)}$ with only nonzero entry $n$ at $v$ for $n\in\dN$. On the other hand, for $1\leq j\leq b_v$, we have an operator
\[
\varsigma_j^v\colon \sS(V_v)\otimes_\dR\tT^{b_v-1}(V_v)\to\sS(V_v)\otimes_\dR\tT^{b_v}(V_v)
\]
such that its value on $x\in V_v$ is the operator that inserts $x$ at the $j$-th spot; we regard it as a map
\[
\varsigma_j^v\colon\sS(V_\infty)\otimes_\dR\tT_{b-1_v}(V_\infty)\to\sS(V_\infty)\otimes_\dR\tT_b(V_\infty).
\]

Then by the same proof, \cite{FM06}*{Theorem~5.9(i)} generalizes to the relation
\[
[\varphi_{1,b}]=[\varsigma_j^v\varphi_{1,b-1_v}]+\frac{1}{4\pi}\sum_{k=1}^{b_v-1}[A_{jk}^v\varphi_{1,b-2_v}],\quad
1\leq j\leq b_v
\]
for every $v\in\Hom(F,\dR)$; and \cite{FM06}*{Theorem~6.11} generalizes to the relation
\[
[\omega_{\psi_{F,\infty}}(\rL_v)\varphi_{1,b}]=\frac{1}{8\pi}\sum_{j=1}^{b_v}\sum_{k=1}^{b_v-1}[A_{jk}^v\varphi_{1,b-2_v}]
\]
for every $v\in\Hom(F,\dR)$, where $\rL_v$ denotes the ``lowering operator'' $\frac{1}{2}\(\begin{smallmatrix} 1& -\ti \\ -\ti & -1 \end{smallmatrix}\)$ in $\fg_{1,v}$.\footnote{There seems to be a sign error in the formula for $\omega(L)$ in the proof of \cite{FM06}*{Theorem~6.11}.}
\end{remark}

\section{Funke--Millson theory for unitary groups}
\label{ss:3}

In this section, we develop the parallel theory for unitary groups. Let $E/F$ be an \'{e}tale quadratic extension of rings with $\ol{\phantom{a}}\in\Aut(E/F)$ the nontrivial involution, $r$ and $m$ positive integers.

\begin{definition}\label{de:schur_unitary}
Consider a (nondegenerate) hermitian space $V$ over $E/F$ of rank $m$ and an $F$-ring $\dL$ such that the set $\fU$ of $F$-ring homomorphisms from $E$ to $\dL$ is nonempty (hence consists of two elements).

Consider a pair $\lambda\in\fW_m^\fU$ satisfying $\lambda_{\ol{u}}=\ol{\lambda_u}$ for every $u\in\fU$ and put $b\coloneqq(b(\lambda^+_u))_u$. Assume that $m\cdot\prod_{u\in\fU}b_u!$ is invertible in $\dL$. Put
\[
\tT_b(V)\coloneqq\bigotimes_{u\in\fU}\tT^{b_u}(V_u)
\]
(tensor over $\dL)$ as an algebraic representation of $\rU(V)$ over $\dL$, where $\rU(V)$ acts on $V_u\coloneqq V\otimes_{E,u}\dL$ via the map $u$. For every pair $I=(i_u)_u$ with $1\leq i_u\leq b_u$, we have the contraction operator
\[
C_I\colon\tT_b(V)\to\tT_{b-1}(V)
\]
by ``pairing'' the $i_u$-th factor of $\tT_{b_u}(V_u)$ for the two elements $u\in\fU$, and the expansion operator
\[
A_I\colon\tT_{b-1}(V)\to\tT_b(V)
\]
by inserting the coevaluation map $\dL\to\bigotimes_{u\in\fU}V_u$ to the $i_u$-th spot of $\tT^{b_u}(V_u)$ for $u\in\fU$. If we put
\[
\tT_{[b]}(V)\coloneqq\bigcap_{I}\Ker A_I,
\]
then
\begin{align}\label{eq:decomposition_unitary}
\tT_b(V)=\tT_{[b]}(V)\oplus\(\bigoplus_{r=1}^{\min_{u\in\fU}\{\lfloor b_u/2\rfloor\}}\tT_{[b]}(V)_{b-r}\),
\end{align}
where
\[
\tT_{[b]}(V)_{b-r}\coloneqq\sum A_{I_1}\circ\cdots\circ A_{I_r}\(\tT_{[b-r]}(V)\).
\]
It is clear that the decomposition \eqref{eq:decomposition_unitary} is stable under the action of both $\rU(V)$ and $\fS_b$. In particular, we may define
\[
\tS_{[\lambda]}(V)\coloneqq(\otimes_{u\in\fU}\pi_{\lambda_u^+})\(\tT_{[b]}(V)\)
\]
(Notation \ref{no:lambda_orthogonal}), which is a direct summand of the $\dL[\rU(V)]$-module $\tT_{[b]}(V)$ hence $\tT^{b}(V)$.
\end{definition}

Now we construct of generating series with weights and the corresponding Funke--Millson forms in the unitary case. Suppose that $E/F$ is a totally imaginary extension of a totally real field, so that $\ol{\phantom{a}}$ is the complex conjugation. For every element $u\in\Hom(E,\dC)$, we denote by $v(u)\in\Hom(F,\dR)$ its induced real place of $F$.

\begin{notation}
Denote by $\Herm_r$ the $F$-scheme of $r$-by-$r$ hermitian matrices with respect to $E/F$. Put $\Herm^\circ_r\coloneqq\Herm_r\cap\Res_{E/F}\GL_r$. Denote by $\Herm_r(F)^+\subseteq\Herm_r(F)$ the subset of totally semi-positive definite elements.
\end{notation}

Let $W_r$ be the standard skew-hermitian space over $E/F$ of dimension $2r$, whose skew-hermitian form is given by the matrix $\tw_r$. Put $G_r\coloneqq\rU(W_r)$ as a reductive group over $F$, equipped with homomorphisms
\begin{align*}
m&\colon\Res_{E/F}\GL_r\to G_r,\qquad  a\mapsto
\begin{pmatrix} a  &  \\  &\pres{t}{\ol{a}}^{-1}   \end{pmatrix}, \\
n&\colon\Herm_{r,F}\to G_r,\qquad  b\mapsto
\begin{pmatrix} 1_r  & b \\  & 1_r   \end{pmatrix}
\end{align*}
of algebraic groups over $F$.

For every $v\in\Hom(F,\dR)$, denote by $\fg_{r,v}$ the complexified Lie algebra of $G_{r,v}$. Denote by
\[
\dH_r\coloneqq
\left\{\tau\in(\Res_{E/F}\Mat_r)(F_\infty)\left|\frac{1}{2\ti}\(\tau-\pres{t}{\ol\tau}\)\in\Herm_r^\circ(F_\infty)^+\right.\right\}
\]
the corresponding hermitian Siegel upper half space of genus $r$. In what follows, we write
\[
x_{\tau}\coloneqq\frac{1}{2}\(\tau+\pres{t}{\ol\tau}\),\quad
y_{\tau}\coloneqq\frac{1}{2\ti}\(\tau-\pres{t}{\ol\tau}\)
\]
for $\tau\in\dH_r$, so that $\tau=x_{\tau}+\ti y_{\tau}$.

For every $T\in\Herm_r(F)$, we introduce the holomorphic function
\[
\tq^T(\tau)\coloneqq\prod_{v\in\Hom(F,\dR)}\te^{2\pi\ti\tr(T\tau)_v}
\]
on $\dH_r$.

For every $v\in\Hom(F,\dR)$, the map
\[
(U_1,U_2)\mapsto u^{-1}\frac{1}{2}
\begin{pmatrix}
U_1+U_2   & -\ti U_1+\ti U_2 \\
\ti U_1-\ti U_2   & U_1+U_2 \\
\end{pmatrix},
\]
where $u\in\Hom(E,\dC)$ is an element above $v$, identifies $\rU(r)\times\rU(r)$ with a maximal compact subgroup $K_{r,v}$ of $G_r(F_v)$ that is independent of the choice of $u$. In addition, the map $(U_1,U_2)\to\det(U_1 U_2^{-1})$ defines a character $\delta_v\colon K_{r,v}\to\dC^\times$, which is also independent of the choice of $u$. Put
\[
K_{r,\infty}\coloneqq\prod_{v\in\Hom(F,\dR)}K_{r,v},\quad
\delta\coloneqq\prod_{v\in\Hom(F,\dR)}\delta_v\colon K_{r,\infty}\to\dC^\times.
\]
Define
\[
\delta^{1/2}\colon\widetilde{K_{r,\infty}}\to\dC^\times
\]
as the pullback of $\det$ along the map $\dC^\times\to\dC^\times$ sending $z$ to $z^2$. For every tuple $b=(b_u)_u\in\dN^{\Hom(E,\dC)}$, we have the representation $\tT^b(\dC^r)$ (Notation \ref{no:lambda_orthogonal}) of $\fS_b\times K_{r,\infty}$, where the action of $K_{r,\infty}$ is the unique one such that for every $k\in K_{r,\infty}$ and every $u\in\Hom(E,\dC)$, if we write $u(k)=(U_1,U_2)\in\rU(r)\times\rU(r)$, then $U_1$ acts on $\tT^{b_u}(\dC^r)$ via itself as an element of $\GL_r(\dC)$ and $U_2$ acts on $\tT^{b_{\ol{u}}}(\dC^r)$ via the element $\ol{U_2}$ of $\GL_r(\dC)$.

\begin{definition}
We say that a tuple $\lambda=(\lambda_u)_u\in\fW_m^{\Hom(E,\dC)}$ is \emph{hermitian} if $\lambda_{\ol{u}}=\ol{\lambda_u}$ holds for every element $u\in\Hom(E,\dC)$.
\end{definition}

Fix a distinguished embedding $u_0\in\Hom(E,\dC)$ with $v_0\in\Hom(F,\dR)$ the underlying real place of $F$. Let $V$ be a (nondegenerate) hermitian space over $E/F$ of signature $(p,q)$ at $v_0\in\Hom(F,\dR)$ and $(m,0)$ at other real places, for some integer $p,q\geq 1$ and $m=p+q$. Put $H\coloneqq\rU(V)$ as a reductive group over $F$. Write $V_\infty\coloneqq V\otimes_\dQ\dR$, and $V_u\coloneqq V\otimes_{E,u}\dC$ for $u\in\Hom(E,\dC)$. Take a hermitian tuple $\lambda=(\lambda_u)_u\in\fW_m^{\Hom(E,\dC)}$. Put
\[
\tS_{[\lambda]}(V_\infty)\coloneqq\bigotimes_{v\in\Hom(F,\dR)}\tS_{[\lambda_v]}(V_v)
\]
(tensor over $\dC$) as a representation of $H(F_\infty)$ with complex coefficients, where we have applied Definition \ref{de:schur_unitary} to $E/F=E_v/F_v$, $\fU=\{u\in\Hom(E,\dC)\res v(u)=v\}$, $\lambda_v=(\lambda_u)_{u\in\fU}$, and $\dL=\dC$ for every $v\in\Hom(F,\dR)$.

The corresponding symmetric domain $\fD$ of $\Res_{F/\dQ}H$ (with respect to $u_0$) is the set of negative definite $q$-planes in $V_{u_0}$, which is naturally a hermitian symmetric domain of (complex) dimension $pq$. Let $L\subseteq H(\bA_F^\infty)$ be a neat open compact subgroup. We have the complex Shimura variety
\[
X_L\coloneqq H(F)\backslash\(\fD\times H(\bA_F^\infty)/L\),
\]
together with the complex local system
\[
H(F)\backslash\(\fD\times \tS_{[\lambda]}(V_\infty)\times H(\bA_F^\infty)/L\)
\]
on $X_L$, which we still denote as $\tS_{[\lambda]}(V_\infty)$ by abuse of notation.

For every tuple $x=(x_1,\dots,x_r)\in V^r$ satisfying $T(x)\in\Sym_r(F)^+$ and every element $h\in H(\bA_F^\infty)/L$, we have a special morphism
\[
\zeta_{x,h}\colon Y_{(hLh^{-1})^x}\xrightarrow{\cdot h} X_L
\]
give by the right translation by $h$, where $Y_{L_x}$ is the complex Shimura variety for the point-wise stabilizer $H^x$ of $x$ in $H$ and the open compact subgroup $(hLh^{-1})^x\coloneqq hLh^{-1}\cap H^x(\bA_F^\infty)$. The morphism $\zeta_{x,h}$ is finite and unramified of codimension $q\cdot\rank T(x)$.

Now we suppose that $r\geq s(\lambda_u)$ for every $u\in\Hom(E,\dC)$, so that $\lambda^+\coloneqq(\lambda_u^+)_u$ is naturally an element of $(\fW_r^+)^{\Hom(E,\dC)}$. Put $b=(b_u\coloneqq b(\lambda_u^+))_u\in\dN^{\Hom(E,\dC)}$ and $b_v\coloneqq(b_u)_{v(u)=v}$ for every $v\in\Hom(F,\dR)$. We have the complex representation
\[
\tT_b(V_\infty)\coloneqq\bigotimes_{v\in\Hom(F,\dR)}\tT_{b_v}(V_v)
\]
(tensor over $\dC$) of $\fS_b\times H(F_\infty)$, together with canonical projections
\begin{align}\label{eq:harmonic_unitary}
[\phantom{a}]\colon\tT_b(V_\infty)\to
\tT_{[b]}(V_\infty)\coloneqq\bigotimes_{v\in\Hom(F,\dR)}\tT_{[b_v]}(V_v)
\end{align}
(Definition \ref{de:schur_unitary}) and
\begin{align}\label{eq:schur_unitary}
\pi_{\lambda^+}\coloneqq\otimes_u\pi_{\lambda_u^+}\colon\tT_{[b]}(V_\infty)\to\tS_{[\lambda]}(V_\infty)
\end{align}
(Notation \ref{no:lambda_orthogonal}). We adopt the convention that $\tT_b(V_\infty)=\{0\}$ for $b\in\dZ^{\Hom(F,\dR)}\setminus\dN^{\Hom(F,\dR)}$.

\begin{notation}
Put $\fF(\lambda^+,r)\coloneqq\prod_{u\in\Hom(E,\dC)}\fF(\lambda^+_u,r)$.
\begin{enumerate}
  \item Denote by $\{\epsilon_1,\dots,\epsilon_r\}$ the standard basis of $\dQ^r$. For every element $f=(f_u)_u\in\fF(\lambda^+,r)$, we have the vector
      \[
      \epsilon_f\coloneqq \otimes_{u\in\Hom(E,\dC)}\epsilon_{f_u}\in\tT^b(\dQ^r),
      \]
      so that $\{\pi_{\lambda^+}(\epsilon_f)\}$ is a basis of $\tS_{\lambda^+}(\dQ^r)$ \cite{FM06}*{Theorem~3.7}.

  \item For every tuple $x=(x_1,\dots,x_r)\in V_\infty^r$ and every element $f=(f_u)_u\in\fF(\lambda^+,r)$, we have the vector
      \[
      x_f\coloneqq\otimes_{u\in\Hom(E,\dC)}(x_u)_{f_u}\in\tT_b(V_\infty),
      \]
      where $x_u$ denotes the natural image of $x$ in $(V_u)^r$, hence the element $\pi_{\lambda^+}([x_f])\in\tS_{[\lambda]}(V_\infty)$.
\end{enumerate}
\end{notation}

For every $f\in\fF(\lambda^+,r)$, the element $\pi_{\lambda^+}([x_f])$ is invariant under $H^x(F_\infty)$. Thus we obtain a (unique) map
\begin{align}\label{eq:special_cycle_unitary}
\tS_{\lambda^+}(\dC^r)\to\zeta_{x,h}^*\tS_{[\lambda]}(V_\infty)
\end{align}
of complex local systems on $X_L$ sending $\pi_{\lambda^+}(\epsilon_f)$ to the global section $\zeta_{x,h}^*\pi_{\lambda^+}([x_f])$ for every $f\in\fF(\lambda^+,r)$, where $\tS_{\lambda^+}(\dC^r)$ is regarded as the constant local system. By purity, \eqref{eq:special_cycle_unitary} induces a map
\[
\rH^0(Y_{L^x},\tS_{\lambda^+}(\dC^r))=\rH^0(Y_{L^x},\dC)\otimes_\dC\tS_{\lambda^+}(\dC^r)
\to\rH^{q\rank T(x),q\rank T(x)}(X_L,\tS_{[\lambda]}(V_\infty))
\]
or
\[
\rH^0(Y_{L^x},\dC)\to\tS_{\lambda^+}(\dC^r)^\vee\otimes_\dC\rH^{q\rank T(x),q\rank T(x)}(X_L,\tS_{[\lambda]}(V_\infty)).
\]
Denote by
\[
\rZ^\lambda(x,h)\in\tS_{\lambda^+}(\dC^r)^\vee\otimes_\dC\rH^{q\rank T(x),q\rank T(x)}(X_L,\tS_{[\lambda]}(V_\infty))
\]
the image of the fundamental class of $Y_{L^x}$ under the above map.\footnote{The element $\rZ^\lambda(x,h)$ vanishes unless $\rank T(x)\geq s(\lambda_u)$ for every $u\in\Hom(E,\dC)$.} Finally, put
\[
\rC^\lambda(x,h)\coloneqq \rZ^\lambda(x,h)\wedge \Omega_L^{r-\rank T(x)}\in
\tS_{\lambda^+}(\dC^r)^\vee\otimes_\dC\rH^{rq,rq}(X_L,\tS_{[\lambda]}(V_\infty)),
\]
where $\Omega_L\in\rH^{q,q}(X_L,\dC)$ denotes the top Chern class of the dual tautological bundle on $X_L$. It is easy to see that $\rC^\lambda(x,h)$ depends only on $h^{-1}x\in L\backslash(V^\infty)^r$. Thus, it makes sense to define, for every element $x\in L\backslash(V^\infty)^r$ satisfying $T(x)\in\Sym_r(F)^+$, an element
\[
\rC^\lambda(x)\in
\tS_{\lambda^+}(\dC^r)^\vee\otimes_\dC\rH^{rq,rq}(X_L,\tS_{[\lambda]}(V_\infty)),
\]
which is compatible under the pullback along the map $X_{L'}\to X_L$ for $L'\subseteq L$.

\begin{definition}\label{de:cycle_unitary}
Take a Schwartz function $\phi^\infty\in\sS((V^\infty)^r)^L$. For every $T\in\Herm_r(F)^+$, define
\[
\rC^\lambda_T(\phi^\infty)\coloneqq\sum_{x\in L\backslash(V^\infty)^r_T}
\phi^\infty(x)\cdot \rC^\lambda(x),
\]
which is indeed a finite sum. The \emph{generating series of special cycles in cohomology} of weight $\lambda$ is defined as the formal $q$-expansion (of genus $r$):
\[
\rC^\lambda(\phi^\infty)\coloneqq\sum_{T\in\Herm_r(F)^+}\rC^\lambda_T(\phi^\infty)\cdot \tq^T
\]
with coefficients in $\tS_{\lambda^+}(\dC^r)^\vee\otimes_\dC\rH^{rq,rq}(X_L,\tS_{[\lambda]}(V_\infty))$.
\end{definition}

Next, we construct the Funke--Millson form associated with the special cycles. Fix a base point of $\fD$ and denote its stabilizer in $H(F_{v_0})$ by $L_{v_0}$ (so that $\fD\simeq H(F_{v_0})/L_{v_0}$). The base point gives rise to an orthogonal decomposition $V_{v_0}=V_{v_0}^+\oplus V_{v_0}^-$ in which $V_{v_0}^{+/-}$ is positive/nagative definite. Put $L_\infty\coloneqq L_{v_0}\times\prod_{v\neq v_0}H(F_v)$, which is a maximal compact subgroup of $H(F_\infty)$. Denote by $\fh\coloneqq V_{u_0}\otimes_\dC V_{\ol{u_0}}$ the complexified Lie algebra of the real Lie group $H(F_{v_0})$, which has the Harish-Chandra decomposition $\fh=\fl\oplus\fp$ where $\fl$ is the complexified Lie algebra of $L_{v_0}$.

Choose an orthonormal basis\footnote{This in particular requires that every element in the basis has norm either $1$ or $-1$.} $\{e_{v,1},\dots,e_{v,m}\}$ of $V_v$ for every $v\in\Hom(F,\dR)$ so that for $v=v_0$, $e_{v_0,\alpha}\in V_{v_0}^+$ for $1\leq \alpha\leq p$ and $e_{v_0,\mu}\in V_{v_0}^-$ for $p+1\leq \mu\leq m$. For $u\in\Hom(E,\dC)$, we write $e^u_j\in V_u$ for $u(e_{v(u),j})$.

We recall the description of $\fp$ from \cite{KM90}*{\S 5}. Define elements $X_{jk}\in\fh$ for $1\leq j<k\leq m$ and $X'_{jk}\in\fh$ for $1\leq j\leq k\leq m$ by
\[
X_{jk}\coloneqq -(e^{u_0}_j\otimes e^{\ol{u_0}}_k-e^{u_0}_k\otimes e^{\ol{u_0}}_j),\qquad
X'_{jk}\coloneqq -\ti(e^{u_0}_j\otimes e^{\ol{u_0}}_k+e^{u_0}_k\otimes e^{\ol{u_0}}_j).
\]
Then $\{X_{jk}\res 1\leq j<k\leq m\}\cup\{X'_{jk}\res 1\leq j\leq k\leq m\}$ forms a (complex) basis for $\fh$. Let $\{\omega_{jk},\omega'_{jk}\}$ be the dual basis for $\fh^\vee$. Then $\{\omega_{\alpha\mu},\omega'_{\alpha\mu}\res 1\leq \alpha\leq p,p+1\leq \mu\leq m\}$ is a basis for $\fp^\vee$. Put
\[
\xi_{\alpha\mu}\coloneqq\omega_{\alpha\mu}+\ti\omega'_{\alpha\mu},\qquad 1\leq \alpha\leq p,p+1\leq \mu\leq m.
\]
Then $\{\xi_{\alpha\mu}\}$ forms a basis of the (holomorphic) cotangent space of $\fD=H(F_{v_0})/L_{v_0}$ at the identity coset.

The basis $\{e_{v,1},\dots,e_{v,m}\res v\in\Hom(F,\dR)\}$ we chose gives rise to every $u\in\Hom(E,\dC)$ a coordinate system $(z_{ij}^u)_{1\leq i\leq r,1\leq j\leq m}\in(\dC^m)^r$ to $V_\infty^r$. Define the Gaussian function $\nu_{r,0}\in\sS(V_\infty^r)$ by the formula
\[
\nu_{r,0}(x)\coloneqq\prod_{v\in\Hom(F,\dR)}\prod_{i=1}^r\prod_{j=1}^m\te^{-\pi\prod_{v(u)=v}z_{ij}^u}.
\]
Following \cite{KM86}, for $1\leq i\leq n$, we define two (Howe) operators
\[
\rh_i,\ol\rh_i\colon\sS(V_\infty^r)\otimes_\dC\wedge^\bullet\fp^\vee\to\sS(V_\infty^r)\otimes_\dC\wedge^{\bullet+ rq}\fp^\vee
\]
by the formulae
\begin{align*}
\rh_i&\coloneqq\frac{1}{2^q}\prod_{\mu=p+1}^m\(\sum_{\alpha=1}^p \(z_{i\alpha}^{\ol{u_0}}-\frac{1}{\pi}\frac{\partial}{\partial z_{i\alpha}^{u_0}}\)\otimes L(\xi_{\alpha\mu})\), \\
\ol\rh_i&\coloneqq\frac{1}{2^q}\prod_{\mu=p+1}^m\(\sum_{\alpha=1}^p \(z_{i\alpha}^{u_0}-\frac{1}{\pi}\frac{\partial}{\partial z_{i\alpha}^{\ol{u_0}}}\)\otimes L(\ol{\xi_{\alpha\mu}})\),
\end{align*}
where $L(\xi)$ denotes the left multiplication by an element $\xi\in\fp^\vee$. Put
\[
\varphi_{r,0}\coloneqq\(\ti^q\cdot\rh_1\ol\rh_1\nu_{r,0}\)\wedge\cdots\wedge\(\ti^q\cdot\rh_r\ol\rh_r\nu_{r,0}\)
\in\sS(V_\infty^r)\otimes_\dC\wedge^{rq,rq}\fp^\vee,
\]
where $\wedge^{rq,rq}\fp^\vee$ denotes the subspace of $\wedge^{2rq}\fp^\vee$ of Hodge type $(rq,rq)$. Then $\varphi_{r,0}$ is the Kudla--Millson form (for the trivial weight).

In general, write
\[
\tT(V_\infty)\coloneqq\bigoplus_{c\in\dN^{\Hom(E,\dC)}}\tT_c(V_\infty),
\]
and consider
\begin{align}\label{eq:algebra}
\fA\coloneqq\End_\dC\(\sS(V_\infty^r)\otimes_\dC\wedge^\bullet\fp^\vee\otimes_\dC\tT(V_\infty)\).
\end{align}
For $1\leq i\leq n$ and $u\in\Hom(E,\dC)$, define an element
\[
\rE_i^u\coloneqq
\begin{dcases}
\frac{1}{2}\sum_{\alpha=1}^p\(z_{i\alpha}^u-\frac{1}{\pi}\frac{\partial}{\partial z_{i\alpha}^{\ol{u}}}\)\otimes 1\otimes L(e^u_\alpha),&\text{if $v(u)=v_0$},\\
\frac{1}{2}\sum_{j=1}^m\(z_{ij}^u-\frac{1}{\pi}\frac{\partial}{\partial z_{ij}^{\ol{u}}}\)\otimes 1\otimes L(e^u_j),&\text{if $v(u)\neq v_0$}
\end{dcases}
\]
in $\fA$, where $L(e^u_j)$ denotes the left multiplication by $e^u_j$ in $\tT(V_\infty)$. For every $u\in\Hom(E,\dC)$, we define a homomorphism $\rE^u\colon\dQ^r\to\fA$ by the formula
\[
\rE^u(\epsilon_i)=\rE_i^u.
\]
The collection $\{\rE^u\}_u$ induces a complex linear map
\[
\rE^b\colon\tT^b(\dC^r)\to\fA
\]
using composition on $\fA$. Finally, we define an element
\[
\varphi_{r,b}\in\Hom_\dC\(\tT^b(\dC^r),\sS(V_\infty^r)\otimes_\dC\wedge^{rq,rq}\fp^\vee\otimes_\dC\tT_b(V_\infty)\)
\]
by the formula
\[
\varphi_{r,b}(w)=\rE^b(w)(\varphi_{r,0}),\quad w\in\tT^b(\dC^r).
\]
It is clear that $\varphi_{r,b}$ takes values in $\(\sS(V_\infty^r)\otimes_\dC\wedge^{rq,rq}\fp^\vee\otimes_\dC\tT_b(V_\infty)\)^{L_\infty}$, which is canonically isomorphic to $\(\sS(V_\infty^r)\otimes_\dC\sA^{rq,rq}(\fD)\otimes_\dC\tT_b(V_\infty)\)^{H(F_\infty)}$ by taking $H(F_\infty)$-translation. Thus, we will, according to the context, also regard
\[
\varphi_{r,b}\in\Hom_\dC\(\tT^b(\dC^r),\sS(V_\infty^r)\otimes_\dC\sA^{rq,rq}(\fD)\otimes_\dC\tT_b(V_\infty)\).
\]
It is straightforward to check that $\varphi_{r,b}$ does not depend on the choice of $L_{v_0}$ and the bases $\{e_{v,j}\}$.

\begin{lem}\label{le:invariance}
The element $\varphi_{r,b}$ belongs to
\[
\Hom_\dC\(\tT^b(\dC^r)\otimes\delta^{m/2},\sS(V_\infty^r)\otimes_\dC\sA^{rq,rq}(\fD)\otimes_\dC\tT_b(V_\infty)\)
^{\widetilde{K_{r,\infty}}\times H(F_\infty)\times\fS_b},
\]
where $\widetilde{K_{r,\infty}}\times H(F_\infty)$ acts on $\sS(V_\infty^r)$ via the Weil representation $\omega_{\psi_{F,\infty}}$ (with respect to $\psi_{F,\infty}$), and the other actions are obvious ones.
\end{lem}

\begin{proof}[Proof of Lemma \ref{le:invariance}]
The invariance under $H(F_\infty)$ is immediate from the construction. The invariance under $\fS_b$ follows by the same proof of \cite{FM06}*{Proposition~5.2}. For the invariance under $\widetilde{K_{r,\infty}}$, we use the Fock model and mimic the proof of \cite{FM06}*{Theorem~6.2}. The Fock model of the infinitesimal Weil representation of $G_r(F_\infty)\times H(F_\infty)$ with respect to the bases $\{e_{v,j}\}$ is realized on the polynomial ring
\[
\sF(V_\infty^r)\coloneqq\dC[\sfz_{ij}^u\res u\in\Hom(E,\dC),1\leq i\leq r,1\leq j\leq m].
\]
We normalize the intertwining operator so that the following operators between the Schr\"{o}dinger and the Fock models correspond:
\begin{align*}
z_{i\alpha}^{u}-\frac{1}{\pi}\frac{\partial}{\partial z_{i\alpha}^{\ol{u}}} &\leftrightarrow \frac{-\ti}{\sqrt{2}\pi}\sfz^{u}_{i\alpha},\quad v(u)=v_0,1\leq\alpha\leq p,\\
z_{i\mu}^{u}-\frac{1}{\pi}\frac{\partial}{\partial z_{i\mu}^{\ol{u}}} &\leftrightarrow \frac{\ti}{\sqrt{2}\pi}\sfz^{u}_{i\mu},\quad v(u)=v_0,p+1\leq\mu\leq m,\\
z_{ij}^{u}-\frac{1}{\pi}\frac{\partial}{\partial z_{ij}^{\ol{u}}} &\leftrightarrow \frac{-\ti}{\sqrt{2}\pi}\sfz^{u}_{ij},\quad v(u)\neq v_0,1\leq j\leq m.
\end{align*}
Indeed, for every fixed $u$, this is the Fock model adopted in \cite{FH21}*{Appendix~B} (while our variables $\sfz^u_{i\alpha},\sfz^{u}_{i\mu},\sfz^{u}_{ij}$ are denoted by $z'_{\alpha(r+i)},z'_{\mu i},z'_{j(r+i)}$ there, respectively).

For a multi-index $\balpha=(\alpha_1,\dots,\alpha_q)\in\{1,\dots,p\}^q$, we put
\[
\sfz_{i\balpha}^u\coloneqq\sfz_{i\alpha_1}^u\cdots\sfz_{i\alpha_q}^u\in\sF(V_\infty^r),\quad 1\leq i\leq n
\]
for $u$ above $v_0$, and denote by $\omega_{\balpha}$ and $\ol\omega_{\balpha}$ the forms in $\sA^{q,0}(\fD)$ and $\sA^{0,q}(\fD)$ giving by the elements
\[
L(\xi_{\alpha_1 p+1})\cdots L(\xi_{\alpha_q p+q}),\quad
L(\ol{\xi_{\alpha_1 p+1}})\cdots L(\ol{\xi_{\alpha_q p+q}})
\]
in $\wedge^q\fp^\vee$. Then in the Fock model,
\[
\varphi_{r,0}=\ti^{rq}\(\frac{-1}{8\pi^2}\)^{rq}\sum_{\balpha_1,\dots,\balpha_r}\sum_{\balpha'_1,\dots,\balpha'_r}
\sfz^{\ol{u_0}}_{1\balpha_1}\cdots\sfz^{\ol{u_0}}_{r\balpha_r}\cdot\sfz^{u_0}_{1\balpha'_1}\cdots\sfz^{u_0}_{r\balpha'_r}
\otimes\(\omega_{\balpha_1}\wedge\ol\omega_{\balpha'_1}\)\wedge\cdots\wedge\(\omega_{\balpha_r}\wedge\ol\omega_{\balpha'_r}\)
\]
in $\sF(V_\infty^r)\otimes_\dC\sA^{rq,rq}(\fD)$.

In the case of nontrivial weights, take an element $u\in\Hom(E,\dC)$. We introduce the set
\[
\tB_u\coloneqq
\begin{dcases}
\{1,\dots,p\}^{b_u}, & \text{if $v(u)=v_0$},\\
\{1,\dots,m\}^{b_u}, & \text{if $v(u)\neq v_0$}.
\end{dcases}
\]
For $\bbeta=(\beta_1,\dots,\beta_{b_u})\in\tB_u$, put
\[
\sfz^{\ol{u}}_{i\bbeta}\coloneqq\sfz_{i\beta_1}^{\ol{u}}\cdots\sfz_{i\beta_{b_u}}^{\ol{u}}\in\sF(V_\infty^r),\quad 1\leq i\leq n,
\]
and
\[
e^u_{\bbeta}\coloneqq e^u_{\beta_1}\otimes\cdots\otimes e^u_{\beta_{b_u}}\in\tT^{b_u}(V_u).
\]
Define a complex linear map $\varphi_b^u\colon\tT^{b_u}(\dC^r)\to\sF(V_\infty^r)\otimes_\dC\tT^{b_u}(V_u)$ by the formula
\[
\varphi_b^u(\epsilon_{i_1}\otimes\cdots\otimes\epsilon_{{i_{b_u}}})=\(\frac{-\ti}{2\sqrt{2}\pi}\)^{b_u}
\sum_{\bbeta\in\tB_u}\sfz^{\ol{u}}_{i_1\beta_1}\cdots\sfz^{\ol{u}}_{i_{b_u}\beta_{b_u}}\otimes e^u_{\bbeta}.
\]
Taking the tensor product over all $u\in\Hom(E,\dC)$ and using the product in $\sF(V_\infty^r)$, we obtain a linear map
\[
\varphi_b\colon\tT^b(\dC^r)\to\sF(V_\infty^r)\otimes_\dC\tT_b(V_\infty).
\]

By definition, we have $\varphi_{r,b}=\varphi_{r,0}\cdot\varphi_b$ in the Fock model. Denote by $\tilde\omega_{\psi_{F,\infty}}$ the Weil representation twisted by the character $\delta^{-m/2}$.

Note that the complexified Lie algebra of $\widetilde{K_{r,\infty}}$ is generated by elements $\{\sigma^u_{j,k}\res u\in\Hom(E,\dC),1\leq j,k\leq r\}$, where $\sigma^u_{j,k}$ denotes the endomorphism of $\dC^r$, regarded in the copy indexed by $u\in\Hom(E,\dC)$, that sends $\epsilon_j$ to $\epsilon_k$ and annihilates the rest elements in the basis and the rest copies of $\dC^r$. Thus, it suffices to show that, for every $u\in\Hom(E,\dC)$ and $1\leq j,k\leq r$, $\sigma^u_{j,k}$ annihilates $\varphi_{r,b}$, that is,
\begin{align}\label{eq:invariance}
\tilde\omega_{\psi_{F,\infty}}(\sigma^u_{j,k})\(\varphi_{r,b}(\epsilon_{i_1}\otimes\cdots\otimes\epsilon_{{i_{b_{u'}}}})\)
=\varphi_{r,b}\(\sigma^u_{j,k}(\epsilon_{i_1}\otimes\cdots\otimes\epsilon_{{i_{b_{u'}}}})\)
\end{align}
for every $u'\in\Hom(E,\dC)$ (here, $\epsilon_i$'s are regarded in the copy indexed by $u'$). We show this for $u$ above $v_0$ and leave the other similar case to the readers. Indeed, by \cite{FM06}*{Lemma~A.1}, we have
\begin{align*}
\tilde\omega_{\psi_{F,\infty}}(\sigma^u_{j,k})\(\varphi_{r,b}(\epsilon_{i_1}\otimes\cdots\otimes\epsilon_{{i_{b_{u'}}}})\)
&=\tilde\omega_{\psi_{F,\infty}}(\sigma^u_{j,k})(\varphi_{r,0})
\cdot\varphi_b\(\sigma^u_{j,k}(\epsilon_{i_1}\otimes\cdots\otimes\epsilon_{{i_{b_{u'}}}})\) \\
&\quad+\varphi_{r,0}\cdot\sum_{\alpha=1}^p\sfz^u_{k\alpha}\frac{\partial}{\partial\sfz^u_{j\alpha}}
\(\varphi_b(\epsilon_{i_1}\otimes\cdots\otimes\epsilon_{{i_{b_{u'}}}})\).
\end{align*}
We have $\tilde\omega_{\psi_{F,\infty}}(\sigma^u_{j,k})(\varphi_{r,0})=0$ since $\varphi_{r,0}\in\(\delta^{-m/2}\otimes\sS(V_\infty^r)\otimes_\dC\sA^{rq,rq}(\fD)\)^{\widetilde{K_{r,\infty}}}$ by \cite{KM86}*{Theorem~3.1}. On the other hand, it is straightforward to see that
\[
\sum_{\alpha=1}^p\sfz^u_{k\alpha}\frac{\partial}{\partial\sfz^u_{j\alpha}}
\(\varphi_b(\epsilon_{i_1}\otimes\cdots\otimes\epsilon_{{i_{b_{u'}}}})\)
=\varphi_b\(\sigma^u_{j,k}(\epsilon_{i_1}\otimes\cdots\otimes\epsilon_{{i_{b_{u'}}}})\)
\]
(note that both sides are zero if $u'\neq u$). Thus, \eqref{eq:invariance} holds. The lemma follows.
\end{proof}

\begin{remark}[Pullback trick]\label{re:pullback}
When $r=1$, we can recover the Funke--Millson form $\varphi_{1,b}$ from the orthogonal case. Put $V'\coloneqq\Res_{E/F}V$, regarded as a quadratic space over $F$ of dimension $2m$, with signature $(2p,2q)$ at $v_0$ and $(2m,0)$ at other real places of $F$. Thus, for every $v\in\Hom(F,\dR)$, $V_v\otimes_\dR\dC=\bigoplus_{v(u)=v}V_u$ as complex vector spaces. Define $b'\in\dN^{\Hom(F,\dR)}$ by the formula $b'_v=\sum_{v(u)=v}b_u$. Then $\tT^{b'}(\dC)$ is canonically identified with $\tT^b(\dC)$ and we have a natural projection map $\tT_{b'}(V'_\infty)_\dC\to\tT_b(V_\infty)$. Denote by $\fD'$ the (real) symmetric domain of $\Res_{F/\dQ}\SO(V')$ of dimension $4pq$, so that we have a natural embedding $\fD\to\fD'$ of real symmetric domains. Taking the pullback of forms along such an embedding, combined with the projection $\tT_{b'}(V'_\infty)_\dC\to\tT_b(V_\infty)$ and the identification $V'_\infty=V_\infty$ as real vector spaces, we obtain a map
\[
\varrho^\bullet\colon\sS(V'_\infty)\otimes_\dC\sA^\bullet(\fD')\otimes_\dC\tT_{b'}(V'_\infty)_\dC
\to\sS(V_\infty)\otimes_\dC\sA^\bullet(\fD)\otimes_\dC\tT_b(V_\infty).
\]
It is straightforward to check that the following diagram
\[
\xymatrix{
\tT^{b'}(\dC) \ar[r]^-{\varphi_{1,b'}}\ar@{=}[d] & \sS(V'_\infty)\otimes_\dC\sA^{2q}(\fD')\otimes_\dC\tT_{b'}(V'_\infty)_\dC \ar[d]^-{\varrho^{2q}} \\
\tT^{b}(\dC) \ar[r]^-{\varphi_{1,b}} & \sS(V_\infty)\otimes_\dC\sA^{2q}(\fD)\otimes_\dC\tT_b(V_\infty)
}
\]
commutes.\footnote{In Section \ref{ss:2}, we only mention the Funke--Millson form for $2q=2$; the general case is of course done in \cite{FM06}, more precisely, \S5.2.}
\end{remark}

Applying \eqref{eq:harmonic_unitary} to $\varphi_{r,b}$, we obtain an element
\[
\varphi_{r,[b]}\in\Hom_\dC\(\tT^b(\dC^r)\otimes\delta^{m/2},\sS(V_\infty^r)\otimes_\dC\sA^{rq,rq}(\fD)
\otimes_\dC\tT_{[b]}(V_\infty)\)^{\widetilde{K_{r,\infty}}\times H(F_\infty)\times\fS_b}.
\]
Then applying the projector $\pi_{\lambda^+}$ to $\varphi_{r,[b]}$, we obtain the element
\[
\varphi_{r,\lambda}\in\Hom_\dC\(\tS_{\lambda^+}(\dC^r)\otimes\delta^{m/2},\sS(V_\infty^r)\otimes_\dC\sA^{rq,rq}(\fD)
\otimes_\dC\tS_{[\lambda]}(V_\infty)\)^{\widetilde{K_{r,\infty}}\times H(F_\infty)}.
\]
For every $x\in V_\infty^r$ and $h\in H(\bA_F^\infty)$, we can evaluate $\varphi_{r,\lambda}$ at $x$ and then take its pushforward along the composite map
\begin{align*}
\fD\to X_{hLh^{-1}}\xrightarrow{\cdot h} X_L
\end{align*}
to obtain an element
\[
\varphi_\lambda(x,h)_L\in\tS_{\lambda^+}(\dC^r)^\vee\otimes_\dC\sA^{rq,rq}(X_L,\tS_{[\lambda]}(V_\infty)),
\]
where $\sA^{\bullet,\bullet}(X_L,\tS_{[\lambda]}(V_\infty))\coloneqq\sA^{\bullet,\bullet}(X_L)\otimes_\dC\tS_{[\lambda]}(V_\infty)$. It is clear that $\varphi_\lambda(x,h)$ is invariant under the diagonal action of $H(F)$ on $V_\infty^r\times H(\bA_F^\infty)$.

Take a Schwartz function $\phi^\infty\in\sS((V^\infty)^r)^L$.

\begin{definition}[Generating function of Funke--Millson forms]\label{de:form_unitary}
Define a function $\rF^\lambda(\phi^\infty)$ on $\dH_r$ valued in $\tS_{\lambda^+}(\dC^r)^\vee\otimes_\dC\sA^{rq,rq}(X_L,\tS_{[\lambda]}(V_\infty)_\dC)$ by the formula
\begin{align*}
\rF^\lambda(\phi^\infty)(\tau)\coloneqq\sum_{(x,h)\in H(F)\backslash V^r\times H(\bA_F^\infty)/L}\phi^\infty(h^{-1}x)\cdot \(a.\varphi_\lambda(xa,h)_L\)\cdot\prod_{v\in\Hom(F,\dR)}\te^{2\pi\ti \tr(T(x)x_{\tau})_v},
\end{align*}
where $a\in\GL_r(E_\infty)$ is an arbitrary element satisfying $y_{\tau}=a\cdot\pres{t}{\ol{a}}$ and $\det a\in(F_\infty)_{>0}$, which acts on $\phi_\lambda(xa)$ via the factor $\tS_{\lambda^+}(\dC^r)^\vee$. It is clear that $\rF^\lambda(\phi^\infty)(\tau)$ is absolutely convergent and well-defined (that is, independent of the choice of $a$).
\end{definition}

\begin{proposition}\label{pr:form_unitary}
Take a Schwartz function $\phi^\infty\in\sS((V^\infty)^r)^L$.
\begin{enumerate}
  \item The function $\rF^\lambda(\phi^\infty)$ is a (not necessarily holomorphic) hermitian Siegel modular form (of genus $r$ with respect to the CM extension $E/F$) for the $\widetilde{K_{r,\infty}}$-type $\tS_{\lambda^+}(\dC^r)^\vee\otimes\delta^{-m/2}$ with coefficients in closed forms in $\sA^{rq,rq}(X_L,\tS_{[\lambda]}(V_\infty))$.

  \item In view of (1), the associated cohomology class of $\rF^\lambda(\phi^\infty)$ coincides with $\rC^\lambda(\phi^\infty)$.

  \item The $q$-expansion $\rC^\lambda(\phi^\infty)$ is a (holomorphic) hermitian Siegel modular form for the $\widetilde{K_{r,\infty}}$-type $\tS_{\lambda^+}(\dC^r)^\vee\otimes\delta^{-m/2}$ with coefficients in $\rH^{rq,rq}(X_L,\tS_{[\lambda]}(V_\infty))$. It is a cusp form if $r=\max_u\{s(\lambda_u)\}$.
\end{enumerate}
\end{proposition}

\begin{proof}
First consider (1). By Lemma \ref{le:invariance}, $\rF^\lambda(\phi^\infty)$ is a hermitian Siegel modular form for the $\widetilde{K_{r,\infty}}$-type $\tS_{\lambda^+}(\dC^r)^\vee\otimes\delta^{-m/2}$. It remains to show that it takes values in closed forms on $X_L$. We show that indeed $\varphi_{r,b}$ already takes values in closed forms on $\fD$. By the same argument for \cite{FM06}*{Theorem~5.7}, it suffices to consider the case where $r=1$. When $r=1$, we use the pullback trick in Remark \ref{re:pullback} to pass to the orthogonal case, which has already been addressed in \cite{FM06}. Since it is easy to see that the map $\varrho^\bullet$ commutes with differential operators, $\varphi_{1,b}$ takes values in closed forms in $\sA^{2q}(\fD)$ by \cite{FM06}*{Theorem~6.3}  (and also Remark \ref{re:form_orthogonal}).

The proof of (2) follows from the same argument for \cite{FM06}*{Theorem~7.6} using Lemma \ref{le:form_unitary_1} below to replace \cite{FM06}*{Theorem~5.9(ii)}.

The proof of (3) follows from the same argument for \cite{FM06}*{Theorem~5.10} using Lemma \ref{le:form_unitary_2}(2) below.
\end{proof}

For every $u\in\Hom(E,\dC)$ with $v=v(u)$, we have a map
\[
\sigma^u\colon\dC^r\to(V_v^r)^\vee\otimes_{F_v,u}V_u
\]
sending $\epsilon_i$ to the projection to the $i$-th factor (then composing with $u$) for $1\leq i\leq r$. They together induce a map
\[
\sigma^b\colon\tT^b(\dC^r)\to(V_\infty^r)^\vee\otimes_{F_\infty}\tT_b(V_\infty).
\]
Taking the restriction along the direct summand $\tS_{\lambda^+}(\dC^r)\subset\tT^b(\dC^r)$, we obtain the map
\[
\sigma^b\res_{\tS_{\lambda^+}(\dC^r)}\colon\tS_{\lambda^+}(\dC^r)\to(V_\infty^r)^\vee\otimes_{F_\infty}\tT_b(V_\infty).
\]
Composing with the projection $\pi_{\lambda^+}\circ[\phantom{a}]\colon\tT_b(V_\infty)\to\tS_{[\lambda]}(V_\infty)$, we obtain the map
\[
\sigma_\lambda\colon\tS_{\lambda^+}(\dC^r)\to(V_\infty^r)^\vee\otimes_{F_\infty}\tS_{[\lambda]}(V_\infty).
\]
Regard $(V_\infty^r)^\vee\otimes_{F_\infty}\tS_{[\lambda]}(V_\infty)$ as a subspace of $\fA$ \eqref{eq:algebra} in the following way: for $l\otimes v\in(V_\infty^r)^\vee\otimes_{F_\infty}\tS_{[\lambda]}(V_\infty)$ and $\phi\otimes X\otimes w\in\sS(V_\infty^r)\otimes\wedge^{rq,rq}\fp^\vee\otimes\tT(V_\infty)$, $((l\otimes v).(\phi\otimes X\otimes w))(x)=l(x)\phi(x)\otimes X\otimes(v\otimes w)$ for $x\in V_\infty^r$. Then we obtain a map
\[
\sigma_\lambda\colon\tS_{\lambda^+}(\dC^r)\to\fA.
\]

\begin{lem}\label{le:form_unitary_1}
We have in cohomology that
\[
[\varphi_{r,\lambda}]=[\sigma_\lambda(\varphi_{r,0})].
\]
\end{lem}

\begin{proof}
This follows from Lemma \ref{le:form_unitary_2}(1) below, via the same argument for \cite{FM06}*{Theorem~5.9}.
\end{proof}

Take an element $u\in\Hom(E,\dC)$. For $(i,j)\in\{1,\dots,b_u\}\times\{1,\dots,b_{\ol{u}}\}$, we have an operator
\[
A_{ij}^u\colon\tT^{b_u-1}(V_u)\otimes_\dC\tT^{b_{\ol{u}}-1}(V_{\ol{u}})\to\tT^{b_u}(V_u)\otimes_\dC\tT^{b_{\ol{u}}}(V_{\ol{u}})
\]
given by the insertion of the map $\dC\to V_u\otimes_\dC V_{\ol{u}}$ that is the dual of the hermitian form to the $(i,j)$-th spot; we regard it as a map
\[
A_{ij}^u\colon\tT_{b-1_u-1_{\ol{u}}}(V_\infty)\to\tT_b(V_\infty),
\]
where $n_u=(0,\dots,n,\dots,0)$ denotes the element in $\dN^{\Hom(E,\dC)}$ with only nonzero entry $n$ at $u$ for $n\in\dN$.
On the other hand, for $1\leq i\leq b_u$, we have an operator
\[
\varsigma_i^u\colon \sS(V_{v(u)})\otimes_\dC\tT^{b_u-1}(V_u)\to\sS(V_{v(u)})\otimes_\dC\tT^{b_u}(V_u)
\]
such that its value on $x\in V_{v(u)}$ is the operator that inserts $u(x)$ at the $i$-th spot; we regard it as a map
\[
\varsigma_i^u\colon\sS(V_\infty)\otimes_\dC\tT_{b-1_u}(V_\infty)\to\sS(V_\infty)\otimes_\dC\tT_b(V_\infty).
\]

\begin{lem}\label{le:form_unitary_2}
Consider the case where $r=1$.
\begin{enumerate}
  \item For every $u\in\Hom(E,\dC)$ and every $1\leq i\leq b_u$, we have
      \[
      [\varphi_{1,b}]=[\varsigma_i^u\varphi_{1,b-1_u}]+
      \frac{1}{2\pi}\sum_{j=1}^{b_{\ol{u}}}[A_{ij}^u\varphi_{1,b-1_u-1_{\ol{u}}}].
      \]

  \item For every $u\in\Hom(E,\dC)$, we have
      \[
      [\omega_{\psi_{F,\infty}}(\rL_{v(u)})\varphi_{1,b}]=
      \frac{1}{4\pi}\(\sum_{i=1}^{b_u}\sum_{j=1}^{b_{\ol{u}}}[A_{ij}^u\varphi_{1,b-1_u-1_{\ol{u}}}]
      +\sum_{i=1}^{b_{\ol{u}}}\sum_{j=1}^{b_u}[A_{ij}^{\ol{u}}\varphi_{1,b-1_u-1_{\ol{u}}}]\),
      \]
      where $\rL_v$ denotes the lowering operator $\frac{1}{2}\(\begin{smallmatrix} 1& -\ti \\ -\ti & -1 \end{smallmatrix}\)$ in $\fg_{1,v}$ for $v\in\Hom(F,\dR)$.
\end{enumerate}
\end{lem}

\begin{proof}
We use the pullback trick in Remark \ref{re:pullback} to reduce the statements to the known ones in the orthogonal case. Recall that $b'_{v(u)}=b_u+b_{\ul{u}}$. In the projection map $p_u\colon (\tT^{b_{v(u)}}V'_v)_\dC\to\tT^{b_u}V_u\otimes_\dC\tT^{b_{\ol{u}}}V_{\ol{u}}$, we label the first $b_u$ spots for the factor $\tT^{b_u}V_r$ and the latter $b_{\ul{u}}$ spots for the factor $\tT^{b_{\ol{u}}}V_{\ol{u}}$. Denote by $A_{jk}^{\prime v}$ and $\varsigma_j^{\prime v}$ the corresponding operators in the quadratic case from Remark \ref{re:form_orthogonal}. It is easy to see that for $1\leq i\leq d_u$, we have
\begin{align}
p_u\circ A^{\prime v(u)}_{ik}=
\begin{dcases}\label{eq:form_unitary}
2 \cdot A^u_{i(k+1-d_u)}\circ p_u,& d_u\leq k<d_u+d_{\ol{u}},\\
0 ,& 1\leq k<d_u.
\end{dcases}
\end{align}

For (1), it is clear that for $1\leq i\leq d_u$, $p_u\circ \varsigma_i^{\prime v(u)}=\varsigma_i^u \circ p_u$. Thus, by \eqref{eq:form_unitary} and the commutative diagram in Remark \ref{re:pullback}, it reduces to \cite{FM06}*{Theorem~5.9(i)} (and also Remark \ref{re:form_orthogonal}).

For (2), by \eqref{eq:form_unitary} (and its counterpart for $\ol{u}$) and the commutative diagram in Remark \ref{re:pullback}, it reduces to \cite{FM06}*{Theorem~5.10(i)}  (and also Remark \ref{re:form_orthogonal}).
\end{proof}

\section{Relative Hodge structure}
\label{ss:4}

In this section, we review the notion of relative Hodge structure with both real and complex coefficients.

\begin{notation}
Let $X$ be a smooth complex manifold.
\begin{enumerate}
  \item Denote by $\ol{X}$ the complex conjugation of $X$, which is again a smooth complex manifold.

  \item For a sheaf $V$ of complex vector spaces over $X$, we put $\ol{V}\coloneqq V\otimes_{\dC,\ol{\phantom{a}}}\dC$.

  \item Denote by $\sA_X^d$ and $\sA_X^{r,s}$ the sheaf of complex valued smooth $d$-forms and $(r,s)$-forms, respectively, on the underlying smooth manifold of $X$.

  \item For a coherent sheaf $\sV$ on $X$, we denote by $\sV^\infty$ the underlying complex smooth sheaf (on the underlying smooth manifold of $X$).\footnote{In particular, $\sO_X^\infty=\sA_X^0$ and more generally $(\Omega_X^r)^\infty=\sA_X^{r,0}$ for $r>0$.} Note that $\ol{\sV}$ is canonically a coherent sheaf on $\ol{X}$, and $\ol{\sV}^\infty$ can be canonically identified with $\ol{\sV^\infty}$ once we identify the underlying smooth manifolds of $X$ and $\ol{X}$.

  \item Let $\sV$ be a coherent sheaf on $X$ and $\cW$ a complex smooth subsheaf of $\sV^\infty$. We say that $\cW$ is \emph{holomorphic} if there exists a (unique) coherent subsheaf $\sW$ of $\sV$ on $X$ such that $\cW=\sW^\infty$. We say that $\cW$ is \emph{anti-holomorphic} if $\ol{\cW}$ is a holomorphic complex smooth subsheaf of $\ol{\sV^\infty}=\ol{\sV}^\infty$ (over $\ol{X}$).

  \item For a sheaf written like $\sF_X$ on $X$, we simply write $\sF(X)$ for $\sF_X(X)$.
\end{enumerate}
\end{notation}

We first recall the notion of relative Hodge structure (a.k.a. variation of Hodge structures).

\begin{definition}
Let $X$ be a smooth complex manifold and $w$ an integer.
\begin{enumerate}
  \item A \emph{pure (relative) real Hodge structure on $X$ of weight $w$} is a pair $\cW=(W,\tF^\bullet_\cW)$ in which
      \begin{itemize}
        \item $W$ is a real local system on $X$ of finite rank (and put $\sW\coloneqq \sO_X\otimes_\dR W$), and

        \item $\tF^\bullet_\cW$ is a decreasing filtration of $\sW^\infty$ by holomorphic complex smooth subsheaves,
      \end{itemize}
      such that $\tF^\bullet_\cW$ satisfies Griffiths transversality, and
        \[
        \sW^\infty=\bigoplus_{p+q=w}\cW^{(p,q)},
        \]
      where (after we identify $\sW^\infty$ with $\ol{\sW^\infty}$ using the real structure on $W$)
        \[
        \cW^{(p,q)}\coloneqq \tF^p_\cW\cap\ol{\tF^q_\cW}.
        \]

  \item Denote by $\sfM^w(X,\dR)$ the abelian category of pure real Hodge structures on $X$ of weight $w$.

  \item For an $\dR$-ring $\dL$, we denote by $\sfM^w_\dL(X,\dR)$ the category of $\dL$-linear objects in $\sfM^w(X,\dR)$.
\end{enumerate}
\end{definition}

\begin{definition}[\cite{SS}*{\S2}]
Let $X$ be a smooth complex manifold and $w$ an integer.
\begin{enumerate}
  \item A \emph{pure (relative) complex Hodge structure on $X$ of weight $w$} is a triple $\cV=(V,\tF^{\prime\bullet}_\cV,\tF^{\prime\prime\bullet}_\cV)$ in which
      \begin{itemize}
        \item $V$ is a complex local system on $X$ of finite rank (and put $\sV\coloneqq \sO_X\otimes_\dC V$),

        \item $\tF^{\prime\bullet}_\cV$ is a decreasing filtration of $\sV^\infty$ by holomorphic complex smooth subsheaves, and

        \item $\tF^{\prime\prime\bullet}_\cV$ is a decreasing filtration of $\sV^\infty$ by anti-holomorphic complex smooth subsheaves,
      \end{itemize}
      such that both $\tF^{\prime\bullet}_\cV$ and $\ol{\tF^{\prime\prime\bullet}_\cV}$ satisfy Griffiths transversality, and
        \[
        \sV^\infty=\bigoplus_{p+q=w}\cV^{(p,q)},
        \]
      where
        \[
        \cV^{(p,q)}\coloneqq \tF^{\prime p}_\cV\cap\tF^{\prime\prime q}_\cV.
        \]

  \item For a pure complex Hodge structure $\cV=(V,\tF^{\prime\bullet}_\cV,\tF^{\prime\prime\bullet}_\cV)$ on $X$ of weight $w$,
      \begin{itemize}
        \item we define its \emph{complex conjugation} to be
            \[
            \ol\cV\coloneqq(\ol{V},\ol{\tF^{\prime\prime\bullet}_\cV},\ol{\tF^{\prime\bullet}_\cV}),
            \]
            which is a pure complex Hodge structure on $\ol{X}$ of weight $w$;

        \item we define its \emph{dual} to be
            \[
            \cV^\vee\coloneqq(V^\vee,(\tF^{\prime -\bullet}_\cV)^\perp,(\tF^{\prime\prime -\bullet}_\cV)^\perp),
            \]
      \end{itemize}
            which is a pure complex Hodge structure on $X$ of weight $-w$.

  \item Denote by $\sfM^w(X,\dC)$ the abelian category of pure complex Hodge structures on $X$ of weight $w$.
\end{enumerate}
\end{definition}

\begin{remark}
We define a functor
\[
\pres\dC-\colon\sfM^w_\dC(X,\dR)\to\sfM^w(X,\dC)
\]
such that for $\cW=(W,\tF^\bullet_\cW)$, $\pres\dC\cW$ is the complex Hodge structure $\cV=(V,\tF^{\prime\bullet}_\cV,\tF^{\prime\prime\bullet}_\cV)$ in which $V=W$ but regarded as a complex local system via the $\dC$-action (so that $\sW=\sV\oplus\ol\sV$) and
\[
\tF^{\prime p}_\cV=\tF^p_\cW\cap\sV^\infty,\quad \tF^{\prime\prime q}_\cV=\ol{\tF^q_\cW\cap\ol\sV^\infty}.
\]
We define another functor
\[
\pres\dR-\colon\sfM^w(X,\dC)\to\sfM^w_\dC(X,\dR)
\]
such that for $\cV=(V,\tF^{\prime\bullet}_\cV,\tF^{\prime\prime\bullet}_\cV)$, $\pres\dR\cV$ is the real Hodge structure $\cW=(W,\tF^\bullet_\cW)$ in which $W=V$ but regarded as a real local system (so that $\sW=\sV\oplus\ol\sV$) and
\[
\tF^p_\cW=\tF^{\prime p}_\cV\oplus\ol{\tF^{\prime\prime p}_\cV}.
\]
It is clear that the functors $\pres\dC-$ and $\pres\dR-$ are inverse to each other hence induce a complex linear equivalence between complex linear categories $\sfM^w_\dC(X,\dR)$ and $\sfM^w(X,\dC)$.
\end{remark}

\begin{definition}\label{de:polarization}
Let $\cV=(V,\tF^{\prime\bullet}_\cV,\tF^{\prime\prime\bullet}_\cV)$ be a complex Hodge structure of weight zero on $X$. We say that a flat hermitian pairing $(\;,\;)_V\colon V\times V\to\dC$ is a \emph{polarization} of complex Hodge structure if, after writing $\pres\dR\cV$ as $\cW=(W,\tF^\bullet_\cW)$, the flat symmetric pairing $(\;,\;)_W\colon W\times W\to\dC$ defined as the real trace of $(\;,\;)_V$ is a polarization of the real Hodge structure $\cW$ (of weight zero).
\end{definition}

We collect some facts about the Hodge decomposition with coefficients in a Hodge structure. In the real case, we follow Deligne in an unpublished notes (see \cite{Zuc79}*{\S 1~\&~\S 2} for an account). Let $\cW=(W,\tF^\bullet_\cW)$ be a real Hodge structure over $X$. Recall that the Gauss--Manin connection $\nabla\colon\sW\to\Omega_X^1\otimes_{\sO_X}\sW$ is defined as $\partial\otimes 1$ on $\sW=\sO_X\otimes_\dR W$, from which we obtain a complex $\Omega_X^\bullet(W)\coloneqq\Omega_X^\bullet\otimes_\dR W$. It extends to an operator $\rD$ on $\sA_X^\bullet(W)\coloneqq\sA_X^\bullet\otimes_\dR W$ by taking $\rD\coloneqq\rd\otimes 1$.

The Griffith transversality implies that
\begin{align*}
\rD(\sA^{r,s}_X\otimes_{\sO_X^\infty}\tF^p_\cW)&\subseteq
\sA^{r+1,s}_X\otimes_{\sO_X^\infty}\tF^{p-1}_\cW\oplus\sA^{r,s+1}_X\otimes_{\sO_X^\infty}\tF^p_\cW, \\
\rD(\sA^{r,s}_X\otimes_{\sO_X^\infty}\ol{\tF^q_\cW})&\subseteq
\sA^{r,s+1}_X\otimes_{\sO_X^\infty}\ol{\tF^{q-1}_\cW}\oplus\sA^{r+1,s}_X\otimes_{\sO_X^\infty}\ol{\tF^q_\cW}.
\end{align*}
Combining them, we obtain
\[
\rD(\sA^{r,s}_X\otimes_{\sO_X^\infty}\cW^{(p,q)})\subseteq\sA^{r+1,s}_X\otimes_{\sO_X^\infty}(\cW^{(p,q)}\oplus\cW^{(p-1,q+1)})
\oplus\sA^{r,s+1}_X\otimes_{\sO_X^\infty}(\cW^{(p,q)}\oplus\cW^{(p+1,q-1)}).
\]
In other words, the \emph{real Gauss--Manin operator} $\rD$ splits into four components:
\begin{align*}
\partial'&\colon\sA^{r,s}_X\otimes_{\sO_X^\infty}\cW^{(p,q)}\to\sA^{r+1,s}_X\otimes_{\sO_X^\infty}\cW^{(p,q)},\\
\ol\partial'&\colon\sA^{r,s}_X\otimes_{\sO_X^\infty}\cW^{(p,q)}\to\sA^{r,s+1}_X\otimes_{\sO_X^\infty}\cW^{(p,q)},\\
\nabla'&\colon\sA^{r,s}_X\otimes_{\sO_X^\infty}\cW^{(p,q)}\to\sA^{r+1,s}_X\otimes_{\sO_X^\infty}\cW^{(p-1,q+1)},\\
\ol\nabla'&\colon\sA^{r,s}_X\otimes_{\sO_X^\infty}\cW^{(p,q)}\to\sA^{r,s+1}_X\otimes_{\sO_X^\infty}\cW^{(p+1,q-1)}.
\end{align*}

\begin{notation}\label{no:gm_real}
Put $\rD'\coloneqq\partial'+\ol\nabla'$, $\rD''\coloneqq\ol\partial'+\nabla'$, and $\rD^\tc\coloneqq(4\pi\ti)^{-1}(\rD'-\rD'')$ on $\sA_X^\bullet(W)$.
\end{notation}

In the complex case, one can easily derive the theory from the real one. Let $\cV=(V,\tF^{\prime\bullet}_\cV,\tF^{\prime\prime\bullet}_\cV)$ be a complex Hodge structure over $X$. Again we have the Gauss--Manin connection $\nabla\colon\sV\to\Omega_X^1\otimes_{\sO_X}\sV$ defined as $\partial\otimes 1$ on $\sV=\sO_X\otimes_\dC V$, from which we obtain a complex $\Omega_X^\bullet(V)\coloneqq\Omega_X^\bullet\otimes_\dC V$. It extends to an operator $\rD$ on $\sA_X^\bullet(V)\coloneqq\sA_X^\bullet\otimes_\dC V$ by taking $\rD\coloneqq\rd\otimes 1$. Write $\cW=(W,\tF^\bullet_\cW)$ for $\pres\dR\cV$. Then the operator $\pres\dR\rD$ on $\sA_X^\bullet(W)=\sA_X^\bullet\otimes_\dR W=\sA_X^\bullet\otimes_\dC V\oplus\sA_X^\bullet\otimes_\dC\ol{V}$ in the real case equals $\rD\oplus\ol\rD$. From this and the fact that $\cW^{(p,q)}=\cV^{(p,q)}\oplus\ol{\cV^{(q,p)}}$, it is easy to derive that the \emph{complex Gauss--Manin operator} $\rD$ splits into four components:
\begin{align*}
\partial'&\colon\sA^{r,s}_X\otimes_{\sO_X^\infty}\cV^{(p,q)}\to\sA^{r+1,s}_X\otimes_{\sO_X^\infty}\cV^{(p,q)},\\
\ol\partial'&\colon\sA^{r,s}_X\otimes_{\sO_X^\infty}\cV^{(p,q)}\to\sA^{r,s+1}_X\otimes_{\sO_X^\infty}\cV^{(p,q)},\\
\nabla'&\colon\sA^{r,s}_X\otimes_{\sO_X^\infty}\cV^{(p,q)}\to\sA^{r+1,s}_X\otimes_{\sO_X^\infty}\cV^{(p-1,q+1)},\\
\ol\nabla'&\colon\sA^{r,s}_X\otimes_{\sO_X^\infty}\cV^{(p,q)}\to\sA^{r,s+1}_X\otimes_{\sO_X^\infty}\cV^{(p+1,q-1)},
\end{align*}
such that $\pres\dR\partial$ (the corresponding component for $\pres\dR\rD$) is the direct sum of $\partial'$ (for $(p,q)$) and $\ol{\partial'}$ (for $(q,p)$), and similarly for the other three.

\begin{notation}\label{no:gm_complex}
Put $\rD'\coloneqq\partial'+\ol\nabla'$, $\rD''\coloneqq\ol\partial'+\nabla'$, and $\rD^\tc\coloneqq(4\pi\ti)^{-1}(\rD'-\rD'')$ on $\sA_X^\bullet(V)$.
\end{notation}

\section{Archimedean local height pairing over complex curves with coefficients}
\label{ss:5}

In this section, we review the notion of archimedean local height pairing over complex curves with coefficients, following \cite{Bry89}. Let $X$ be a one-dimensional complex manifold.

We start from the case of real coefficients. Consider a pure real Hodge structure $\cW=(W,\tF^\bullet_\cW)$ on $X$ of weight $0$ (that is, an object in $\sfM^0(X,\dR)$), equipped with a flat symmetric pairing $(\;,\;)_W\colon W\times W\to \dR$ that is a polarization of real Hodge structure. Put
\[
\Div(X,\cW)\coloneqq\bigoplus_{x\in X} W_x\cap\cW^{(0,0)}_x
\]
as a real vector space. For every element $\sfv\in\Div(X,\cW)$, we have its support $|\sfv|$ which is a (possibly empty) finite subset of $X$.

\begin{notation}\label{no:green_real}
We introduce the following complex vector spaces.
\begin{itemize}
  \item Put $\sA^{1,1}(X,W)\coloneqq\sA^{1,1}(X)\otimes_\dR W$.

  \item Denote by $\sD(X)$ the space of complex $0$-currents on $X$, and put $\sD(X,W)\coloneqq\sD(X)\otimes_\dR W$.

  \item Denote by $\sG(X)$ the space of complex smooth functions on $X$ away from a finite set that are locally integrable on $X$, and put $\sG(X,W)\coloneqq\sG(X)\otimes_\dR W$.
\end{itemize}
In particular, for $g\in\sG(X,W)$, $\rD\rD^\tc g$ is naturally an element of $\sD(X,W)$ via the pairing $(\;,\;)_W$.
\end{notation}

The polarization induces maps
\[
[-]\colon\sA^{1,1}(X,W)\to\sD(X,W)
\]
by integration, and
\[
\delta_-\colon\Div(X,\cW)\to\sD(X,W)
\]
by evaluation.

\begin{definition}\label{de:green_real}
Take an element $\sfv\in\Div(X,\cW)$. A \emph{Green function} for $\sfv$ is an element $g_\sfv\in\sG(X,W)$ defined away from $|\sfv|$ such that
\[
\rD\rD^\tc g_\sfv+\delta_\sfv=[\omega_\sfv]
\]
(Notation \ref{no:gm_real}) holds in $\sD(X,W)$ for a (necessarily unique) element $\omega_\sfv\in\sA^{1,1}(X,W)$, called the \emph{tail form} of $g_\sfv$.
\end{definition}

\begin{remark}
In Definition \ref{de:green_real}, it is clear that the cohomology class of the tail form $\omega_\sfv$ coincides with the image of $\sfv$ under the (geometric) cycle class map $\Div(X,\cW)\to\rH^2(X,W(1))$.
\end{remark}

\begin{notation}
Denote by
\begin{itemize}
  \item $\Div_\sharp(X,\cW)$ the real vector space of pairs $(\sfv,g_\sfv)$ with $\sfv\in\Div(X,\cW)$ and $g_\sfv$ a Green function for $\sfv$;

  \item $\Div_\sharp(X,\cW)^\heartsuit$ the real subspace of $\Div_\sharp(X,\cW)$ consisting of pairs $(\sfv,g_\sfv^\heartsuit)$ in which $\sfv\in\Div(X,\cW)$ and the tail form of $g_\sfv^\heartsuit$ vanishes;

  \item $\Div^\Box(X,\cW)$ the real linear subcone of $\Div(X,\cW)\times\Div(X,\cW)$ consisting of pairs $(\sfv_1,\sfv_2)$ satisfying $|\sfv_1|\cap|\sfv_2|=\emptyset$;

  \item $\Div^\Box_\sharp(X,\cW)$ the preimage of $\Div^\Box(X,\cW)$ under the natural forgetful map
      \[
      \Div_\sharp(X,\cW)\times\Div_\sharp(X,\cW)\to\Div(X,\cW)\times\Div(X,\cW);
      \]

  \item $\Div(X,\cW)^0$ the kernel of the cycle class map $\Div(X,\cW)\to\rH^2(X,W(1))$, and $\Div^\Box(X,\cW)^0$ the preimage of $\Div^\Box(X,\cW)$ under the natural inclusion map
      \[
      \Div(X,\cW)^0\times\Div(X,\cW)^0\to\Div(X,\cW)\times\Div(X,\cW).
      \]
\end{itemize}
Since $g_\sfv$ determines both $\sfv$ and $\omega_\sfv$, in what follows, we will simply write the component of Green function for an element in $\Div_\sharp(X,\cW)$.
\end{notation}

Similar to the case of constant coefficients, we have a star product pairing
\[
\ast\colon\Div^\Box_\sharp(X,\cW)\to\sD(X)
\]
defined by the formula
\[
g_{\sfv_1}\ast g_{\sfv_2}=\tr_W\(g_{\sfv_1}\wedge\delta_{\sfv_2}+\omega_{\sfv_1}\wedge g_{\sfv_2}\)
\]
where $\tr_W\colon W\otimes_\dR W\to\dR$ is the real linear map induced by $(\;,\;)_W$. It is easy to see that $\ast$ is real linear in both variables and symmetric.

\begin{definition}\label{de:canonical_real}
Suppose that $X$ is proper. We define the \emph{canonical (archimedean) height pairing} to be the map
\[
\langle\;,\;\rangle_X\colon\Div^\Box(X,\cW)^0\to\dR
\]
such that for $(\sfv_1,\sfv_2)\in\Div^\Box(X,\cW)^0$,
\[
\langle \sfv_1,\sfv_2\rangle_X\coloneqq\frac{1}{2}\int_X g_{\sfv_1}^\heartsuit\ast\delta_{\sfv_2},
\]
where $(\sfv_i,g_{\sfv_i}^\heartsuit)\in\Div_\sharp(X,\cW)^\heartsuit$ is an arbitrary lift of $\sfv_i$. Indeed, the lift exists by \cite{Bry89}*{Proposition~2.9}, and it is clear that $\langle \sfv_1,\sfv_2\rangle_X$ does not depend on the choice of the lift.
\end{definition}

\begin{remark}\label{re:beilinson_real}
The canonical height pairing in the above definition has an alternative interpretation using absolute Hodge cohomology (denoted by $\rH^\bullet_\hs$) similar to Beilinson's index in \cite{Bei87}. Take a pair $(\sfv_1,\sfv_2)\in\Div^\Box(X,\cW)^0$; for $i=1,2$, put $U_i\coloneqq X\setminus|\sfv_i|$, write $\gamma_i$ for the absolute cycle class of $\sfv_i$ in $\rH^2_{\hs,|\sfv_i|}(X,\cW(1))$ whose image in $\rH^2_\hs(X,\cW(1))$ must vanish by \cite{Bry89}*{Lemma~2.2}, take an element $\tilde\gamma_i\in\rH^1_\hs(U_i,\cW(1))$ that maps to $\gamma_i$ under the co-boundary map $\rH^1_\hs(U_i,\cW(1))\to\rH^2_{\hs,|\sfv_i|}(X,\cW(1))$; then $\langle \sfv_1,\sfv_2\rangle$ equals to the image of $\tr_W\(\tilde\gamma_1\cup\tilde\gamma_2\)\in\rH^2_\hs(U_1\cap U_2,\dR(2))$ under the composite map
\[
\rH^2_\hs(U_1\cap U_2,\dR(2))\to\rH^3_\hs(X,\dR(2))\xrightarrow{\Tr_{X/\bullet}}\rH^1_\hs(\bullet,\dR(1))=\dR,
\]
where $\bullet$ denotes a single point. Indeed, the above image equals the negative of the height pairing in \cite{Bry89}*{\S2}, which then equals our pairing $\langle \sfv_1,\sfv_2\rangle$ by \cite{Bry89}*{Theorem~2.8} since the function $h$ there satisfies that $(\sfv_1,-2h)\in\Div_\sharp(X,\cW)^\heartsuit$ by the Poincar\'e--Lelong formula.
\end{remark}

\begin{remark}\label{re:height_sesqui}
Put $\Div_\sharp(X,\cW)_\dC\coloneqq\Div_\sharp(X,\cW)\otimes_\dR\dC$ and via the similar process we obtain $\Div^\Box_\sharp(X,\cW)_\dC$ and $\Div^\Box(X,\cW)^0_\dC$. We extend $\ast$ and $\langle\;,\;\rangle_X$ to maps
\[
\ast\colon\Div^\Box_\sharp(X,\cW)_\dC\to\sD(X),\qquad
\langle\;,\;\rangle_X\colon\Div^\Box(X,\cW)^0_\dC\to\dC,
\]
respectively, that are complex linear in the first variable and complex conjugate-linear in the second, and conjugate-symmetric.
\end{remark}

We now move to the case of complex coefficients. Consider a pure complex Hodge structure $\cV=(V,\tF^{\prime\bullet}_\cV,\tF^{\prime\prime\bullet}_\cV)$ on $X$ of weight $0$ (that is, an object in $\sfM^0(X,\dC)$), equipped with a flat hermitian pairing $(\;,\;)_V\colon V\times V\to \dC$ that is a polarization of complex Hodge structure (Definition \ref{de:polarization}). Put
\[
\Div(X,\cV)\coloneqq\bigoplus_{x\in X} V_x\cap\cV^{(0,0)}_x
\]
as a complex vector space. For every element $\sfv\in\Div(X,\cV)$, we have its support $|\sfv|$ which is a (possibly empty) finite subset of $X$.

\begin{notation}\label{no:green_complex}
We introduce the following complex vector spaces.
\begin{itemize}
  \item Put $\sA^{1,1}(X,V)\coloneqq\sA^{1,1}(X)\otimes_\dC V$.

  \item Put $\sD(X,V)\coloneqq\sD(X)\otimes_\dC V$ and $\sG(X,V)\coloneqq\sG(X)\otimes_\dC V$.
\end{itemize}
In particular, for $g\in\sG(X,V)$, $\rD\rD^\tc g$ is naturally an element of $\sD(X,\ol{V})$ via the pairing $(\;,\;)_V$.
\end{notation}

The polarization induces maps
\[
[-]\colon\sA^{1,1}(X,V)\to\sD(X,\ol{V})
\]
by integration, and
\[
\delta_-\colon\Div(X,\cV)\to\sD(X,\ol{V})
\]
by evaluation.

\begin{definition}\label{de:green_complex}
Take an element $\sfv\in\Div(X,\cV)$. A \emph{Green function} for $\sfv$ is an element $g_\sfv\in\sG(X,V)$ defined away from $|\sfv|$ such that
\[
\rD\rD^\tc g_\sfv+\delta_\sfv=[\omega_\sfv]
\]
(Notation \ref{no:gm_complex}) holds in $\sD(X,\ol{V})$ for a (necessarily unique) element $\omega_\sfv\in\sA^{1,1}(X,V)$, called the \emph{tail form} of $g_\sfv$.
\end{definition}

\begin{remark}
In Definition \ref{de:green_complex}, it is clear that the cohomology class of the tail form $\omega_\sfv$ coincides with the image of $\sfv$ under the (geometric) cycle class map $\Div(X,\cV)\to\rH^2(X,V(1))$.
\end{remark}

\begin{notation}
Denote by
\begin{itemize}
  \item $\Div_\sharp(X,\cV)$ the complex vector space of pairs $(\sfv,g_\sfv)$ with $\sfv\in\Div(X,\cV)$ and $g_\sfv$ a Green function for $\sfv$;

  \item $\Div_\sharp(X,\cV)^\heartsuit$ the complex subspace of $\Div_\sharp(X,\cV)$ consisting of pairs $(\sfv,g_\sfv^\heartsuit)$ in which $\sfv\in\Div(X,\cV)$ and the tail form of $g_\sfv^\heartsuit$ vanishes;

  \item $\Div^\Box(X,\cV)$ the complex linear subcone of $\Div(X,\cV)\times\Div(X,\cV)$ consisting of pairs $(\sfv_1,\sfv_2)$ satisfying $|\sfv_1|\cap|\sfv_2|=\emptyset$;

  \item $\Div^\Box_\sharp(X,\cV)$ the preimage of $\Div^\Box(X,\cV)$ under the natural forgetful map
      \[
      \Div_\sharp(X,\cV)\times\Div_\sharp(X,\cV)\to\Div(X,\cV)\times\Div(X,\cV);
      \]

  \item $\Div(X,\cV)^0$ the kernel of the cycle class map $\Div(X,\cV)\to\rH^2(X,V(1))$, and $\Div^\Box(X,\cV)^0$ the preimage of $\Div^\Box(X,\cV)$ under the natural inclusion map
      \[
      \Div(X,\cV)^0\times\Div(X,\cV)^0\to\Div(X,\cV)\times\Div(X,\cV).
      \]
\end{itemize}
Since $g_\sfv$ determines both $\sfv$ and $\omega_\sfv$, in what follows, we will simply write the component of Green function for an element in $\Div_\sharp(X,\cV)$.
\end{notation}

Similar to the case of real coefficients, we have a star product pairing
\[
\ast\colon\Div^\Box_\sharp(X,\cV)\to\sD(X)
\]
defined by the formula
\[
g_{\sfv_1}\ast g_{\sfv_2}=\tr_V\(g_{\sfv_1}\wedge\delta_{\ol{\sfv_2}}+\omega_{\sfv_1}\wedge\ol{g_{\sfv_2}}\)
\]
where $\tr_V\colon V\otimes_\dC\ol{V}\to\dC$ is the complex linear map induced by $(\;,\;)_V$.

\begin{definition}\label{de:canonical_complex}
Suppose that $X$ is proper. We define the \emph{canonical (archimedean) height pairing} to be the map
\[
\langle\;,\;\rangle_X\colon\Div^\Box(X,\cV)^0\to\dC
\]
such that for $(\sfv_1,\sfv_2)\in\Div^\Box(X,\cV)^0$,
\[
\langle \sfv_1,\sfv_2\rangle_X\coloneqq\frac{1}{2}\int_X g_{\sfv_1}^\heartsuit\ast\delta_{\ol{\sfv_2}},
\]
where $(\sfv_1,g_{\sfv_1}^\heartsuit)\in\Div_\sharp(X,\cV)^\heartsuit$ is an arbitrary lift of $\sfv_1$. It is easy to see that $\langle\;,\;\rangle_X$ is complex linear in the first variable and complex conjugate-linear in the second, and conjugate-symmetric.
\end{definition}

\begin{remark}\label{re:beilinson_complex}
The canonical height pairing in the above definition has an alternative interpretation using absolute Hodge cohomology similar to the case of real coefficients (Remark \ref{re:beilinson_real}). We leave the detail to the readers.
\end{remark}

\section{Archimedean arithmetic Siegel--Weil formula over orthogonal Shimura curves}
\label{ss:6}

We continue the discussion in \S\ref{ss:2} but with $m=3$ and $r=1$. In particular, $\lambda_v=(b_v,0,-b_v)$ for every $v\in\Hom(F,\dR)$.

\begin{definition}\label{de:hodge_orthogonal}
We define real Hodge structures on the local systems $\tT_b(V_\infty)$ and $\tS_{[\lambda]}(V_\infty)$ over $\fD$, respectively.
\begin{itemize}
  \item For $v\in\Hom(F,\dR)\setminus\{v_0\}$, define $\cW_v$ to be the unique real Hodge structure on $\sV_v\coloneqq V_v\otimes_\dR\sO_\fD$ satisfying that $\cW_v^{(0,0)}=\sV_v^\infty$.

  \item For $v=v_0$, denote by $V_\fD^-$ the tautological real subbundle of negative two planes of $V_v\otimes_\dR\sA^\dR_\fD$ over (the underlying smooth manifold of) $\fD$ and by $V_\fD^+$ its orthogonal complement, where $\sA^\dR_\fD$ denotes the sheaf of \emph{real} smooth functions on $\fD$. The the universal orientation on $V_v\otimes_\dR\sA^\dR_\fD$ gives rise to a $\dC$-linear structure on $V_\fD^-$. Define $\cW_v$ to be the unique real Hodge structure on $\sV_v\coloneqq V_v\otimes_\dR\sO_\fD$ satisfying that $\cW_v^{(0,0)}=V_\fD^+\otimes_\dR\dC$ and that $\cW_v^{(-1,1)}$ (resp.\ $\cW_v^{(1,-1)}$) is the subsheaves of $V_\fD^-\otimes_\dR\dC$ on which the two $\dC$-linear actions are the same (resp.\ conjugate). It is easy to see that the quadratic form on $V_v$ induces a polarization of Hodge structure on $\cW_v$.\footnote{Indeed, $\cW_v$ is the unique real Hodge structure of weight zero on $\sV_v$ of which the quadratic form on $V_v$ induces a polarization, satisfying $\cW_v^{(p,-p)}=0$ for $|p|\geq 2$.}
\end{itemize}
By taking tensor product of $\cW_v$ and further applying \eqref{eq:harmonic_orthogonal} and \eqref{eq:schur_orthogonal}, we obtain real Hodge structures $\cW_b$ and $\cW_\lambda$ on $\tT_b(V_\infty)$ and $\tS_{[\lambda]}(V_\infty)$ over $\fD$, respectively. These real Hodge structures are all $H(F_\infty)$-invariant and descend to the orthogonal Shimura curve
\[
X_L=H(F)\backslash\(\fD\times H(\bA_F^\infty)/L\).
\]
In particular, the underlying real local system of $\cW_\lambda$ is $\tS_{[\lambda]}(V_\infty)$.
\end{definition}

Fix a base point of $\fD$ and denote its stabilizer in $H(F_{v_0})$ by $L_{v_0}$ (so that $\fD\simeq H(F_{v_0})/L_{v_0}$). The base point gives rise to an orthogonal decomposition $V_{v_0}=V_{v_0}^+\oplus V_{v_0}^-$ in which $V_{v_0}^+$ is positive definite of dimension one and $V_{v_0}^-$ is negative definite of dimension two together with an orientation. Put $L_\infty\coloneqq L_{v_0}\times\prod_{v\neq v_0}H(F_v)$ as a subgroup of $H(F_\infty)$. Denote by $\fh$ the Lie algebra of $H(F_{v_0})$, which has the Harich-Chandra decomposition $\fh=\fl\oplus\fp$ where $\fl$ is the Lie algebra of $L_{v_0}$. Choose an orthonormal basis $\{e_{v,1},e_{v,2},e_{v,3}\}$ of $V_v$ for every $v\in\Hom(F,\dR)$ so that for $v=v_0$, $e_{v_0,1}\in V_{v_0}^+$ and $e_{v_0,2},e_{v_0,3}\in V_{v_0}^-$ such that $e_{v_0,2}\wedge e_{v_0,3}$ gives the orientation. Write $\{X_{12},X_{13}\}$ for the standard basis of $\fp$ (see \cite{FM06}*{Page~918}) and $\{\omega_{12},\omega_{13}\}$ its dual basis. Note that
\[
\Omega\coloneqq\frac{1}{2\pi}\cdot\omega_{12}\wedge\omega_{13}\in\wedge^2\fp^\vee
\]
is independent of the choice of the basis (of $V_{v_0}$).

The basis $\{e_{v,1},e_{v,2},e_{v,3}\}$ gives rise to a coordinate system $x=(x_{v,1},x_{v,2},x_{v,3})_v\in(\dR^3)^{\Hom(F,\dR)}$.
Define the Gaussian function $\nu_0\in\sS(V_\infty)$ by the formula
\[
\nu_0(x)\coloneqq\prod_{v\in\Hom(F,\dR)}\te^{-\pi(x_{v,1}^2+x_{v,2}^2+x_{v,3}^2)}.
\]
For the Kudla--Millson form $\varphi_0\coloneqq\varphi_{1,0}$, we have
\[
\varphi_0(x)=(4\pi x_{v_0,1}^2-1)\nu_0(x)\cdot\Omega
\]
(in particular, $\varphi_0(0)=-\Omega$). Put
\[
\fA\coloneqq\End_\dC\(\sS(V_\infty)\otimes_\dR\wedge^\bullet\fp^\vee\otimes_\dR\tT(V_\infty)\),
\]
in which
\[
\tT(V_\infty)\coloneqq\bigoplus_{c\in\dN^{\Hom(F,\dR)}}\tT_c(V_\infty).
\]
We have operators
\[
\rE_v\coloneqq
\begin{dcases}
\frac{1}{2}\(x_{v,1}-\frac{1}{2\pi}\frac{\partial}{\partial x_{v,1}}\)\otimes 1\otimes L(e_{v,1}),&\text{if $v=v_0$},\\
\frac{1}{2}\sum_{j=1}^3\(x_{v,j}-\frac{1}{2\pi}\frac{\partial}{\partial x_{v,j}}\)\otimes 1\otimes L(e_{v,j}),&\text{if $v\neq v_0$}
\end{dcases}
\]
in $\fA$ for $v\in\Hom(F,\dR)$, where $L(e_{v,j})$ denotes the left multiplication by $e_{v,j}$ in $\tT(V_\infty)$.

Then by \cite{FM06}*{Definition~5.1},
\begin{align*}
\varphi_b\coloneqq\varphi_{1,b}=\(\prod_{v\in\Hom(F,\dR)}\rE_v^{b_v}\)\varphi_0
&\in\(\sS(V_\infty)\otimes_\dR\wedge^2\fp^\vee\otimes_\dR\tT_b(V_\infty)\)^{L_\infty} \\
&=\(\sS(V_\infty)\otimes_\dC\sA^2(\fD)\otimes_\dR\tT_b(V_\infty)\)^{H(F_\infty)}.
\end{align*}
Applying the operator \eqref{eq:harmonic_orthogonal} to $\varphi_b$, we obtain
\[
[\varphi_b]\in\(\sS(V_\infty)\otimes_\dC\sA^2(\fD)\otimes_\dR\tT_{[b]}(V_\infty)\)^{H(F_\infty)}.
\]
Since $m=3$, $[\varphi_b]$ already belongs to $\sS(V_\infty)\otimes_\dC\sA^2(\fD)\otimes_\dR\tS_{[\lambda]}(V_\infty)$, so that
\[
\varphi_\lambda\coloneqq\pi_{\lambda^+}([\varphi_b])=[\varphi_b].
\]
This special feature when $m=3$ will appear repeatedly later.

On the other hand, we introduce
\begin{align*}
\nu_b\coloneqq\(\prod_{v\in\Hom(F,\dR)}\rE_v^{b_v}\)\nu_0
\in\(\sS(V_\infty)\otimes_\dR\tT_b(V_\infty)\)^{L_\infty}
=\(\sS(V_\infty)\otimes_\dC\sA^0(\fD)\otimes_\dR\tT_b(V_\infty)\)^{H(F_\infty)},
\end{align*}
and similarly
\[
\nu_\lambda\coloneqq[\nu_b]\in\(\sS(V_\infty)\otimes_\dC\sA^0(\fD)\otimes_\dR\tS_{[\lambda]}(V_\infty)\)^{H(F_\infty)}.
\]

The lemma below is an analogue of \cite{BF04}*{Theorem~4.4} in the case of general weights (but only for $p=1$).

\begin{lem}\label{le:holomorphy_orthogonal}
We have
\[
\omega_{\psi_{F,v_0}}(\rL_{v_0})\varphi_b=-\rD\rD^\tc\nu_b+
\frac{1}{8\pi}\sum_{j=1}^{b_{v_0}}\sum_{k=1}^{b_{v_0}-1}A_{jk}^{v_0}\varphi_{b-2_{v_0}}.
\]
In particular,
\[
\omega_{\psi_{F,v_0}}(\rL_{v_0})\varphi_\lambda=-\rD\rD^\tc\nu_\lambda.
\]
\end{lem}

\begin{remark}\label{re:gauss_manin}
For the space $\(\sS(V_\infty)\otimes_\dR\wedge^\bullet\fp^\vee\otimes_\dR\tT(V_\infty)\)^{L_\infty}$, the four components of the Gauss--Manin operator are given by the following formulae:
\begin{align*}
\partial'&=\frac{1}{2}\cdot\omega_{\psi_{F,v_0}}(X_{12}-\ti X_{13})\otimes L(\omega_{12}+\ti \omega_{13})\otimes 1,\\
\ol\partial'&=\frac{1}{2}\cdot\omega_{\psi_{F,v_0}}(X_{12}+\ti X_{13})\otimes L(\omega_{12}-\ti \omega_{13})\otimes 1,\\
\nabla'&=\frac{1}{2}\cdot 1\otimes L(\omega_{12}+\ti \omega_{13})\otimes\rho(X_{12}-\ti X_{13}),\\
\ol\nabla'&=\frac{1}{2}\cdot 1\otimes L(\omega_{12}-\ti \omega_{13})\otimes\rho(X_{12}+\ti X_{13}),
\end{align*}
where $L(\omega)$ denotes the left multiplication by an element $\omega\in\fp_\dC^\vee$, and $\rho$ denotes the right induced action of $\fp$ on $\tT(V_\infty)$ (through the factor $\tT(V_{v_0})$).
\end{remark}

\begin{proof}
For the identity in the lemma, both sides have the same factor for places away from $v_0$. Thus, without loss of generality, we may assume $F=\dQ$ and suppress $v_0$ in all subscripts (while simply write $V$ for $V_\infty$). We will freely use Remark \ref{re:gauss_manin}.

By definition, we have
\[
-\rD\rD^\tc\nu_b=(2\pi\ti)^{-1}\(\partial'\ol\partial'\nu_b+\ol\nabla'\nabla'\nu_b\).
\]
For further computation, we use the Fock model $\dC[\sfz_1,\sfz_2,\sfz_3]$ of the Weil representation $\omega_{\psi_F}$ as in \cite{FM06}*{Appendix~A}. Under such model,
\begin{align*}
\nu_b&=\(\frac{-\ti}{4\pi}\)^b\sfz_1^b\otimes1\otimes e_1^{\otimes b}, \\
\varphi_b&=\frac{-1}{8\pi^2}\(\frac{-\ti}{4\pi}\)^b\sfz_1^{b+2}\otimes(\omega_{12}\wedge\omega_{13})\otimes e_1^{\otimes b},
\end{align*}
and by \cite{FM06}*{Lemma~A.2},
\begin{align*}
\partial'&=\(-2\pi\(\frac{\partial^2}{\partial\sfz_1\partial\sfz_2}-\ti\frac{\partial^2}{\partial\sfz_1\partial\sfz_3}\)
+\frac{1}{8\pi}(\sfz_1\sfz_2-\ti\sfz_1\sfz_3)\)\otimes L(\omega_{12}+\ti\omega_{13})\otimes 1, \\
\ol\partial'&=\(-2\pi\(\frac{\partial^2}{\partial\sfz_1\partial\sfz_2}+\ti\frac{\partial^2}{\partial\sfz_1\partial\sfz_3}\)
+\frac{1}{8\pi}(\sfz_1\sfz_2+\ti\sfz_1\sfz_3)\)\otimes L(\omega_{12}-\ti\omega_{13})\otimes 1.
\end{align*}
It follows that
\begin{align}\label{eq:holomorphy1}
\partial'\ol\partial'\nu_b=2\pi\(\frac{-\ti}{4\pi}\)^b
\((b+1)\ti\sfz_1^b-\frac{\ti}{32\pi^2}\sfz_1^{b+2}(\sfz_2^2+\sfz_3^2)\)\otimes\Omega\otimes e_1^{\otimes b}.
\end{align}
For $\ol\nabla'\nabla'\nu_b$, it equals
\[
\pi\ti\(\frac{-\ti}{4\pi}\)^b
\sfz_1^b\otimes\Omega\otimes\rho(X_{12}+\ti X_{13})\rho(X_{12}-\ti X_{13})e_1^{\otimes b}.
\]
An easy exercise in combinatorics shows that
\[
\rho(X_{12}+\ti X_{13})\rho(X_{12}-\ti X_{13})e_1^{\otimes b}
=b(b+1)e_1^{\otimes b}-\sum_{j=1}^{b}\sum_{k=1}^{b-1}A_{jk}e_1^{\otimes(b-2)}.
\]
It follows that
\begin{align}\label{eq:holomorphy2}
\ol\nabla'\nabla'\nu_b=2\pi\(\frac{-\ti}{4\pi}\)^b\frac{b(b+1)\ti}{2}\sfz_1^{b}\otimes\Omega\otimes e_1^{\otimes b}-\frac{\ti}{2}\(\frac{-1}{8\pi^2}\)^{-1}\(\frac{-\ti}{4\pi}\)^{2}\sum_{j=1}^{b}\sum_{k=1}^{b-1}A_{jk}\varphi_{b-2}.
\end{align}
Combining \eqref{eq:holomorphy1} and \eqref{eq:holomorphy2}, we have
\[
-\rD\rD^\tc\nu_b=2\pi\(\frac{-\ti}{4\pi}\)^b
\(\frac{(b+1)(b+2)}{4\pi}\sfz_1^b-\frac{\ti}{64\pi^3}\sfz_1^{b+2}(\sfz_2^2+\sfz_3^2)\)\otimes\Omega\otimes e_1^{\otimes b}-\frac{1}{8\pi}\sum_{j=1}^{b}\sum_{k=1}^{b-1}A_{jk}\varphi_{b-2}.
\]

On the other hand, by \cite{FM06}*{Lemma~A.1}, we have
\[
\omega_{\psi_F}(\rL)=-2\pi\frac{\partial^2}{\partial\sfz_1^2}+\frac{1}{8\pi}(\sfz_2^2+\sfz_3^2),
\]
which implies that
\[
\omega_{\psi_F}(\rL)\varphi_b=2\pi\(\frac{-\ti}{4\pi}\)^b
\(\frac{(b+1)(b+2)}{4\pi}\sfz_1^b-\frac{\ti}{64\pi^3}\sfz_1^{b+2}(\sfz_2^2+\sfz_3^2)\)
\otimes\Omega\otimes e_1^{\otimes b}.
\]

The lemma follows.
\end{proof}

Define a function $\mu\colon V_\infty\to\dR_{>0}$ by the formula
\[
\mu(x)\coloneqq\tq^{T(x)}(\ti)=\prod_{v\in\Hom(F,\dR)}\te^{-2\pi T(x_v)},
\]
and put
\[
\varphi_\lambda^\circ\coloneqq\varphi_\lambda\cdot\mu^{-1},\qquad
\nu_\lambda^\circ\coloneqq\nu_\lambda\cdot\mu^{-1}.
\]
Define an action $\star$ of $\dR_{>0}$ on $V_\infty$ via the formula $t\star x=(tx_{v_0},x_v,\dots)$.

\begin{lem}\label{le:nu_orthogonal}
For $x\in V_\infty$, we have
\[
\rD\rD^\tc\frac{\nu^\circ_\lambda(t^{1/2}\star x)}{t^{b_{v_0}/2}}=-t\frac{\rd}{\rd t}\frac{\varphi^\circ_\lambda(t^{1/2}\star x)}{t^{b_{v_0}/2}}
\]
for every $t\in\dR_{>0}$.
\end{lem}

\begin{proof}
Again without loss of generality, we may assume $F=\dQ$ and suppress $v_0$ in all subscripts (while simply write $V$ for $V_\infty$). In particular, $t^{1/2}\star x$ is simply $t^{1/2}x$. It suffices to prove the lemma for $t=1$, which we will deduce from Lemma \ref{le:holomorphy_orthogonal}.

For every element $\phi\in\sS(V)$ on which $\widetilde{K_1}$ acts by $\delta^{w/2}$ (for $w\in\dZ$), we have
\[
\omega_{\psi_F}(\rL)\phi(x)=\left.\(2\pi T(x)+\frac{\rd}{\rd t}-\frac{w}{4}\)\right|_{t=1}t^{3/4}\phi(t^{1/2}x).
\]
Applying this formula to $\varphi_\lambda$ (so that $w=3+2b$), we have for $x\in V$,
\begin{align*}
\omega_{\psi_F}(\rL)\varphi_\lambda(x)&=
\left.\(2\pi T(x)+\frac{\rd}{\rd t}-\frac{3+2b}{4}\)\right|_{t=1}t^{3/4}\varphi_\lambda(t^{1/2}x) \\
&=\left.\frac{\rd}{\rd t}\right|_{t=1}\(\frac{\varphi_\lambda^\circ(t^{1/2}x)}{t^{b/2}}\cdot t^{3/4+b/2}\mu(t^{1/2}x)\)+2\pi T(x)\varphi_\lambda(x)-\frac{3+2b}{4}\varphi_\lambda(x) \\
&=\(\left.\frac{\rd}{\rd t}\right|_{t=1}\frac{\varphi_\lambda^\circ(t^{1/2}x)}{t^{b/2}}\)\cdot\mu(x).
\end{align*}
Thus, the lemma follows from Lemma \ref{le:holomorphy_orthogonal} and the fact that $\mu$ is a constant function on $\fD$.
\end{proof}

Take an element $x\in V_\infty$ with $x_{v_0}\neq 0$. Denote by $\fD_x\subseteq\fD$ the subset of oriented $2$-planes perpendicular to $x_{v_0}$. If $T(x_{v_0})>0$ (resp.\ $T(x_{v_0})\leq 0$), then $\fD_x$ consists of two points (resp.\ is empty). Put
\[
\sfv^\circ_b(x)\coloneqq\(\fD_x,\otimes_{v\in\Hom(F,\dR)}x_v^{\otimes b_v}\)\in\Div(\fD,\cW_b)
\]
(see Definition \ref{de:hodge_orthogonal} for the Hodge structure) and consequently
\[
\sfv^\circ_\lambda(x)\coloneqq[\sfv^\circ_b(x)]\in\Div(\fD,\cW_\lambda)
\]
(so that $\sfv^\circ_\lambda(x)=0$ if $T(x_{v_0})\leq 0$). Define
\[
g_\lambda^\circ(x)\coloneqq\int_1^\infty\frac{\nu_\lambda^\circ(t^{1/2}\star x)}{t^{b_{v_0}/2}}\frac{\rd t}{t},
\]
which is absolutely convergent over $\fD\setminus\fD_x$ (after we regard $\nu_\lambda^\circ(x)$ as an element in $\sA^0(\fD)\otimes_\dR\tS_{[\lambda]}(V_\infty)$). Finally, put
\[
g_\lambda(x)\coloneqq \mu(x)\cdot g_\lambda^\circ(x),\qquad
\sfv_\lambda(x)\coloneqq \mu(x)\cdot \sfv_\lambda^\circ(x).
\]

\begin{lem}\label{le:green_orthogonal}
For every $x\in V_\infty$ with $x_{v_0}\neq 0$, $g_\lambda(x)$ is a smooth function on $\fD\setminus\fD_x$ with logarithmic growth along $\fD_x$ and rapid decay along the boundary of $\fD$, and satisfies
\[
\rD\rD^\tc g_\lambda(x) + \delta_{\sfv_\lambda(x)}=[\varphi_\lambda(x)].
\]
In particular, $g_\lambda(x)$ is a Green function for $\sfv_\lambda(x)$ with the tail form $\varphi_\lambda(x)$.
\end{lem}

\begin{proof}
Again without loss of generality, we may assume $F=\dQ$ and suppress $v_0$ in all subscripts. It suffices to show the lemma for $g_\lambda^\circ$, $\sfv_\lambda^\circ$ and $\varphi^\circ_\lambda$.

It is clear that $g_\lambda^\circ(x)$ is a smooth function on $\fD\setminus\fD_x$. For the growth along $\fD_x$, we may assume that $\fD_x$ consists of the base point and its conjugate, without loss of generality. We see that
\[
\frac{\nu_\lambda^\circ(t^{1/2}x)}{t^{b/2}}=\te^{-2\pi t(x_2^2+x_3^2)}\cdot [x^{\otimes b}]+O(t^{-1}),\quad t\mapsto\infty.
\]
Thus, $g_\lambda^\circ$ has the same asymptotic behavior as $g_0^\circ$ along $\fD_x$, multiplied by $[x^{\otimes b}]$. Then the logarithmic growth along $\fD_x$ follows from the well-known fact that $g_0^\circ$ has logarithmic growth along $\fD_x$. By Lemma \ref{le:nu_orthogonal}, we have
\[
\rD\rD^\tc g^\circ_\lambda(x) + \delta_{\sfv^\circ_\lambda(x)}=[\varphi^\circ_\lambda(x)]
\]
as $\varphi^\circ_\lambda(t^{1/2}\star x)$ has exponential decay when $t\to\infty$.

The rapid decay follows from the observation that along the boundary of $\fD$, $\nu_\lambda^\circ(t^{1/2}x)$ is bounded by a function with rapid decay depending only on $x$, multiplied by $\te^{-t}$, for $t\geq 1$.

The lemma is proved.
\end{proof}

Now we consider star products of $g_\lambda$. Put
\[
V_\infty^\Box\coloneqq\left\{\left.(x_1,x_2)\in V_\infty^2\right| T(x_{1,v_0},x_{2,v_0})\in\Sym_2^\circ(\dR)\right\}.
\]
For $(x_1,x_2)\in V_\infty^\Box$ (so that both $x_{1,v_0}$ and $x_{2,v_0}$ are nonzero), $(g_\lambda(x_1),g_\lambda(x_2))$ belongs to $\Div_\sharp^\Box(\fD,\cW_\lambda)$. Lemma \ref{le:green_orthogonal} indicates that
\[
g_\lambda(x_1)\ast g_\lambda(x_2)
\]
is an integrable $(1,1)$-current on $\fD$, so that its integration over $\fD$ is well-defined. So, how to compute
\[
\int_\fD g_\lambda(x_1)\ast g_\lambda(x_2)?
\]

Equip $\dR^3$ with the standard Euclidean quadratic form. Denote by $\phi_0\in\sS(\dR^3)$ the standard Gaussian function on $\dR^3$, that is, $\phi_0(y)=\te^{-\pi(y_1^2+y_2^2+y_3^2)}$. Take an element $v\in\Hom(F,\dR)$. Define two functions
\[
\phi_{b_v}\in\sS(\dR^3)\otimes\tT^{b_v}(\dR^3),\quad
\phi_{\lambda_v}\in\sS(\dR^3)\otimes\tS_{[\lambda_v]}(\dR^3)
\]
by the formulae
\[
\phi_{b_v}(y)=\phi_0(y)\cdot y^{\otimes b_v},\quad
\phi_{\lambda_v}(y)=[\phi_{b_v}(y)]=\phi_0(y)[y^{\otimes b_v}],
\]
respectively.\footnote{The function $\phi_{\lambda_v}$ is an eigenfunction for $\widetilde{K_{1,v}}$ with respect to the Weil representation $\omega_{\psi_{F,v}}$ of eigen-character $\delta_v^{3/2+b_v}$; but $\phi_{b_v}$ is not an eigenfunction as long as $b_v\geq 2$ since it is not ``harmonic''.} Define an $\SO(\dR^3)$-invariant function $\Phi_{\lambda_v}\in\sS((\dR^3)^2)$ by the formula
\[
\Phi_{\lambda_v}(y_1,y_2)=(\phi_{\lambda_v}(y_1),\phi_{\lambda_v}(y_2))_{\tS_{[\lambda_v]}(\dR^3)}.
\]
For every element $T_v\in\Sym_2^\circ(F_v)$, denote by $W_{T_v}(s,-,\Phi_{\lambda_v})$ the $T$-th Whittaker function on $\widetilde{G_2(F_v)}$ associated with $\Phi_{\lambda_v}$ and the parameter $s\in\dC$. More precisely, if we denote by $f_{\Phi_{\lambda_v}}(s)$ the standard Siegel--Weil section with respect to the Weil representation $\omega_{\psi_{F,v}}$, then for every $g_v\in G_2(F_v)$, $W_{T_v}(s,g_v,\Phi_{\lambda_v})$ is the analytic continuation of the $T_v$-th Whittaker coefficient of $f_{\Phi_{\lambda_v}}(s)$ (which is absolutely convergent for $\RE s\gg 0$).

Taking products, we have for $g=(g_v)_v\in\widetilde{G_2(F_\infty)}$ and $T=(T_v)_v\in\Sym_2^\circ(F_\infty)$,
\[
W_T(s,g,\Phi_\lambda)\coloneqq\prod_{v\in\Hom(F,\dR)}W_{T_v}(s,g_v,\Phi_{\lambda_v}).
\]
Write
\[
W'_T(0,g,\Phi_\lambda)\coloneqq\left.\frac{\rd}{\rd s}\right|_{s=0}W_{T(x_1,x_2)}(s,g,\Phi_\lambda).
\]
The theorem below, which can be regarded as an arithmetic analogue of the local Siegel--Weil formula at $v_0$, will be proved in the next section.

\begin{theorem}\label{th:sw_local_orthogonal}
For $(x_1,x_2)\in V_\infty^\Box$, we have
\[
\int_\fD g_\lambda(x_1)\ast g_\lambda(x_2)
=4 C_\infty^{-1}\cdot W'_{T(x_1,x_2)}(0,1_4,\Phi_\lambda),
\]
where $C_\infty\coloneqq(8\sqrt{2}\pi^2\ti)^{[F:\dQ]}$.
\end{theorem}

Now we study the height pairing over $X_L$. From now on, we assume $X_L$ proper.

\begin{definition}\label{de:generating}
Take a Schwartz function $\phi^\infty\in\sS(V^\infty)^L$ satisfying $\phi^\infty(0)=0$ and an element $t\in F$.
\begin{enumerate}
  \item Similar to Definition \ref{de:form_orthogonal} (with $r=1$), we define a function $\rF^\lambda_t(\phi^\infty)$ on $\dH_1$ valued in $\sA^{1,1}(X_L,\tS_{[\lambda]}(V_\infty)_\dC)$ by the formula
      \begin{align*}
      \rF^\lambda_t(\phi^\infty)(\tau)\coloneqq\sum_{(x,h)\in H(F)\backslash V_t\times H(\bA_F^\infty)/L}\phi^\infty(h^{-1}x)\cdot \varphi_\lambda(x y_\tau^{1/2},h)_L\cdot\prod_{v\in\Hom(F,\dR)}{(y_\tau)_v}^{-b_v/2}\te^{2\pi\ti(tx_\tau)_v},
      \end{align*}
      which is in fact a finite sum.

  \item For every $x\in V_\infty$ with $x_{v_0}\neq 0$ and $h\in H(\bA_F^\infty)$, we evaluate $g_\lambda$ at $x$ and then take its pushforward along the map \eqref{eq:translate_orthogonal} (which is well-defined by Lemma \ref{le:green_orthogonal}) to obtain an element  $g_\lambda(x,h)_L\in\sG(X_L,\tS_{[\lambda]}(V_\infty))$ (Notation \ref{no:green_real}). It is clear that $\varphi_\lambda(x,h)$ is invariant under the diagonal action of $H(F)$ on $V_\infty\times H(\bA_F^\infty)$. We then define a function $\rG^\lambda_t(\phi^\infty)$ on $\dH_1$ valued in $\sG(X_L,\tS_{[\lambda]}(V_\infty))$ by the formula
      \begin{align*}
      \rG^\lambda_t(\phi^\infty)(\tau)\coloneqq\sum_{(x,h)\in H(F)\backslash V_t\times H(\bA_F^\infty)/L}\phi^\infty(h^{-1}x)\cdot g_\lambda(x y_\tau^{1/2},h)_L\cdot\prod_{v\in\Hom(F,\dR)}{(y_\tau)_v}^{-b_v/2}\te^{2\pi\ti(tx_\tau)_v},
      \end{align*}
      which is in fact a finite sum.

  \item For every $x\in V_\infty$ with $x_{v_0}\neq 0$ and $h\in H(\bA_F^\infty)$, we take the pushforward of $\sfv^\circ_\lambda(x)$ along the map \eqref{eq:translate_orthogonal} to obtain an element
      \[
      \sfv^\circ_\lambda(x,h)_L\in\Div(X_L,\cW_\lambda).
      \]
      It is clear that $\sfv^\circ_\lambda(x,h)_L$ is invariant under the diagonal action of $H(F)$ on $V_\infty\times H(\bA_F^\infty)$. We define
      \[
      \rD^\lambda_t(\phi^\infty)\coloneqq\sum_{(x,h)\in H(F)\backslash V_t\times H(\bA_F^\infty)/L}\phi^\infty(h^{-1}x)\cdot
      \sfv^\circ_\lambda(x,h)_L\in\Div(X_L,\cW_\lambda),
      \]
      which is again a finite sum. When $v_0(t)\leq 0$, $\rD^\lambda_t(\phi^\infty)=0$ by definition. When $v_0(t)>0$, it is straightforward to verify that the image of $\rD^\lambda_t(\phi^\infty)$ under the cycle class map $\Div(X_L,\cW_\lambda)\to\rH^{1,1}(X_L,\tS_{[\lambda]}(V_\infty)_\dC)$ coincides with $\rC^\lambda_t(\phi^\infty)$ (Definition \ref{de:cycle_orthogonal}).
\end{enumerate}
\end{definition}

The proposition below reveals the relation among three functions $\rF^\lambda_t(\phi^\infty)$, $\rG^\lambda_t(\phi^\infty)$, and $\rD^\lambda_t(\phi^\infty)$ we just defined.

\begin{proposition}\label{pr:green_orthogonal}
For $\phi^\infty\in\sS(V^\infty)^L$ satisfying $\phi^\infty(0)=0$ and $t\in F$, the equality
\[
\rD\rD^\tc\rG^\lambda_t(\phi^\infty)+\delta_{\rD^\lambda_t(\phi^\infty)}\cdot \tq^t=[\rF^\lambda_t(\phi^\infty)]
\]
holds as functions on $\dH_1$ valued in $\sD(X_L,\tS_{[\lambda]}(V_\infty))$. In other words, for every $\tau\in\dH_1$, $\rG^\lambda_t(\phi^\infty)(\tau)$ is a Green function for $\tq^t(\tau)\cdot\rD^\lambda_t(\phi^\infty)$ with the tail form $\rF^\lambda_t(\phi^\infty)(\tau)$.
\end{proposition}

\begin{proof}
By Lemma \ref{le:green_orthogonal} and the construction, for every $\tau\in\dH_1$, $\rG^\lambda_t(\phi^\infty)(\tau)$ is a Green function for
\[
\sum_{(x,h)\in H(F)\backslash V_t\times H(\bA_F^\infty)/L}\phi^\infty(h^{-1}x)\cdot\mu(xy_\tau^{1/2})
\sfv_\lambda(xy_\tau^{1/2},h)_L\cdot\prod_{v\in\Hom(F,\dR)}{(y_\tau)_v}^{-b_v/2}\te^{2\pi\ti(tx_\tau)_v}
\]
with the tail form $\rF^\lambda_t(\phi^\infty)(\tau)$. The proposition follows as
\begin{align*}
\sfv_\lambda(xy_\tau^{1/2},h)_L&=\sfv_\lambda(x,h)_L\cdot\prod_{v\in\Hom(F,\dR)}{(y_\tau)_v}^{b_v/2}, \\
\tq^t(\tau)&=\mu(xy_\tau^{1/2})\cdot\prod_{v\in\Hom(F,\dR)}\te^{2\pi\ti(tx_\tau)_v}.
\end{align*}
\end{proof}

\begin{definition}
Define $\rd^\natural h$ to be the unique (positive) Haar measure on $H(\bA_F^\infty)$ (whose existence is ensured by the local Siegel--Weil formula) such that for every $\Phi^\infty\in\sS((V^\infty)^2)$ and every $T\in\Sym_2^\circ(\bA_F^\infty)$,
\[
W_T(0,1_4,\Phi^\infty)=C_\infty^{-1}\int_{H(\bA_F^\infty)}\Phi^\infty(h^{-1}x)\rd^\natural h,
\]
where $x$ is an arbitrary element in $(V^\infty)^2_T$. Here, $C_\infty$ is the constant in Theorem \ref{th:sw_local_orthogonal}.
\end{definition}

\begin{definition}
We say that an element $\Phi^\infty\in\sS((V^\infty)^2)$ is \emph{regularly supported} if there exists a finite place $v$ of $F$ such that $\Phi^\infty(x)=0$ as long as $\det T(x_v)=0$.
\end{definition}

\begin{notation}
For every pair $(t_1,t_2)\in F^2$, put
\[
\fT(t_1,t_2)_{v_0}\coloneqq\left\{T\in\Sym_2(F)\left|T=\(\begin{smallmatrix} t_1 & * \\ * & t_2 \end{smallmatrix}\),v_0(\det T)<0\right.\right\}.
\]
\end{notation}

\begin{theorem}\label{th:sw_global_orthogonal}
For every pair $(t_1,t_2)\in F^2$ and every pair $(\phi^\infty_1,\phi^\infty_2)\in(\sS(V^\infty)^L)^2$ such that $\phi^\infty_1\otimes\ol{\phi^\infty_2}$ is regularly supported, the identity
\[
\vol(L,\rd^\natural h)\int_{X_L}\rG^\lambda_{t_1}(\phi^\infty_1)(\tau_1)\ast\rG^\lambda_{t_2}(\phi^\infty_2)(\tau_2)=
4 \sum_{T\in\fT(t_1,t_2)_{v_0}}W'_T(0,g_{\tau_1,\tau_2},\Phi_\lambda\otimes(\phi^\infty_1\otimes\ol{\phi^\infty_2}))
\]
holds as functions on $\dH_1\times\dH_1$, where
\[
g_{\tau_1,\tau_2}\coloneqq
\begin{pmatrix} 1 & & x_{\tau_1} & \\ & 1 & & x_{\tau_2} \\ && 1 & \\ &&& 1  \end{pmatrix}\cdot
\begin{pmatrix} y_{\tau_1}^{1/2} & & & \\ & y_{\tau_2}^{1/2} & & \\ && y_{\tau_1}^{-1/2} & \\ &&& y_{\tau_2}^{-1/2}  \end{pmatrix}\in G_2(F_\infty).
\]
\end{theorem}

Note that when $\phi^\infty_1\otimes\ol{\phi^\infty_2}$ is regularly supported, $\phi^\infty_1(0)=\phi^\infty_2(0)=0$, and $\rD^\lambda_{t_1}(\phi^\infty_1)(\tau_1)$ and $\rD^\lambda_{t_2}(\phi^\infty_2)(\tau_2)$ are well-defined and have disjoint supports.

\begin{proof}
This is a standard consequence of Proposition \ref{pr:green_orthogonal} and Theorem \ref{th:sw_local_orthogonal}. See, for example, \cite{Kud97}*{\S 12} or \cite{GS19}*{\S 5.3} for the detail.
\end{proof}

To end this section, we derive a result on \emph{canonical} height pairings. Suppose that $\lambda$ is not the trivial weight, so that $\rH^2(X_L,\tS_{[\lambda]}(V_\infty)(1))=0$ hence $\rD^\lambda_t(\phi^\infty)\in\Div(X_L,\cW_\lambda)^0$. Our goal is to obtain an averaging formula for
\[
\left\langle\rD^\lambda_{t_1}(\phi^\infty_1)(\tau_1),\rD^\lambda_{t_2}(\phi^\infty_2)(\tau_2)\right\rangle_{X_L}
\]
(Definition \ref{de:canonical_real}) under the same context of Theorem \ref{th:sw_global_orthogonal}.

\begin{theorem}\label{th:canonical_orthogonal}
Let $K$ be an open compact subgroup of $G_1(\bA_F^\infty)$ and choose a fundamental domain $\dH_1^K$ of $\dH_1$ for the congruent subgroup $G_1(F)\cap K$. Consider an element $t_2\in F$ and a pair $(\phi^\infty_1,\phi^\infty_2)\in(\sS(V^\infty)^L)^2$ such that $\phi^\infty_1\otimes\ol{\phi^\infty_2}$ is regularly supported. For every (holomorphic) Hilbert cusp form $f$ on $\dH_1$ of weights $(b_v+3/2)_v$ and level $G_1(F)\cap K$,
\begin{align*}
&\vol(L,\rd^\natural h)\int_{\dH_1^K}\ol{f(\tau_1)}\sum_{t_1\in F}
\left\langle\rD^\lambda_{t_1}(\phi^\infty_1)(\tau_1),\rD^\lambda_{t_2}(\phi^\infty_2)(\tau_2)\right\rangle_{X_L}\rd\tau_1 \\
&\qquad\quad=2\int_{\dH_1^K}\ol{f(\tau_1)}\sum_{t_1\in F}\sum_{T\in\fT(t_1,t_2)_{v_0}}
W'_T(0,g_{\tau_1,\tau_2},\Phi_\lambda\otimes(\phi^\infty_1\otimes\ol{\phi^\infty_2}))
\rd\tau_1
\end{align*}
holds in which both sides are absolutely convergent, where $\rd\tau_1$ denotes the measure on $\dH_1$ that induces the Petersson inner product.
\end{theorem}

\begin{proof}
The deduction of this theorem from Theorem \ref{th:sw_global_orthogonal} follows from the same argument in the proof of \cite{LL21}*{Proposition~10.1}, as long as we show that
\[
\int_{\dH_1^K}\ol{f(\tau_1)}\sum_{t_1\in F}\rF^\lambda_{t_1}(\phi^\infty_1)(\tau_1)\rd\tau_1=0.
\]
Indeed, if this integral does not vanish, then the local theta lift of $\sigma_{b_{v_0}+3/2}$ -- the holomorphic discrete series of $\widetilde{G_2(F_{v_0})}$ of weight $b_{v_0}+3/2$ -- to $H(F_{v_0})$ does not vanish. However, as we know that the local theta lift of $\sigma_{b_{v_0}+3/2}$ to $\SO(\dR^3)$ is $\tS_{[\lambda_{v_0}]}(\dR^3)$, this cannot happen by the local theta dichotomy.

The theorem is proved.
\end{proof}

\section{Proof of Theorem \ref{th:sw_local_orthogonal}}
\label{ss:7}

We continue the discussion after the statement of Theorem \ref{th:sw_local_orthogonal}. By construction, we have
\[
g_\lambda(x_1)\ast g_\lambda(x_2)=\prod_{v\neq v_0}\Phi_{\lambda_v}(x_{1,v},x_{2,v})\cdot g_{\lambda_{v_0}}(x_{1,v_0})\ast g_{\lambda_{v_0}}(x_{2,v_0}).
\]
By the local Siegel--Weil formula, we have
\[
\frac{W_{T(x_1,x_2)_v}(0,1_4,\Phi_{\lambda_v})}{W_{T(x_1,x_2)_v}(0,1_4,\Phi_{\b{0}})}=
\frac{\Phi_{\lambda_v}(x_{1,v},x_{2,v})}{\Phi_{\b{0}}(x_{1,v},x_{2,v})}.
\]
If $\Phi_{\lambda_v}(x_{1,v},x_{2,v})=0$ for some $v\neq v_0$, then both sides of Theorem \ref{th:sw_local_orthogonal} vanish. If not, then
\[
\Phi_{\lambda_v}(x_{1,v},x_{2,v})=(8\sqrt{2}\pi^2\ti)^{-1}W_{T(x_1,x_2)_v}(0,1_4,\Phi_{\lambda_v})
\]
by \cite{Kud97}*{Proposition~9.5(1)}, hence Theorem \ref{th:sw_local_orthogonal} is reduced to the case of a single place. Therefore, to prove the theorem, we may again assume $F=\dQ$ and suppress $v_0$ in all subscripts (while simply write $V$ for $V_\infty$). Write $\lambda=(b,0,-b)$ for some nonnegative integer $b$.

For $(x_1,x_2)\in V^\Box$, put
\begin{align*}
\nu_\lambda^\circ(x_1,x_2)\coloneqq\tr_{W_\lambda}
\(\nu_\lambda^\circ(x_1)\wedge\varphi_\lambda^\circ(x_2)+\varphi_\lambda^\circ(x_1)\wedge\nu_\lambda^\circ(x_2)\),
\end{align*}
and define
\[
g_\lambda^\circ(x_1,x_2)\coloneqq\int_1^\infty\frac{\nu_\lambda^\circ(t^{1/2}x_1,t^{1/2}x_2)}{t^b}\frac{\rd t}{t},
\]
which is absolutely convergent over $\fD$ and defines an element in $\sA^{1,1}(\fD)$. Moreover, it is easy to see that $g_\lambda^\circ(x_1,x_2)$ has rapid decay along the boundary of $\fD$.

Following \cite{GS19}*{\S 2.7}, $g_\lambda^\circ(x_1,x_2)$ is closely related to $g_\lambda^\circ(x_1)\ast g_\lambda^\circ(x_2)$. Indeed, for $t_1,t_2\in\dR_{>0}$, define
\begin{align*}
\alpha_\lambda(t_1,t_2,x_1,x_2)&\coloneqq\frac{\ti}{2\pi}\cdot
\frac{\nu_\lambda^\circ\(t_1^{1/2}x_1\)}{t_1^{b/2}}\wedge\rD''\frac{\nu_\lambda^\circ\(t_2^{1/2}x_2\)}{t_2^{b/2}}\wedge\frac{\rd t_1\rd t_2}{t_1t_2},\\
\beta_\lambda(t_1,t_2,x_1,x_2)&\coloneqq\frac{\ti}{2\pi}\cdot
\rD'\frac{\nu_\lambda^\circ\(t_1^{1/2}x_1\)}{t_1^{b/2}}\wedge\frac{\nu_\lambda^\circ\(t_2^{1/2}x_2\)}{t_2^{b/2}}\wedge\frac{\rd t_1\rd t_2}{t_1t_2},
\end{align*}
and set
\begin{align*}
\alpha_\lambda(x_1,x_2)&\coloneqq\int_{1\leq t_1\leq t_2< \infty}\alpha_\lambda(t_1,t_2,x_1,x_2),\\
\beta_\lambda(x_1,x_2)&\coloneqq\int_{1\leq t_1\leq t_2< \infty}\beta_\lambda(t_1,t_2,x_1,x_2).
\end{align*}
By Remark \ref{re:gauss_manin} and the discussion for the case $\lambda=0$ \cite{GS19}*{Lemma~2.15}, we know that both $\alpha_\lambda(x_1,x_2)$ and $\beta_\lambda(x_1,x_2)$ are smooth forms on $\fD\setminus(\fD_{x_1}\cup\fD_{x_2})$ and are locally integrable on $\fD$. Denote by $[\alpha_\lambda(x_1,x_2)]$ and $[\beta_\lambda(x_1,x_2)]$ the induced currents.

\begin{lem}\label{le:sw_orthogonal_1}
The identity
\[
g_\lambda^\circ(x_1)\ast g_\lambda^\circ(x_2)-g_\lambda^\circ(x_1,x_2)=\rD'[\alpha_\lambda(x_1,x_2)]+\rD''[\beta_\lambda(x_1,x_2)]
\]
holds as currents. Moreover, we have
\[
\int_\fD g_\lambda^\circ(x_1)\ast g_\lambda^\circ(x_2)-g_\lambda^\circ(x_1,x_2)=0.
\]
\end{lem}

\begin{proof}
The first part follows in the same way by the argument for \cite{GS19}*{Theorem~2.16} after we use Lemma \ref{le:green_orthogonal} (for the subscript $\lambda$) instead of \cite{GS19}*{Proposition~2.6(b)}. The second part follows from the first, in view of the easy fact that both $\alpha_\lambda(x_1,x_2)$ and $\beta_\lambda(x_1,x_2)$ have rapid decay along the boundary of $\fD$.
\end{proof}

\begin{lem}\label{le:sw_orthogonal_2}
For $(x_1,x_2)\in V^\Box$, we have
\[
\int_\fD g_\lambda(x_1)\ast g_\lambda(x_2)=\mu(x_1)\mu(x_2)\lim_{M\to +\infty}\int_1^M\int_\fD\frac{\nu_\lambda^\circ(t^{1/2}x_1,t^{1/2}x_2)}{t^b}\frac{\rd t}{t}.
\]
\end{lem}

\begin{proof}
Lemma \ref{le:sw_orthogonal_1} implies that
\[
\int_\fD g_\lambda(x_1)\ast g_\lambda(x_2)=\mu(x_1)\mu(x_2)\int_\fD g_\lambda^\circ(x_1,x_2).
\]
By construction,
\[
g_\lambda^\circ(x_1,x_2)=\lim_{M\to +\infty}\int_1^M\frac{\nu_\lambda^\circ(t^{1/2}x_1,t^{1/2}x_2)}{t^b}\frac{\rd t}{t},
\]
in which the convergence is uniform on $\fD$. It follow that
\[
\int_\fD g_\lambda^\circ(x_1,x_2)=\lim_{M\to +\infty}\int_\fD\int_1^M\frac{\nu_\lambda^\circ(t^{1/2}x_1,t^{1/2}x_2)}{t^b}\frac{\rd t}{t}.
\]
The lemma follows since the double integral is absolute convergent.
\end{proof}

Define two functions $\Psi_{\lambda1}$ and $\Psi_{\lambda2}$ in $\sS(V^2)$ such that
\begin{align*}
\Psi_{\lambda1}(x_1,x_2)\otimes\Omega=\tr_{W_\lambda}\(\nu_\lambda(x_1)\wedge\varphi_\lambda(x_2)\), \qquad
\Psi_{\lambda2}(x_1,x_2)\otimes\Omega=\tr_{W_\lambda}\(\varphi_\lambda(x_1)\wedge\nu_\lambda(x_2)\)
\end{align*}
hold in $\sS(V^2)\otimes_\dR\wedge^2\fp^\vee$. Finally, put $\Psi_\lambda\coloneqq\Psi_{\lambda1}+\Psi_{\lambda2}$.

\begin{lem}\label{le:sw_orthogonal_3}
There exists a constant $C\in\dC^\times$ independent of $\lambda$, such that for every $(x_1,x_2)\in V^\Box$,
\[
\mu(x_1)\mu(x_2)\int_\fD\nu_\lambda^\circ(x_1,x_2)=C\cdot W_{T(x_1,x_2)}(0,1_4,\Psi_\lambda)
\]
holds.
\end{lem}

\begin{proof}
By construction, the left-hand side equals
\[
\int_\fD\Psi_\lambda(x_1,x_2)\otimes\Omega.
\]
By the local Siegel--Weil formula, there exists a unique Haar measure $\rd h$ on $H(\dR)$ such that for every $\Psi\in\sS(V^2)$ and every $(x_1,x_2)\in V^2$ with $T(x_1,x_2)\in\Sym_2^\circ(\dR)$,
\[
\int_{H(\dR)}\Psi(h^{-1}(x_1,x_2))\rd h=W_{T(x_1,x_2)}(0,1_4,\Psi).
\]
Note that $\Psi_\lambda$ is invariant under the action of $L$ (recall that $\fD\simeq H(\dR)/L$). Thus, the lemma holds if we take $C$ such that
\[
\Omega=C\cdot\rd h/\rd l,
\]
where $\rd l$ is the Haar measure on $L$ with total volume $1$.
\end{proof}

To continue, we need to further analyze degenerate principle series. To ease notation, we simply write $G_r$ for $G_r(\dR)=\Sp_{2r}(\dR)$, so that we have the standard maximal compact subgroup $K_r\subseteq G_r(\dR)$, whose metaplectic double cover is denoted by $\widetilde{K_r}\subseteq\widetilde{G_r}=\Mp_{2r}(\dR)$. For $\alpha\in\dZ/4\dZ$ and $s\in\dC$, we have the principal series representation $\rI^\alpha_r(s)$ of $\widetilde{G_r}$ as in \cite{GS19}*{(3.8)} (with $G'=\widetilde{G_r}$ and $P$ its standard upper-triangular Siegel parabolic subgroup). We have maps
\[
f_\bullet\colon\sS((\dR^3)^2)\to\rI^3_2(0),\qquad
f_\bullet\colon\sS(V^2)\to\rI^3_2(0)
\]
by taking Siegel--Weil sections with respect to $\omega_\psi$ (where $\psi$ sends $t\in\dR$ to $\te^{2\pi t\ti}$). For $s\in\dC$, we denote by $f_\bullet(s)\in\rI^3_2(s)$ the induced standard section of $f_\bullet$.

Let $\iota_1$ and $\iota_2$ be the unique homomorphisms from $\widetilde{G_1}$ to $\widetilde{G_2}$ that lift the homomorphisms
\begin{align*}
\begin{pmatrix} a & b \\ c & d \end{pmatrix}
\mapsto
\begin{pmatrix}  a && b & \\  & 1 && 0 \\ c && d  & \\ & 0 && 1\end{pmatrix},\qquad
\begin{pmatrix} a & b \\ c & d \end{pmatrix}
\mapsto
\begin{pmatrix} 1 & & 0 & \\ & a && b \\ 0 && 1 & \\ & c && d \end{pmatrix}
\end{align*}
from $G_1$ to $G_2$, respectively. Recall that we have the lowering operator $\rL$, and write $\rL_1\coloneqq(\iota_1)_*\rL$ and $\rL_2\coloneqq(\iota_2)_*\rL$.

\begin{proposition}\label{pr:whittaker_orthogonal}
There exist elements $\Phi_{\lambda1}$ and $\Phi_{\lambda2}$ in $\sS((\dR^3)^2)$ such that
\begin{align*}
\rL_1 f_{\Phi_\lambda}(s)&=-\frac{s}{2}\cdot f_{\Psi_{\lambda1}}(s)+s\cdot f_{\Phi_{\lambda1}}(s), \\
\rL_2 f_{\Phi_\lambda}(s)&=-\frac{s}{2}\cdot f_{\Psi_{\lambda2}}(s)+s\cdot f_{\Phi_{\lambda2}}(s),
\end{align*}
hold for $s\in\dC$.
\end{proposition}

When $b=0$ (so that $\lambda=\b{0}$), one may take $\Phi_{\b{0}1}=\Phi_{\b{0}2}=0$ by \cite{GS19}*{Lemma~3.1~\&~(3.89)}.

\begin{proof}
We prove for $\rL_2$ and the case for $\rL_1$ is similar.

We first review some facts from \cite{LZ13}. Denote by $\rI(s)$ the $\widetilde{K_2}$-finite subspace of $\rI^3_2(s)$. Then
\[
\rI(s)=\bigoplus_{\mu=(c_1,c_2)\in\fW_2}\rI_\mu(s),
\]
where $\rI_\mu(s)$ denotes the subspace whose $\widetilde{K_2}$-type has weight $2\mu+(3/2,3/2)$,\footnote{Please be aware of the multiple $2$ here.} which has multiplicity one. Put
\[
\rI_+(s)=\bigoplus_{\substack{\mu=(c_1,c_2)\in\fW_2\\c_2\geq 0}}\rI_\mu(s),\qquad
\rI_-(s)=\bigoplus_{\substack{\mu=(c_1,c_2)\in\fW_2\\c_2< 0}}\rI_\mu(s).
\]
Denote by $\sS'((\dR^3)^2)$ (resp.\ $\sS'(V^2)$) the subspace of $\sS((\dR^3)^2)$ (resp.\ $\sS(V^2)$) consisting of Schwartz functions of the form $P\cdot\Phi_{\b{0}}=P\cdot(\phi_0\otimes\phi_0)$ (resp.\ $P\cdot(\nu_0\otimes\nu_0)$) where $P$ is a polynomial, whose image under the map $f_\bullet$ is contained in $\rI(0)$. More precisely, such images of $\sS'((\dR^3)^2)$ and $\sS'(V^2)$ are $\rI_+(0)$ and $\rI_-(0)$, respectively.

We use the way in \cite{Lee96}*{\S 2} to describe elements in $\rI(s)$. In particular,
\begin{align*}
f_{\Phi_{\b{0}}}(s)&\doteq(-\det\ol{z})^{-3/2}\cdot|\det z|^{-s},\\
f_{\Psi_{\b{0}2}}(s)&\doteq\gamma(z)(\det z)^{-1}(-\det\ol{z})^{-3/2}\cdot|\det z|^{-s},
\end{align*}
where $\gamma(z)=z_{11}\ol{z}_{22}-z_{12}\ol{z}_{21}$. Here, $\doteq$ means that the left-hand side is given by the restriction of the right-hand side to $X^o$. Moreover, $\rL_2$ corresponds to the operator
\[
\frac{1}{2}\varepsilon_{22}^*=\ol{z}_{21}\frac{\partial}{\partial z_{21}}+\ol{z}_{22}\frac{\partial}{\partial z_{22}}
\]
\cite{Lee96}*{(2.5)}. It follows easily that $\rL_2f_{\Phi_{\b{0}}}(s)=-\frac{s}{2}\cdot f_{\Psi_{\b{0}2}}(s)$.

Now we consider a general weight $\lambda=(b,0,-b)$. Denote by $\rR$ the operator given by the element
\[
\frac{-\ti}{8\pi}w''_1\circ w''_2\in\fp_2^+
\]
\if false
\[
-\frac{1}{2\pi}\begin{pmatrix}  & 1 & & \ti \\ 1 & & \ti & \\ & \ti & & -1 \\ \ti & & -1 &  \end{pmatrix}\in\fg_2
\]
\fi
in the notation of \cite{FM06}*{Lemma~A.1}. By the same lemma and the construction of $\Psi_{\lambda2}$, we have
\[
f_{\Psi_{\lambda2}}=q_\lambda\cdot\rR^b f_{\Psi_{\b{0}2}}\in\rI_{(b,-1)}(0)
\]
(see Example \ref{ex:ratio_orthogonal} for $q_\lambda$). On the other hand, by \cite{FM06}*{Lemma~A.1} again, $\rR^b f_{\Phi_{\b{0}}}=f_{\Phi_b}$, where $\Phi_b\coloneqq\phi_b\otimes\phi_b$. We need to study the relation between $f_{\Phi_b}$ and $f_{\Phi_\lambda}$.

Write $\(\begin{smallmatrix} y_1 \\ y_2 \end{smallmatrix}\)=\(\begin{smallmatrix} y_{11} & y_{12} & y_{13} \\ y_{21} & y_{22} & y_{23} \end{smallmatrix}\)$ for the standard coordinates of $(\dR^3)^2$. We use the Fock model $\sS'((\dR^3)^2)\simeq\dC\left[\begin{smallmatrix}\sfy_{11} & \sfy_{12} & \sfy_{13} \\ \sfy_{21} & \sfy_{22} & \sfy_{23} \end{smallmatrix}\right]$, in which $\sfy_{ij}$ corresponds to $2\pi\ti\(y_{ij}-\frac{1}{2\pi}\frac{\partial}{\partial y_{ij}}\)$. Let $P_\lambda$ and $P_b$ be the polynomials in the Fock model corresponding to $\Phi_\lambda$ and $\Phi_b$, respectively. Then we have $\Phi_\lambda=(4\pi\ti)^{-2b}P_\lambda\cdot\Phi_0$ and $P_b=(4\pi\ti)^{-2b}(\sfy_{11}\sfy_{21}+\sfy_{12}\sfy_{22}+\sfy_{13}\sfy_{23})^b$. Define a symmetric matrix
\[
\sft=\begin{pmatrix} \sft_{11} & \sft_{12} \\ \sft_{21} & \sft_{22} \end{pmatrix}
\coloneqq\begin{pmatrix} \sfy_{11}^2+\sfy_{12}^2+\sfy_{13}^2 & \sfy_{11}\sfy_{21}+\sfy_{12}\sfy_{22}+\sfy_{13}\sfy_{23} \\ \sfy_{11}\sfy_{21}+\sfy_{12}\sfy_{22}+\sfy_{13}\sfy_{23} & \sfy_{21}^2+\sfy_{22}^2+\sfy_{23}^2 \end{pmatrix}
\]
in which we assign bi-degree $(2,0)$ to $\sft_{11}$, $(0,2)$ to $\sft_{22}$, and $(1,1)$ to $\sft_{12}=\sft_{21}$. Denote by $\dC[\sft]_{b,b}$ the space of polynomials in entries of $\sft$ of homogeneous bi-degree $(b,b)$ that are symmetric in $\sft_{12}$ and $\sft_{21}$, which is a subspace of the Fock model. It is clear both $P_\lambda$ and $P_b$ belong to $\dC[\sft]_{b,b}$, and there exist unique elements $c\in\dC$ and $P^\flat_b\in\dC[\sft]_{b-2,b-2}$ such that
\[
P_b=c\cdot P_\lambda + (\det\sft)P^\flat_b.
\]
Evaluating both sides at $\sft=\(\begin{smallmatrix} 1 & 1 \\ 1 & 1 \end{smallmatrix}\)$, we obtain that $c=q_\lambda^{-1}$. It follows that
\[
q_\lambda\cdot f_{\Phi_b}=f_{\Phi_\lambda}+f_{(\det\sft)P^\flat_b}
\]
(where we have replaced $P^\flat_b$ by $q_\lambda^{-1}P^\flat_b$). Note that $f_{\Phi_\lambda}\in\rI_{(b,0)}(0)$, and since $P_\lambda$ is harmonic in the $y_2$-variable, $\rL_2f_{\Phi_\lambda}=0$. It follows that
\[
\rL_2f_{\Phi_\lambda}(s)=s f_-(s)+sf_+(s)
\]
for unique elements $f_-\in\rI_{(b,-1)}(0)$ and $f_+\in\rI_+(0)$. For the proposition, it remains to show that $f_-(s)=-\frac{1}{2}f_{\Psi_{b2}}(s)$.

By \cite{FM06}*{Lemma~A.1} again, we know that $\det\sft$ is an eigenvector for the action of $\widetilde{K_2}$, of eigen-character $\delta^2$ (determinant-square). Thus,
\[
f_{(\det\sft)P^\flat_b}\in\bigoplus_{\substack{\mu=(c_1,c_2)\in\fW_2\\c_1+c_2=b\\ c_2>0}}\rI_\mu(0),
\]
so that $\rL_2f_{(\det\sft)P^\flat_b}(s)\in\rI_+(s)$. It follows that $s f_-(s)$ is also the projection of $q_\lambda\cdot\rL_2(\rR^b f_{\Phi_{\b{0}}})(s)$ to the factor $\rI_{(b,-1)}(s)$. Together, it reduces to show that the projection of $\rL_2(\rR^b f_{\Phi_{\b{0}}})(s)$ to the factor $\rI_{(b,-1)}(s)$ coincides with $-\frac{s}{2}\cdot(\rR^b f_{\Psi_{\b{0}2}})(s)$.

Now we have
\[
\rL_2(\rR^b f_{\Phi_{\b{0}}})(s)-(\rL_2\rR^b f_{\Phi_{\b{0}}})(s)\doteq
-\frac{s}{2}\cdot\rR^b\((-\det\ol{z})^{-3/2}\)\cdot\gamma(z)(\det z)^{-1}|\det z|^{-s},
\]
and
\[
-\frac{s}{2}\cdot(\rR^b f_{\Psi_{\b{0}2}})(s)\doteq
-\frac{s}{2}\cdot\rR^b\((-\det\ol{z})^{-3/2}\cdot\gamma(z)(\det z)^{-1}\)\cdot|\det z|^{-s}.
\]
Since $\rR$ annihilates both $\gamma(z)$ and $\det z$ by \cite{Lee96}*{(2.5)}, we have
\[
\rL_2(\rR^b f_{\Phi_{\b{0}}})(s)-(\rL_2\rR^b f_{\Phi_{\b{0}}})(s)=-\frac{s}{2}\cdot(\rR^b f_{\Psi_{\b{0}2}})(s).
\]
The the claim follows as $\rL_2\rR^b f_{\Phi_{\b{0}}}\in\rI_+(0)$.

The proposition is proved.
\end{proof}

\begin{remark}\label{re:harmonic_orthogonal}
Indeed, if we denote by $\rR_1$ and $\rR_2$ the operators given by the elements
\[
\frac{-\ti}{8\pi}w''_1\circ w''_1,\qquad \frac{-\ti}{8\pi}w''_2\circ w''_2
\]
in $\fp_2^+$ (in the notation of \cite{FM06}*{Lemma~A.1}) respectively, then
\[
\Phi_\lambda=Q_\lambda(\rR_1,\rR,\rR,\rR_2)\Phi_{\b{0}}
\]
(see Example \ref{ex:ratio_orthogonal} for $Q_\lambda$).
\end{remark}

For $\Phi\in\sS((\dR^3)^2)\oplus\sS(V^2)$, $T\in\Sym_2^\circ(\dR)$, and $y\in\Sym_2^\circ(\dR)^+$, we define a function
\[
\cW_{b,T}(s,y,\Phi)\coloneqq W_T(s,m(a),\Phi)\cdot(\det y)^{-\frac{3+2b}{4}}\te^{2\pi\tr(Ty)},
\]
which has an analytic continuation to the entire complex plane. Here, $a$ is any element in $\GL_2(\dR)$ such that $y=a\cdot\pres{t}a$ and $\det a\in\dR_{>0}$.

\begin{corollary}\label{co:whittaker_orthogonal}
For every $T\in\Sym_2^\circ(\dR)$ with $\det T<0$, the identity
\[
\cW_{b,T}(0,ty,\Psi_\lambda)=-2t\frac{\rd}{\rd t}\cW'_T(0,ty,\Phi_\lambda)
\]
holds for every $t\in\dR_{>0}$.
\end{corollary}

\begin{proof}
The proof is almost the same as the one for \cite{GS19}*{Lemma~3.5}. Without loss of generality, we may assume $y=\(\begin{smallmatrix}y_1 & \\  & y_2 \end{smallmatrix}\)$ a diagonal matrix with $y_1,y_2\in\dR_{>0}$. Set $y^{1/2}\coloneqq\(\begin{smallmatrix}y_1^{1/2} & \\  & y_2^{1/2} \end{smallmatrix}\)$. Proposition \ref{pr:whittaker_orthogonal} implies that for $i=1,2$,
\[
-2^{-1}s\cW_{b,T}(s,m(y^{1/2}),\Psi_{\lambda i})+ s\cW_{b,T}(s,m(y^{1/2}),\Phi_{\lambda i})=
\rL_i\cW_T(s,m(y^{1/2}),\Phi_\lambda).
\]
On the other hand, it is easy to see that
\begin{align*}
\rL_i\cW_{b,T}(s,m(y^{1/2}),\Phi_\lambda)&=y_i\frac{\rd}{\rd y_i}\cW_{b,T}(s,m(y^{1/2}),\Phi_\lambda).
\end{align*}
Together, we have
\[
\cW_{b,T}(s,m(y^{1/2}),\Psi_\lambda)-2^{-1}\cW_{b,T}(s,m(y^{1/2}),\Phi_{\lambda 1}+\Phi_{\lambda 2})=-2s^{-1}
\sum_{i=1}^2y_i\frac{\rd}{\rd y_i}\cW_{b,T}(s,m(y^{1/2}),\Phi_\lambda).
\]
As both $\Phi_\lambda$ and $\Phi_{\lambda1}+\Phi_{\lambda2}$ belong to $\sS((\dR^3)^2)$, we have
\[
\cW_{b,T}(0,m(y^{1/2}),\Phi_{\lambda 1}+\Phi_{\lambda 2})=\cW_{b,T}(0,m(y^{1/2}),\Phi_\lambda)=0.
\]
Letting $s\to 0$, we obtain
\[
\cW_{b,T}(0,m(y^{1/2}),\Psi_\lambda)=-2\sum_{i=1}^2y_i\frac{\rd}{\rd y_i}\cW'_{b,T}(0,m(y^{1/2}),\Phi_\lambda).
\]
The corollary follows.
\end{proof}

Now we can finish the proof of Theorem \ref{th:sw_local_orthogonal}. Write $T$ for $T(x_1,x_2)$ for short, so that $\det T<0$. By Lemma \ref{le:sw_orthogonal_2} and Lemma \ref{le:sw_orthogonal_3}, we have
\begin{align*}
\int_\fD g_\lambda(x_1)\ast g_\lambda(x_2)
&=C\te^{-2\pi\tr T}\lim_{M\to +\infty}\int_1^M t^{-b}\te^{2\pi\tr tT} \cdot W_{tT}(0,1_4,\Psi_\lambda)\frac{\rd t}{t} \\
&=C\te^{-2\pi\tr T}\lim_{M\to +\infty}\int_1^M t^{-\frac{3+2b}{2}}\te^{2\pi\tr tT} \cdot W_T(0,m(t^{1/2}1_2),\Psi_\lambda)\frac{\rd t}{t} \\
&=C\te^{-2\pi\tr T}\lim_{M\to +\infty}\int_1^M \cW_{b,T}(0,t1_2,\Psi_\lambda)\frac{\rd t}{t},
\end{align*}
which, by Corollary \ref{co:whittaker_orthogonal}, equals
\[
2C\te^{-2\pi\tr T}\lim_{M\to +\infty}\(\cW'_{b,T}(0,1_2,\Phi_\lambda)-\cW'_{b,T}(0,M1_2,\Phi_\lambda)\).
\]
We claim that $\lim_{M\to +\infty}\cW'_{b,T}(0,M1_2,\Phi_\lambda)=0$. This is true for $\lambda=\b{0}$ by \cite{GS19}*{Proposition~3.3}. Indeed, the same argument of that proposition implies the same asymptotic behaviour for $Q.\Phi_{\b{0}}$ for every polynomial $Q$ in $\fp_2^+$, in particular, for $\Phi_\lambda$ by Remark \ref{re:harmonic_orthogonal}.

Thus, one obtain
\[
\int_\fD g_\lambda(x_1)\ast g_\lambda(x_2)=2C\te^{-2\pi\tr T}\cW'_{b,T}(0,1_2,\Phi_\lambda)=2C\cdot W'_T(0,1_4,\Phi_\lambda).
\]
To figure out the value of $2C$, we take $\lambda=\b{0}$. Then by \cite{Kud97}*{Theorem~11.8}, we obtain
\[
2C=(2\sqrt{2}\pi^2\ti)^{-1}.
\]
Theorem \ref{th:sw_local_orthogonal} is finally proved.

\section{Archimedean arithmetic Siegel--Weil formula over unitary Shimura curves}
\label{ss:8}

We continue the discussion in \S\ref{ss:3} but with $m=2$ and $r=1$. In particular, $\lambda_u=(b_u,-b_{\ol{u}})$ for every $u\in\Hom(E,\dC)$.

\begin{definition}\label{de:hodge_unitary}
We define complex Hodge structures on the local systems $\tT_b(V_\infty)$ and $\tS_{[\lambda]}(V_\infty)$ over $\fD$, respectively.
\begin{itemize}
  \item For $u\in\Hom(E,\dC)\setminus\{u_0,\ol{u_0}\}$, define $\cV_u$ to be the unique complex Hodge structure on $\sV_u\coloneqq V_u\otimes_\dC\sO_\fD$ satisfying that $\cV_u^{(0,0)}=\sV_u^\infty$.

  \item For $u=u_0$, denote by $V_\fD^-$ the tautological complex subbundle of negative lines of $\sV_u^\infty=V_u\otimes_\dC\sA_\fD$ (where $\sV_u\coloneqq V_u\otimes_\dC\sO_\fD$) over (the underlying smooth manifold of) $\fD$, which is a holomorphic subbundle of $\sV_u^\infty$, and by $V_\fD^+$ its orthogonal complement. Define $\cV_u$ to be the unique complex Hodge structure on $\sV_u$ satisfying that $\cV_u^{(0,0)}=V_\fD^+$ and $\cV_u^{(1,-1)}=V_\fD^-$.\footnote{Indeed, $\cV_u$ is the unique complex Hodge structure of weight zero on $\sV_u$ of which the hermitian form on $V_u$ induces a polarization (Definition \ref{de:polarization}), satisfying $\cV_u^{(p,-p)}=0$ for $|p|\geq 2$.}

  \item For $u=\ol{u_0}$, define $\cV_u$ to be the complex conjugation of $\cV_{u_0}$, which is a complex Hodge structure on $\sV_u=\ol{\sV_{u_0}}$. Note that $\cV_u$ is also canonically isomorphic to the dual of $\cV_{u_0}$ induced by the hermitian form on $V_u$.
\end{itemize}
By taking tensor product of $\cV_u$ and further applying \eqref{eq:harmonic_unitary} and \eqref{eq:schur_unitary}, we obtain complex Hodge structures $\cV_b$ and $\cV_\lambda$ on $\tT_b(V_\infty)$ and $\tS_{[\lambda]}(V_\infty)$ over $\fD$, respectively. These complex Hodge structures are all $H(F_\infty)$-invariant and descend to the unitary Shimura curve
\[
X_L=H(F)\backslash\(\fD\times H(\bA_F^\infty)/L\).
\]
In particular, the underlying complex local system of $\cV_\lambda$ is $\tS_{[\lambda]}(V_\infty)$.
\end{definition}

Recall that we have chosen an orthonormal basis $\{e_{v,1},e_{v,2}\}$ of $V_v$ for every $v\in\Hom(F,\dR)$, together with all follow-up notation in \S\ref{ss:3}. Note that
\[
\Omega\coloneqq\frac{\ti}{2\pi}\cdot\xi_{12}\wedge\ol{\xi_{12}}\in\wedge^{1,1}\fp^\vee
\]
is independent of the choice of the basis (of $V_{v_0}$).

For the Kudla--Millson form $\varphi_0\coloneqq\varphi_{1,0}$, we have
\[
\varphi_0(x)=(2\pi|z^{u_0}_1|^2-1)\nu_0(x)\cdot\Omega
\]
(in particular, $\varphi_0(0)=-\Omega$). By Lemma \ref{le:invariance}, we have the elements
\begin{align*}
\varphi_b\coloneqq\varphi_{1,b}=\(\prod_{u\in\Hom(E,\dC)}\rE_u^{b_u}\)\varphi_0
&\in\(\sS(V_\infty)\otimes_\dC\wedge^{1,1}\fp^\vee\otimes_\dC\tT_b(V_\infty)\)^{L_\infty} \\
&=\(\sS(V_\infty)\otimes_\dC\sA^{1,1}(\fD)\otimes_\dC\tT_b(V_\infty)\)^{H(F_\infty)}
\end{align*}
and then
\[
[\varphi_b]\in\(\sS(V_\infty)\otimes_\dC\sA^{1,1}(\fD)\otimes_\dC\tT_{[b]}(V_\infty)\)^{H(F_\infty)}
\]
by applying \eqref{eq:harmonic_unitary}. Since $m=2$, $[\varphi_b]$ already belongs to $\sS(V_\infty)\otimes_\dC\sA^{1,1}(\fD)\otimes_\dR\tS_{[\lambda]}(V_\infty)$, so that
\[
\varphi_\lambda\coloneqq\pi_{\lambda^+}([\varphi_b])=[\varphi_b].
\]
This special feature when $m=2$ will appear repeatedly later.

On the other hand, we introduce
\begin{align*}
\nu_b\coloneqq\(\prod_{u\in\Hom(E,\dC)}\rE_u^{b_u}\)\nu_0
\in\(\sS(V_\infty)\otimes_\dC\tT_b(V_\infty)\)^{L_\infty}
=\(\sS(V_\infty)\otimes_\dC\sA^0(\fD)\otimes_\dC\tT_b(V_\infty)\)^{H(F_\infty)},
\end{align*}
and similarly
\[
\nu_\lambda\coloneqq[\nu_b]\in\(\sS(V_\infty)\otimes_\dC\sA^0(\fD)\otimes_\dC\tS_{[\lambda]}(V_\infty)\)^{H(F_\infty)}.
\]

The lemma below is an analogue of \cite{FH21}*{Theorem~3.3} in the case of general weights (but only for $(p,q)=(1,1)$).

\begin{lem}\label{le:holomorphy_unitary}
We have
\[
\omega_{\psi_{F,v_0}}(\rL_{v_0})\varphi_b=-\rD\rD^\tc\nu_b+
\frac{1}{4\pi}\(\sum_{i=1}^{b_{u_0}}\sum_{j=1}^{b_{\ol{u_0}}}[A_{ij}^{u_0}\varphi_{b-1_{u_0}-1_{\ol{u_0}}}]
+\sum_{i=1}^{b_{\ol{u_0}}}\sum_{j=1}^{b_{u_0}}[A_{ij}^{\ol{u_0}}\varphi_{b-1_{u_0}-1_{\ol{u_0}}}]\).
\]
In particular,
\[
\omega_{\psi_{F,v_0}}(\rL_{v_0})\varphi_\lambda=-\rD\rD^\tc\nu_\lambda.
\]
\end{lem}

\begin{proof}
For the identity in the lemma, both sides have the same factor for places away from $v_0$. Thus, without loss of generality, we may assume $F=\dQ$ and suppress $v_0$ in all subscripts (while simply write $V$ for $V_\infty$). Write $b'\coloneqq b_{u_0}$ and $b''\coloneqq b_{\ol{u_0}}$. For $i=1,2$, write $e'_i\coloneqq e_i^{u_0},e''_i\coloneqq e_i^{\ol{u_0}}$ and $z_i\coloneqq z_i^{u_0}$.

For the space $\(\sS(V)\otimes\wedge^\bullet\fp^\vee\otimes_\dC\tT(V_\infty)\)^{L_\infty}$, the four components of the Gauss--Manin operator are given by the following formulae:
\begin{align*}
\partial'&=\frac{1}{2}\cdot\omega_{\psi_F}(X_{12}-\ti X'_{12})\otimes L(\xi_{12})\otimes 1,\\
\ol\partial'&=\frac{1}{2}\cdot\omega_{\psi_F}(X_{12}+\ti X'_{12})\otimes L(\ol{\xi_{12}})\otimes 1,\\
\nabla'&=\frac{1}{2}\cdot 1\otimes L(\xi_{12})\otimes\rho(X_{12}-\ti X'_{12}),\\
\ol\nabla'&=\frac{1}{2}\cdot 1\otimes L(\ol{\xi_{12}})\otimes\rho(X_{12}+\ti X'_{12}),
\end{align*}
where $L(\xi)$ denotes the left multiplication by an element $\xi\in\fp^\vee$, and $\rho$ denotes the right induced action of $\fp$ on $\tT(V)$.

By definition, we have
\[
-\rD\rD^\tc\nu_b=(2\pi\ti)^{-1}\(\partial'\ol\partial'\nu_b+\ol\nabla'\nabla'\nu_b\).
\]
For further computation, we use the Fock model $\dC[\sfz'_1,\sfz'_2;\sfz''_1,\sfz''_2]$ of the Weil representation $\omega_{\psi_F}$ as in \cite{FH21}*{Appendix~A~\&~B} so that the following correspond:
\begin{align*}
z_1-\frac{1}{\pi}\frac{\partial}{\partial\ol{z}_1} \leftrightarrow \frac{-\ti}{\sqrt{2}\pi}\sfz'_1,&\qquad
z_2-\frac{1}{\pi}\frac{\partial}{\partial\ol{z}_2} \leftrightarrow \frac{\ti}{\sqrt{2}\pi}\sfz'_2;\\
\ol{z}_1-\frac{1}{\pi}\frac{\partial}{\partial z_1} \leftrightarrow \frac{-\ti}{\sqrt{2}\pi}\sfz''_1,&\qquad
\ol{z}_2-\frac{1}{\pi}\frac{\partial}{\partial z_2} \leftrightarrow \frac{\ti}{\sqrt{2}\pi}\sfz''_2.
\end{align*}
Under such model,
\begin{align*}
\nu_b&=\(\frac{-\ti}{2\sqrt{2}\pi}\)^{b'+b''}(\sfz'_1)^{b'}(\sfz''_1)^{b''}\otimes1\otimes\((e'_1)^{\otimes b'}\otimes(e''_1)^{\otimes b''}\), \\
\varphi_b&=\frac{-1}{4\pi}\(\frac{-\ti}{2\sqrt{2}\pi}\)^{b'+b''}(\sfz'_1)^{b'+1}(\sfz''_1)^{b''+1}
\otimes\Omega\otimes\((e'_1)^{\otimes b'}\otimes(e''_1)^{\otimes b''}\),\\
\end{align*}
and by \cite{FH21}*{\S A.2},
\begin{align*}
\partial'&=\(-4\pi\frac{\partial^2}{\partial\sfz'_1\partial\sfz''_2}+\frac{1}{4\pi}\sfz''_1\sfz'_2\)\otimes L(\xi_{12})\otimes 1, \\
\ol\partial'&=\(-4\pi\frac{\partial^2}{\partial\sfz''_1\partial\sfz'_2}+\frac{1}{4\pi}\sfz'_1\sfz''_2\)\otimes L(\ol{\xi_{12}})\otimes 1.
\end{align*}
It follows that
\begin{align}\label{eq:holomorphy3}
\partial'\ol\partial'\nu_b=(2\pi\ti)\(\frac{-\ti}{2\sqrt{2}\pi}\)^{b'+b''}\((b'+1)(\sfz'_1)^{b'}(\sfz''_1)^{b''}
-\frac{1}{16\pi^2}(\sfz'_1)^{b'+1}(\sfz''_1)^{b''+1}\sfz'_2\sfz''_2\)\otimes\Omega\otimes\((e'_1)^{\otimes b'}\otimes(e''_1)^{\otimes b''}\).
\end{align}
For $\ol\nabla'\nabla'\nu_b$, it equals
\[
(2\pi\ti)\(\frac{-\ti}{2\sqrt{2}\pi}\)^{b'+b''}
(\sfz'_1)^{b'}(\sfz''_1)^{b''}\otimes\Omega\otimes\rho\(\frac{X_{12}+\ti X'_{12}}{2}\)\rho\(\frac{X_{12}-\ti X'_{12}}{2}\)\((e'_1)^{\otimes b'}\otimes(e''_1)^{\otimes b''}\).
\]
Note that the actions of $\rho\(\frac{X_{12}-\ti X'_{12}}{2}\)$ and $\rho\(\frac{X_{12}+\ti X'_{12}}{2}\)$ on $V\otimes_\dR\dC$ are given by the following assignments
\begin{align*}
e'_1\mapsto 0,\quad e'_2\mapsto e'_1,\quad e''_1\mapsto e''_2,\quad e''_2\mapsto 0,\\
e'_1\mapsto e'_2,\quad e'_2\mapsto 0,\quad e''_1\mapsto 0,\quad e''_2\mapsto e''_1,
\end{align*}
respectively, which implies that
\begin{align*}
\quad &\rho\(\frac{X_{12}+\ti X'_{12}}{2}\)\rho\(\frac{X_{12}-\ti X'_{12}}{2}\)\((e'_1)^{\otimes b'}\otimes(e''_1)^{\otimes b''}\) \\
&=(b'+1)b''\((e'_1)^{\otimes b'}\otimes(e''_1)^{\otimes b''}\)-\sum_{i=1}^{b'}\sum_{j=1}^{b''} A_{ij}\((e'_1)^{\otimes b'-1}\otimes(e''_1)^{\otimes b''-1}\).
\end{align*}
Combining with \eqref{eq:holomorphy3}, we obtain
\begin{align*}
-\rD\rD^\tc\nu_b=\widetilde\varphi_b-\frac{1}{2\pi}\sum_{i=1}^{b'}\sum_{j=1}^{b''} A_{ij}\varphi_{b-1},
\end{align*}
where
\[
\widetilde\varphi_b\coloneqq
\(\frac{-\ti}{2\sqrt{2}\pi}\)^{b'+b''}\((b'+1)(b''+1)(\sfz'_1)^{b'}(\sfz''_1)^{b''}
-\frac{1}{16\pi^2}(\sfz'_1)^{b'+1}(\sfz''_1)^{b''+1}\sfz'_2\sfz''_2\)\otimes\Omega\otimes\((e'_1)^{\otimes b'}\otimes(e''_1)^{\otimes b''}\).
\]

On the other hand, by \cite{FH21}*{\S~A.2}, we have
\[
\omega_{\psi_F}(\rL)=-4\pi\frac{\partial^2}{\partial\sfz'_1\partial\sfz''_1}+\frac{1}{4\pi}\sfz'_1\sfz''_1,
\]
which implies that $\omega_{\psi_F}(\rL)\varphi_b=\widetilde\varphi_b$.

The lemma follows.

\end{proof}

Define a function $\mu\colon V_\infty\to\dR_{>0}$ by the formula
\[
\mu(x)\coloneqq\tq^{T(x)}(\ti)=\prod_{v\in\Hom(F,\dR)}\te^{-2\pi T(x_v)},
\]
and put
\[
\varphi_\lambda^\circ\coloneqq\varphi_\lambda\cdot\mu^{-1},\qquad
\nu_\lambda^\circ\coloneqq\nu_\lambda\cdot\mu^{-1}.
\]
Define an action $\star$ of $\dR_{>0}$ on $V_\infty$ via the formula $t\star x=(tx_{v_0},x_v,\dots)$.

\begin{lem}\label{le:nu_unitary}
For $x\in V_\infty$, we have
\[
\rD\rD^\tc\frac{\nu^\circ_\lambda(t^{1/2}\star x)}{t^{(b_{u_0}+b_{\ol{u_0}})/2}}=-t\frac{\rd}{\rd t}\frac{\varphi^\circ_\lambda(t^{1/2}\star x)}{t^{(b_{u_0}+b_{\ol{u_0}})/2}}
\]
for every $t\in\dR_{>0}$.
\end{lem}

\begin{proof}
This follows from the same argument in the the proof of Lemma \ref{le:nu_orthogonal} (with $w=4+2(b_{u_0}+b_{\ol{u_0}})$).
\end{proof}

Take an element $x\in V_\infty$ with $x_{v_0}\neq 0$. Denote by $\fD_x\subseteq\fD$ the subset of complex lines perpendicular to $x_{v_0}$. If $T(x_{v_0})>0$ (resp.\ $T(x_{v_0})\leq 0$), then $\fD_x$ consists of one point (resp.\ is empty). Put
\[
\sfv^\circ_b(x)\coloneqq\(\fD_x,\otimes_{u\in\Hom(F,\dR)}x_u^{\otimes b_u}\)\in\Div(\fD,\cV_b)
\]
(see Definition \ref{de:hodge_unitary} for the Hodge structure) and consequently
\[
\sfv^\circ_\lambda(x)\coloneqq[\sfv^\circ_b(x)]\in\Div(\fD,\cV_\lambda)
\]
(so that $\sfv^\circ_\lambda(x)=0$ if $T(x_{v_0})\leq 0$). Define
\[
g_\lambda^\circ(x)\coloneqq\int_1^\infty\frac{\nu_\lambda^\circ(t^{1/2}\star x)}{t^{(b_{u_0}+b_{\ol{u_0}})/2}}\frac{\rd t}{t},
\]
which is absolutely convergent over $\fD\setminus\fD_x$ (after we regard $\nu_\lambda^\circ(x)$ as an element in $\sA^0(\fD)\otimes_\dC\tS_{[\lambda]}(V_\infty)$). Finally, put
\[
g_\lambda(x)\coloneqq \mu(x)\cdot g_\lambda^\circ(x),\qquad
\sfv_\lambda(x)\coloneqq \mu(x)\cdot \sfv_\lambda^\circ(x).
\]

The lemma below is an analogue of Lemma \ref{le:green_orthogonal} in the unitary case, with the similar proof which we omit.

\begin{lem}\label{le:green_unitary}
For every $x\in V_\infty$ with $x_{v_0}\neq 0$, $g_\lambda(x)$ is a smooth function on $\fD\setminus\fD_x$ with logarithmic growth along $\fD_x$ and rapid decay along the boundary of $\fD$, and satisfies
\[
\rD\rD^\tc g_\lambda(x) + \delta_{\sfv_\lambda(x)}=[\varphi_\lambda(x)].
\]
In particular, $g_\lambda(x)$ is a Green function for $\sfv_\lambda(x)$ with the tail form $\varphi_\lambda(x)$.
\end{lem}

Now we consider star products of $g_\lambda$. Put
\[
V_\infty^\Box\coloneqq\left\{\left.(x_1,x_2)\in V_\infty^2\right| T(x_{1,v_0},x_{2,v_0})\in\Herm_2^\circ(\dR)\right\}.
\]
For $(x_1,x_2)\in V_\infty^\Box$ (so that both $x_{1,v_0}$ and $x_{2,v_0}$ are nonzero), $(g_\lambda(x_1),g_\lambda(x_2))$ belongs to $\Div_\sharp^\Box(\fD,\cV_\lambda)$. Lemma \ref{le:green_unitary} indicates that
\[
g_\lambda(x_1)\ast g_\lambda(x_2)
\]
is an integrable $(1,1)$-current on $\fD$, so that its integration over $\fD$ is well-defined. So, how to compute
\[
\int_\fD g_\lambda(x_1)\ast g_\lambda(x_2)?
\]

Take an element $v\in\Hom(F,\dR)$. Put $\dC_v\coloneqq F_v\otimes_FE$, which canonically contains $\dR$. Equip $\dC_v^2$ with the standard Euclidean hermitian form. Denote by $\phi_0\in\sS(\dC_v^2)$ the standard Gaussian function on $\dC_v^2$, that is, $\phi_0(y)=\te^{-\pi(|y_1|^2+|y_2|^2)}$. Applying Definition \ref{de:schur_unitary} to $E/F=\dC_v/F_v$, $V=\dC_v^2$, $\dL=\dC$, we obtain a complex representation
\[
\tT_{b_v}(\dC_v^2)\coloneqq\bigotimes_{\substack{u\in\Hom(E,\dC)\\ v(u)=v}}\tT^{b_u}(\dC^2)
\]
of $\rU(\dC_v^2)$ (where $\rU(\dC_v^2)$ acts on $\tT^{b_u}(\dC^2)$ via $u$), and the direct summand $\tS_{[\lambda_v]}(\dC_v^2)$ of $\tT_{b_v}(\dC_v^2)$. Define two functions
\[
\phi_{b_v}\in\sS(\dC_v^2)\otimes\tT_{{b_v}}(\dC_v^2),\quad
\phi_{\lambda_v}\in\sS(\dC_v^3)\otimes\tS_{[\lambda_v]}(\dC_v^2)
\]
by the formulae
\[
\phi_{b_v}(y)=\phi_0(y)\cdot\otimes_{v(u)=v}u(y)^{\otimes b_u},\quad
\phi_{\lambda_v}(y)=[\phi_{b_v}(y)]=\phi_0(y)[\otimes_{v(u)=v}u(y)^{\otimes b_u}],
\]
respectively. Define a $\rU(\dC_v^2)$-invariant function $\Phi_{\lambda_v}\in\sS((\dC_v^2)^2)$ by the formula
\[
\Phi_{\lambda_v}(y_1,y_2)=(\phi_{\lambda_v}(y_1),\phi_{\lambda_v}(y_2))_{\tS_{[\lambda_v]}(\dC_v^2)}.
\]
For every element $T_v\in\Herm_2^\circ(F_v)$, denote by $W_{T_v}(s,-,\Phi_{\lambda_v})$ the $T$-th Whittaker function on $G_2(F_v)$\footnote{Since $m=2$ is even, the Whittaker function, a priori on $\widetilde{G_2(F_v)}$, descends to a function on $G_2(F_v)$.} associated with $\Phi_{\lambda_v}$ and the parameter $s\in\dC$.

Taking products, we have for $g=(g_v)_v\in G_2(F_\infty)$ and $T=(T_v)_v\in\Herm_2^\circ(F_\infty)$,
\[
W_T(s,g,\Phi_\lambda)\coloneqq\prod_{v\in\Hom(F,\dR)}W_{T_v}(s,g_v,\Phi_{\lambda_v}).
\]
Write
\[
W'_T(0,g,\Phi_\lambda)\coloneqq\left.\frac{\rd}{\rd s}\right|_{s=0}W_{T(x_1,x_2)}(s,g,\Phi_\lambda).
\]
The theorem below, which can be regarded as an arithmetic analogue of the local Siegel--Weil formula at $v_0$, will be proved in the next section.

\begin{theorem}\label{th:sw_local_unitary}
For $(x_1,x_2)\in V_\infty^\Box$, we have
\[
\int_\fD g_\lambda(x_1)\ast g_\lambda(x_2)
=C_\infty^{-1}\cdot W'_{T(x_1,x_2)}(0,1_4,\Phi_\lambda),
\]
where $C_\infty\coloneqq(8\pi^3)^{[F:\dQ]}$.
\end{theorem}

Now we study the height pairing over $X_L$. From now on, we assume $X_L$ proper.

Take a Schwartz function $\phi^\infty\in\sS(V^\infty)^L$ satisfying $\phi^\infty(0)=0$ and an element $t\in F$. By Definition \ref{de:generating} after replacing $b_v$ by $\sum_{v(u)=v}b_u$, we have functions $\rF^\lambda_t(\phi^\infty)$ and $\rG^\lambda_t(\phi^\infty)$ on $\dH_1$ valued in $\sA^{1,1}(X_L,\tS_{[\lambda]}(V_\infty))$ and $\sG(X_L,\tS_{[\lambda]}(V_\infty))$, respectively, and an element $\rD^\lambda_t(\phi^\infty)\in\Div(X_L,\cW_\lambda)$ satisfying that when $v_0(t)\leq 0$, $\rD^\lambda_t(\phi^\infty)=0$; and when $v_0(t)>0$, its image under the cycle class map $\Div(X_L,\cW_\lambda)\to\rH^{1,1}(X_L,\tS_{[\lambda]}(V_\infty))$ coincides with $\rC^\lambda_t(\phi^\infty)$ (Definition \ref{de:cycle_unitary}).

The proposition below is the analogue of Proposition \ref{pr:green_orthogonal} in the unitary case, whose proof is similar which we omit.

\begin{proposition}\label{pr:green_unitary}
For $\phi^\infty\in\sS(V^\infty)^L$ satisfying $\phi^\infty(0)=0$ and $t\in F$, the equality
\[
\rD\rD^\tc\rG^\lambda_t(\phi^\infty)+\delta_{\rD^\lambda_t(\phi^\infty)}\cdot \tq^t=[\rF^\lambda_t(\phi^\infty)]
\]
holds as functions on $\dH_1$ valued in $\sD(X_L,\tS_{[\lambda]}(V_\infty))$. In other words, for every $\tau\in\dH_1$, $\rG^\lambda_t(\phi^\infty)(\tau)$ is a Green function for $\tq^t(\tau)\cdot\rD^\lambda_t(\phi^\infty)$ with the tail form $\rF^\lambda_t(\phi^\infty)(\tau)$.
\end{proposition}

\begin{definition}\label{de:measure_unitary}
Define $\rd^\natural h$ to be the unique (positive) Haar measure on $H(\bA_F^\infty)$ (whose existence is ensured by the local Siegel--Weil formula) such that for every $\Phi^\infty\in\sS((V^\infty)^2)$ and every $T\in\Herm_2^\circ(\bA_F^\infty)$,
\[
W_T(0,1_4,\Phi^\infty)=C_\infty^{-1}\int_{H(\bA_F^\infty)}\Phi^\infty(h^{-1}x)\rd^\natural h,
\]
where $x$ is an arbitrary element in $(V^\infty)^2_T$. Here, $C_\infty$ is the constant in Theorem \ref{th:sw_local_unitary}.
\end{definition}

\begin{definition}\label{de:regular_unitary}
We say that an element $\Phi^\infty\in\sS((V^\infty)^2)$ is \emph{regularly supported} if there exists a finite place $v$ of $F$ such that $\Phi^\infty(x)=0$ as long as $\det T(x_v)=0$.
\end{definition}

\begin{notation}\label{no:moment_unitary}
For every pair $(t_1,t_2)\in F^2$, put
\[
\fT(t_1,t_2)_{v_0}\coloneqq\left\{T\in\Herm_2(F)\left|T=\(\begin{smallmatrix} t_1 & * \\ * & t_2 \end{smallmatrix}\),v_0(\det T)<0\right.\right\}.
\]
\end{notation}

The theorem below is the analogue of Theorem \ref{th:sw_global_orthogonal} in the unitary case.

\begin{theorem}\label{th:sw_global_unitary}
For every pair $(t_1,t_2)\in F^2$ and every pair $(\phi^\infty_1,\phi^\infty_2)\in(\sS(V^\infty)^L)^2$ such that $\phi^\infty_1\otimes\ol{\phi^\infty_2}$ is regularly supported, the identity
\[
\vol(L,\rd^\natural h)\int_{X_L}\rG^\lambda_{t_1}(\phi^\infty_1)(\tau_1)\ast\rG^\lambda_{t_2}(\phi^\infty_2)(\tau_2)=
4 \sum_{T\in\fT(t_1,t_2)_{v_0}}W'_T(0,g_{\tau_1,\tau_2},\Phi_\lambda\otimes(\phi^\infty_1\otimes\ol{\phi^\infty_2}))
\]
holds as functions on $\dH_1\times\dH_1$, where
\[
g_{\tau_1,\tau_2}\coloneqq
\begin{pmatrix} 1 & & x_{\tau_1} & \\ & 1 & & x_{\tau_2} \\ && 1 & \\ &&& 1  \end{pmatrix}\cdot
\begin{pmatrix} y_{\tau_1}^{1/2} & & & \\ & y_{\tau_2}^{1/2} & & \\ && y_{\tau_1}^{-1/2} & \\ &&& y_{\tau_2}^{-1/2}  \end{pmatrix}\in G_2(F_\infty).
\]
\end{theorem}

To end this section, we derive a result on \emph{canonical} height pairings. Suppose that $\lambda$ is not the trivial weight, so that $\rH^2(X_L,\tS_{[\lambda]}(V_\infty)(1))=0$ hence $\rD^\lambda_t(\phi^\infty)\in\Div(X_L,\cV_\lambda)^0$. The theorem below is the analogue of Theorem \ref{th:canonical_orthogonal} in the unitary case, providing an averaging formula for
\[
\left\langle\rD^\lambda_{t_1}(\phi^\infty_1)(\tau_1),\rD^\lambda_{t_2}(\phi^\infty_2)(\tau_2)\right\rangle_{X_L}
\]
(Definition \ref{de:canonical_complex}) under the same context of Theorem \ref{th:sw_global_unitary}.

\begin{theorem}\label{th:canonical_unitary}
Let $K$ be an open compact subgroup of $G_1(\bA_F^\infty)$ and choose a fundamental domain $\dH_1^K$ of $\dH_1$ for the congruent subgroup $G_1(F)\cap K$. Consider an element $t_2\in F$ and a pair $(\phi^\infty_1,\phi^\infty_2)\in(\sS(V^\infty)^L)^2$ such that $\phi^\infty_1\otimes\ol{\phi^\infty_2}$ is regularly supported. For every (holomorphic) hermitian Hilbert cusp form $f$ on $\dH_1$ of weights $(2b_u+1,-2b_{\ol{u}}-1)_u$ and level $G_1(F)\cap K$,
\begin{align*}
&\vol(L,\rd^\natural h)\int_{\dH_1^K}\ol{f(\tau_1)}\sum_{t_1\in F}
\left\langle\rD^\lambda_{t_1}(\phi^\infty_1)(\tau_1),\rD^\lambda_{t_2}(\phi^\infty_2)(\tau_2)\right\rangle_{X_L}\rd\tau_1 \\
&\qquad\quad=\frac{1}{2}\int_{\dH_1^K}\ol{f(\tau_1)}\sum_{t_1\in F}\sum_{T\in\fT(t_1,t_2)_{v_0}}
W'_T(0,g_{\tau_1,\tau_2},\Phi_\lambda\otimes(\phi^\infty_1\otimes\ol{\phi^\infty_2}))
\rd\tau_1
\end{align*}
holds in which both sides are absolutely convergent, where $\rd\tau_1$ denotes the measure on $\dH_1$ that induces the Petersson inner product.
\end{theorem}

\begin{proof}
The deduction of this theorem from Theorem \ref{th:sw_global_unitary} follows from the same argument in the proof of \cite{LL21}*{Proposition~10.1}, as long as we show that
\[
\int_{\dH_1^K}\ol{f(\tau_1)}\sum_{t_1\in F}\rF^\lambda_{t_1}(\phi^\infty_1)(\tau_1)\rd\tau_1=0.
\]
Indeed, if this integral does not vanish, then the local theta lift of $\sigma_{b_{u_0}+1/2,-b_{\ol{u_0}}-1/2}$ -- the holomorphic discrete series of $G_2(F_{v_0})$ of Harish-Chandra weight $(b_{u_0}+1/2,-b_{\ol{u_0}}-1/2)$ -- to $H(F_{v_0})$ does not vanish. However, as we know that the local theta lift of $\sigma_{b_{u_0}+1/2,-b_{\ol{u_0}}-1/2}$ to $\rU(\dC_{v_0}^2)$ is $\tS_{[\lambda_{v_0}]}(\dC_{v_0}^2)$, this cannot happen by the local theta dichotomy.

The theorem is proved.
\end{proof}

\section{Proof of Theorem \ref{th:sw_local_unitary}}
\label{ss:9}

The proof of Theorem \ref{th:sw_local_unitary} follows from the same line of Theorem \ref{th:sw_global_orthogonal}. We will sketch the process by only indicating the difference. First, by the same argument (but using \cite{Liu11}*{Proposition~4.5(2)} instead of \cite{Kud97}*{Proposition~9.5(1)}), we may again assume $F=\dQ$ and suppress $v_0$ in all subscripts (while simply write $V$ for $V_\infty$). We identify $E\otimes_\dQ\dR$ with $\dC$ via the embedding $u_0$, and write $b'\coloneqq b_{u_0}$ and $b''\coloneqq b_{\ol{u_0}}$ so that $\lambda_{u_0}=(b',-b'')$ and $\lambda_{\ol{u_0}}=(b'',-b')$.

For $(x_1,x_2)\in V^\Box$, put
\begin{align*}
\nu_\lambda^\circ(x_1,x_2)\coloneqq\tr_{V_\lambda}
\(\nu_\lambda^\circ(x_1)\wedge\ol{\varphi_\lambda^\circ(x_2)}+\varphi_\lambda^\circ(x_1)\wedge\ol{\nu_\lambda^\circ(x_2)}\),
\end{align*}
and define
\[
g_\lambda^\circ(x_1,x_2)\coloneqq\int_1^\infty\frac{\nu_\lambda^\circ(t^{1/2}x_1,t^{1/2}x_2)}{t^{b'+b''}}\frac{\rd t}{t},
\]
which is absolutely convergent over $\fD$ and defines an element in $\sA^{1,1}(\fD)$. Moreover, it is easy to see that $g_\lambda^\circ(x_1,x_2)$ has rapid decay along the boundary of $\fD$.

\begin{lem}\label{le:sw_unitary_2}
For $(x_1,x_2)\in V^\Box$, we have
\[
\int_\fD g_\lambda(x_1)\ast g_\lambda(x_2)=\mu(x_1)\mu(x_2)\lim_{M\to +\infty}\int_1^M\int_\fD\frac{\nu_\lambda^\circ(t^{1/2}x_1,t^{1/2}x_2)}{t^{b'+b''}}\frac{\rd t}{t}.
\]
\end{lem}

\begin{proof}
This follows from the similar proof of Lemma \ref{le:sw_orthogonal_2}.
\end{proof}

Define two functions $\Psi_{\lambda1}$ and $\Psi_{\lambda2}$ in $\sS(V^2)$ such that
\begin{align*}
\Psi_{\lambda1}(x_1,x_2)\otimes\Omega=\tr_{V_\lambda}\(\nu_\lambda(x_1)\wedge\ol{\varphi_\lambda(x_2)}\), \qquad
\Psi_{\lambda2}(x_1,x_2)\otimes\Omega=\tr_{V_\lambda}\(\varphi_\lambda(x_1)\wedge\ol{\nu_\lambda(x_2)}\)
\end{align*}
hold in $\sS(V^2)\otimes_\dC\wedge^2\fp^\vee$. Finally, put $\Psi_\lambda\coloneqq\Psi_{\lambda1}+\Psi_{\lambda2}$.

\begin{lem}\label{le:sw_unitary_3}
There exists a constant $C\in\dC^\times$ independent of $\lambda$, such that for every $(x_1,x_2)\in V^\Box$,
\[
\mu(x_1)\mu(x_2)\int_\fD\nu_\lambda^\circ(x_1,x_2)=C\cdot W_{T(x_1,x_2)}(0,1_4,\Psi_\lambda)
\]
holds.
\end{lem}

\begin{proof}
This follows from the similar proof of Lemma \ref{le:sw_orthogonal_3}.
\end{proof}

To continue, we need to further analyze degenerate principle series. To ease notation, we simply write $G_r$ for $G_r(\dR)=\rU_{r,r}(\dR)\subseteq\GL_{2r}(\dC)$, so that we have the standard maximal compact subgroup $K_r\subseteq G_r(\dR)$. For $s\in\dC$, we have the principal series representation $\rI_r(s)$ of $G_r$ as in \cite{GS19}*{(3.30)} (with $G'=G_r$, $P$ its standard upper-triangular Siegel parabolic subgroup, and $\chi$ trivial). We have maps
\[
f_\bullet\colon\sS((\dC^2)^2)\to\rI_2(0),\qquad
f_\bullet\colon\sS(V^2)\to\rI_2(0)
\]
by taking Siegel--Weil sections with respect to $\omega_\psi$ (where $\psi$ sends $t\in\dR$ to $\te^{2\pi t\ti}$). For $s\in\dC$, we denote by $f_\bullet(s)\in\rI_2(s)$ the induced standard section of $f_\bullet$.

Let $\iota_1$ and $\iota_2$ be the homomorphisms from $G_1$ to $G_2$ given by
\begin{align*}
\begin{pmatrix} a & b \\ c & d \end{pmatrix}
\mapsto
\begin{pmatrix}  a && b & \\  & 1 && 0 \\ c && d  & \\ & 0 && 1\end{pmatrix},\qquad
\begin{pmatrix} a & b \\ c & d \end{pmatrix}
\mapsto
\begin{pmatrix} 1 & & 0 & \\ & a && b \\ 0 && 1 & \\ & c && d \end{pmatrix},
\end{align*}
respectively. Recall that we have the lowering operator $\rL$, and write $\rL_1\coloneqq(\iota_1)_*\rL$ and $\rL_2\coloneqq(\iota_2)_*\rL$.

\begin{proposition}\label{pr:whittaker_unitary}
There exist elements $\Phi_{\lambda1}$ and $\Phi_{\lambda2}$ in $\sS((\dC^2)^2)$ such that
\begin{align*}
\rL_1 f_{\Phi_\lambda}(s)&=-s\cdot f_{\Psi_{\lambda1}}(s)+s\cdot f_{\Phi_{\lambda1}}(s), \\
\rL_2 f_{\Phi_\lambda}(s)&=-s\cdot f_{\Psi_{\lambda2}}(s)+s\cdot f_{\Phi_{\lambda2}}(s),
\end{align*}
hold for $s\in\dC$.
\end{proposition}

When $b'=b''=0$ (so that $\lambda=\b{0}$), one may take $\Phi_{\b{0}1}=\Phi_{\b{0}2}=0$ by \cite{GS19}*{(3.36)~\&~(3.89)}.

\begin{proof}
We prove for $\rL_2$ and the case for $\rL_1$ is similar.

We first review some facts from \cite{Lee94}. Denote by $\rI(s)$ the $K_2$-finite subspace of $\rI_2(s)$. Then
\[
\rI(s)=\bigoplus_{\mu=(c_1,c_2)\in\fW_2}\rI_{\mu}(s),
\]
where $\rI_\mu(s)$ denotes the subspace whose $K_2$-type has weight $(\mu+(1,1))\boxtimes(-\mu-(1,1))$, which has multiplicity one. Put
\[
\rI_{++}(s)=\bigoplus_{\substack{\mu=(c_1,c_2)\in\fW_2\\ c_2\geq 0}}\rI_\mu(s),\qquad
\rI_{+-}(s)=\bigoplus_{\substack{\mu=(c_1,c_2)\in\fW_2\\ c_1\geq 0>c_2}}\rI_\mu(s),\qquad
\rI_{--}(s)=\bigoplus_{\substack{\mu=(c_1,c_2)\in\fW_2\\ 0>c_1}}\rI_\mu(s).
\]
Denote by $\sS'((\dC^2)^2)$ (resp.\ $\sS'(V^2)$) the subspace of $\sS((\dC^2)^2)$ (resp.\ $\sS(V^2)$) consisting of Schwartz functions of the form $P\cdot\Phi_{\b{0}}=P\cdot(\phi_0\otimes\phi_0)$ (resp.\ $P\cdot(\nu_0\otimes\nu_0)$) where $P$ is a polynomial, whose image under the map $f_\bullet$ is contained in $\rI(0)$. More precisely, such images of $\sS'((\dC^2)^2)$ and $\sS'(V^2)$ are $\rI_{++}(0)$ and $\rI_{+-}(0)$, respectively.

We use the way in \cite{Lee94}*{\S 2} to describe elements in $\rI(s)$.\footnote{However, the side of action of the induced representation in \cite{Lee94} is opposite to ours.} In particular,
\begin{align*}
f_{\Phi_{\b{0}}}(s)&\doteq(\det\ol{w})^{-1}\cdot|\det w|^{-2s},\\
f_{\Psi_{\b{0}2}}(s)&\doteq\gamma(z,w)\cdot |\det w|^{-2-2s},
\end{align*}
where $\gamma(z,w)=w_{11}z_{22}-w_{12}z_{21}$. Here, $\doteq$ means that the left-hand side is given by the restriction of the right-hand side to $X^{oo}$. Moreover, $\rL_2$ corresponds to the operator
\[
p_{22}^+=z_{21}\frac{\partial}{\partial w_{21}}+\ol{w}_{21}\frac{\partial}{\partial\ol{z}_{21}}+
z_{22}\frac{\partial}{\partial w_{22}}+\ol{w}_{22}\frac{\partial}{\partial\ol{z}_{22}}
\]
\cite{Lee94}*{(2.6)}. It follows easily that $\rL_2f_{\Phi_{\b{0}}}(s)=-s\cdot f_{\Psi_{\b{0}2}}(s)$.

Now we consider a general weight $\lambda=((b',-b''),(b'',-b'))$. Denote by $\rR'$ and $\rR''$ the operators given by the elements
\[
\frac{-\ti}{4\pi}(w_2\circ w_3-\ti w_2\circ w_3\ti),\qquad
\frac{-\ti}{4\pi}(w_1\circ w_4-\ti w_1\circ w_4\ti)
\]
in $\fp^+_2$ in the notation of \cite{FH21}*{Lemma~B.2}, and put $\rR^b\coloneqq\rR^{\prime\prime b''}\circ\rR^{\prime b'}$. By the same lemma and the construction of $\Psi_{\lambda2}$, we have
\[
f_{\Psi_{\lambda2}}=q_\lambda\cdot\rR^b f_{\Psi_{\b{0}2}}\in\rI_{(b'+b'',-1)}(0)
\]
where $q_\lambda$ is a nonzero rational number defined similarly as in Example \ref{ex:ratio_orthogonal}. On the other hand, by \cite{FH21}*{Lemma~B.2} again, $\rR^b f_{\Phi_{\b{0}}}=f_{\Phi_b}$, where $\Phi_b\coloneqq\phi_b\otimes\ol{\phi_b}$. We need to study the relation between $f_{\Phi_b}$ and $f_{\Phi_\lambda}$.

Write $\(\begin{smallmatrix} y_1 \\ y_2 \end{smallmatrix}\)=\(\begin{smallmatrix} y_{11} & y_{12} \\ y_{21} & y_{22} \end{smallmatrix}\)$ for the standard coordinates of $(\dC^2)^2$. We use the Fock model $\sS'((\dC^2)^2)\simeq\dC\left[\begin{smallmatrix}\sfy'_{11} & \sfy'_{12} & \sfy''_{11} & \sfy''_{12} \\ \sfy'_{21} & \sfy'_{22} & \sfy''_{21} & \sfy''_{22} \end{smallmatrix}\right]$, in which $\sfy'_{ij}$ and $\sfy''_{ij}$ correspond to $\sqrt{2}\pi\ti\(y_{ij}-\frac{1}{\pi}\frac{\partial}{\partial\ol{y_{ij}}}\)$ and $\sqrt{2}\pi\ti\(\ol{y_{ij}}-\frac{1}{\pi}\frac{\partial}{\partial y_{ij}}\)$, respectively. Let $P_\lambda$ and $P_b$ be the polynomials in the Fock model corresponding to $\Phi_\lambda$ and $\Phi_b$, respectively. Then we have $\Phi_\lambda=(2\sqrt{2}\pi\ti)^{-2(b'+b'')}P_\lambda\cdot\Phi_0$ and $P_b=(2\sqrt{2}\pi\ti)^{-2(b'+b'')}
(\sfy'_{11}\sfy''_{21}+\sfy'_{12}\sfy''_{22})^{b'}(\sfy''_{11}\sfy'_{21}+\sfy''_{12}\sfy'_{22})^{b''}$. Define a matrix
\[
\sft=\begin{pmatrix} \sft_{11} & \sft_{12} \\ \sft_{21} & \sft_{22} \end{pmatrix}
\coloneqq\begin{pmatrix} \sfy'_{11}\sfy''_{11}+\sfy'_{12}\sfy''_{12} & \sfy'_{11}\sfy''_{21}+\sfy'_{12}\sfy''_{22} \\ \sfy''_{11}\sfy'_{21}+\sfy''_{12}\sfy'_{22} & \sfy'_{21}\sfy''_{21}+\sfy'_{22}\sfy''_{22} \end{pmatrix}
\]
in which we assign bi-degree $(2,0)$ to $\sft_{11}$, $(0,2)$ to $\sft_{22}$, and $(1,1)$ to $\sft_{12}$ and $\sft_{21}$. For every integer $d\geq 0$, denote by $\dC[\sft]_{d,d}$ the space of polynomials in entries of $\sft$ of homogeneous bi-degree $(d,d)$ that are conjugate-symmetric in $\sft_{12}$ and $\sft_{21}$, which is a subspace of the Fock model. It is clear both $P_\lambda$ and $P_b$ belong to $\dC[\sft]_{b'+b'',b'+b''}$, and there exist unique elements $c\in\dC$ and $P^\flat_b\in\dC[\sft]_{b'+b''-2,b'+b''-2}$ such that
\[
P_b=c\cdot P_\lambda + (\det\sft)P^\flat_b.
\]
Evaluating both sides at $\sft=\(\begin{smallmatrix} 1 & 1 \\ 1 & 1 \end{smallmatrix}\)$, we obtain that $c=q_\lambda^{-1}$. It follows that
\[
q_\lambda\cdot f_{\Phi_b}=f_{\Phi_\lambda}+f_{(\det\sft)P^\flat_b}
\]
(where we have replaced $P^\flat_b$ by $q_\lambda^{-1}P^\flat_b$). Note that $f_{\Phi_\lambda}\in\rI_{(b'+b'',0)}(0)$, and since $P_\lambda$ is harmonic in the $y_2$-variable, $\rL_2f_{\Phi_\lambda}=0$. It follows that
\[
\rL_2f_{\Phi_\lambda}(s)=s f_{+-}(s)+sf_{++}(s)
\]
for unique elements $f_{+-}\in\rI_{(b'+b'',-1)}(0)$ and $f_{++}\in\rI_{++}(0)$. For the proposition, it remains to show that $f_{+-}(s)=-f_{\Psi_{b2}}(s)$.

By \cite{FH21}*{Lemma~B.2} again, we know that $\det\sft$ is an eigenvector for the action of $K_2$, of eigen-character $\delta$. Thus,
\[
f_{(\det\sft)P^\flat_b}\in\bigoplus_{\substack{\mu=(c_1,c_2)\in\fW_2\\ c_1+c_2=b'+b''\\ c_2>0}}\rI_\mu(0),
\]
so that $\rL_2f_{(\det\sft)P^\flat_b}(s)\in\rI_{++}(s)$. It follows that $s f_{+-}(s)$ is also the projection of $q_\lambda\cdot\rL_2(\rR^b f_{\Phi_{\b{0}}})(s)$ to the factor $\rI_{(b'+b'',-1)}(s)$. Together, it reduces to show that the projection of $\rL_2(\rR^b f_{\Phi_{\b{0}}})(s)$ to the factor $\rI_{(b'+b'',-1)}(s)$ coincides with $-s\cdot(\rR^b f_{\Psi_{\b{0}2}})(s)$.

Now we have
\[
\rL_2(\rR^b f_{\Phi_{\b{0}}})(s)-(\rL_2\rR^b f_{\Phi_{\b{0}}})(s)\doteq
-s\cdot\rR^b\((\det\ol{w})^{-1}\)\cdot\gamma(z,w)(\det w)^{-1}|\det w|^{-2s},
\]
and
\[
-s\cdot(\rR^b f_{\Psi_{\b{0}2}})(s)\doteq
-s\cdot\rR^b\((\det\ol{w})^{-1}\cdot\gamma(z,w)(\det w)^{-1}\)\cdot|\det w|^{-2s}.
\]
Since both $\rR'$ and $\rR''$ annihilate both $\gamma(z,w)$ and $\det w$ by \cite{Lee94}*{(2.6)}, we have
\[
\rL_2(\rR^b f_{\Phi_{\b{0}}})(s)-(\rL_2\rR^b f_{\Phi_{\b{0}}})(s)=-s\cdot(\rR^b f_{\Psi_{\b{0}2}})(s).
\]
The the claim follows as $\rL_2\rR^b f_{\Phi_{\b{0}}}\in\rI_{++}(0)$.

The proposition is proved.
\end{proof}

For $\Phi\in\sS((\dC^2)^2)\oplus\sS(V^2)$, $T\in\Herm_2^\circ(\dR)$, and $y\in\Herm_2^\circ(\dR)^+$, we define a function
\[
\cW_{b,T}(s,y,\Phi)\coloneqq W_T(s,m(a),\Phi)\cdot(\det y)^{-\frac{2+b'+b''}{2}}\te^{2\pi\tr(Ty)},
\]
which has an analytic continuation to the entire complex plane. Here, $a$ is any element in $\GL_2(\dC)$ such that $y=a\cdot\pres{t}{\ol{a}}$ and $\det a\in\dR_{>0}$. The corollary below is an analogue of Corollary \ref{co:whittaker_orthogonal} in the unitary case, with the same proof using Proposition \ref{pr:whittaker_unitary} instead of Proposition \ref{pr:whittaker_orthogonal}.

\begin{corollary}\label{co:whittaker_unitary}
For every $T\in\Herm_2^\circ(\dC)$ with $\det T<0$, the identity
\[
\cW_{b,T}(0,ty,\Psi_\lambda)=-t\frac{\rd}{\rd t}\cW'_T(0,ty,\Phi_\lambda)
\]
holds for every $t\in\dR_{>0}$.
\end{corollary}

Now we can finish the proof of Theorem \ref{th:sw_local_unitary}. Write $T$ for $T(x_1,x_2)$ for short, so that $\det T<0$. By Lemma \ref{le:sw_unitary_2} and Lemma \ref{le:sw_unitary_3}, we have
\begin{align*}
\int_\fD g_\lambda(x_1)\ast g_\lambda(x_2)
&=C\te^{-2\pi\tr T}\lim_{M\to +\infty}\int_1^M t^{-(b'+b'')}\te^{2\pi\tr tT} \cdot W_{tT}(0,1_4,\Psi_\lambda)\frac{\rd t}{t} \\
&=C\te^{-2\pi\tr T}\lim_{M\to +\infty}\int_1^M t^{-(2+b+b'')}\te^{2\pi\tr tT} \cdot W_T(0,m(t^{1/2}1_2),\Psi_\lambda)\frac{\rd t}{t} \\
&=C\te^{-2\pi\tr T}\lim_{M\to +\infty}\int_1^M \cW_{b,T}(0,t1_2,\Psi_\lambda)\frac{\rd t}{t},
\end{align*}
which, by Corollary \ref{co:whittaker_unitary}, equals
\[
C\te^{-2\pi\tr T}\lim_{M\to +\infty}\(\cW'_{b,T}(0,1_2,\Phi_\lambda)-\cW'_{b,T}(0,M1_2,\Phi_\lambda)\).
\]
Similar to the orthogonal case, we have $\lim_{M\to +\infty}\cW'_{b,T}(0,M1_2,\Phi_\lambda)=0$.

Thus, one obtain
\[
\int_\fD g_\lambda(x_1)\ast g_\lambda(x_2)=C\te^{-2\pi\tr T}\cW'_{b,T}(0,1_2,\Phi_\lambda)=C\cdot W'_T(0,1_4,\Phi_\lambda).
\]
To figure out the value of $C$, we take $\lambda=\b{0}$. Then by \cite{Liu11}*{Theorem~4.17}, we obtain
\[
C=(8\pi^3)^{-1}.
\]
Theorem \ref{th:sw_local_unitary} is finally proved.

\end{document}